\documentclass[amsfont]{amsart}

\newcommand{\href}[1]{#1} 

\usepackage{amsfonts,amsthm,mathrsfs,latexsym,verbatim,enumerate,setspace,color}
\usepackage[utf8]{inputenc}

\newtheorem{thm}{Theorem}[section]
\newtheorem{lma}[thm]{Lemma}
\newtheorem{prp}[thm]{Proposition}
\newtheorem{cjr}[thm]{Conjecture}

\newtheorem{crly}[thm]{Corollary}

\newtheorem{rem}[thm]{Remark}
\newtheorem*{claim*}{Claim}

\newcounter{claim_nb}[thm]
\newtheorem{claim}[claim_nb]{Claim}

\newcounter{claim_nbs}[section]
\newtheorem{claims}[claim_nbs]{Claim}

\newcounter{subclaim_nb}[claim_nbs]
\newtheorem{subclaim}[subclaim_nb]{Subclaim}

\newenvironment{cproof}
{\begin{proof}
 [Proof.]
 \vspace{-1.2\parsep}}
{ \end{proof}}

\newenvironment{subproof}
{\begin{proof}
 [Proof of Subclaim.]
 \vspace{-1.2\parsep}}
{ \end{proof}}

\usepackage{ifthen}
\newboolean{PrintVersion}
\setboolean{PrintVersion}{false}

\usepackage{amsmath,amssymb,amstext}
\usepackage[pdftex]{graphicx}

\begin{document}

\title{Packing odd $T$-joins with at most two terminals}

\author
{
Ahmad Abdi 
\and
Bertrand Guenin 
}

\date{\today}

\maketitle

\begin{abstract} 
Take a graph $G$, an edge subset $\Sigma\subseteq E(G)$, and a set of terminals $T\subseteq V(G)$ where $|T|$ is even.
The triple $(G,\Sigma,T)$ is called a {\it signed graft}. A $T$-join is {\it odd} if it contains an odd number of edges from $\Sigma$. 
Let $\nu$ be the maximum number of edge-disjoint odd $T$-joins. A {\it signature} is a set of the form $\Sigma\triangle \delta(U)$ where $U\subseteq V(G)$ and $|U\cap T|$ is even. Let $\tau$ be the minimum cardinality a $T$-cut or a signature can achieve. Then $\nu\leq \tau$ and we say that $(G,\Sigma,T)$ {\it packs} if equality holds here.

We prove that $(G,\Sigma,T)$ packs if the signed graft is Eulerian and it excludes two special non-packing minors. Our result confirms the Cycling Conjecture for the class of clutters of odd $T$-joins with at most two terminals. Corollaries of this result include, the characterizations of weakly and evenly bipartite graphs, packing two-commodity paths, packing $T$-joins with at most four terminals, and a new result on covering edges with cuts.
\end{abstract}




\section{The main result}
A {\em signed graph} is a pair $(G,\Sigma)$ where $G$ is a graph and $\Sigma\subseteq E(G)$.
A subset $S$ of the edges is {\em odd} (resp. {\em even}) in $(G,\Sigma)$ if $|S\cap \Sigma|$ is odd (resp. even). In particular, an edge $e$ is odd if $e\in\Sigma$ and it is even otherwise. 
A {\em graft} is a pair $(G,T)$ where $G$ is a graph, $T\subseteq V(G)$ and $|T|$ is even.
Vertices in $T$ are {\em terminal} vertices.
A {\em $T$-join} is an edge subset that induces a subgraph of $G$ with the odd degree vertices equal to $T$.
A {\em $T$-cut} is a cut $\delta(U)=\{uv\in E:u\in U,v\notin U\}$ where $|U\cap T|$ is odd.
A {\em signed graft} is a triple $(G,\Sigma,T)$ where $(G,\Sigma)$ is a signed graph and $(G,T)$ is a graft.
Thus an {\em odd $T$-join} of $(G,\Sigma,T)$ is a $T$-join of $G$ that contains an odd number of edges of $\Sigma$.
Take an edge subset $C\subseteq E(G)$. Then $C$ is a {\em circuit} if it induces a connected subgraph where every vertex has degree two, and $C$ is a {\em cycle} if it induces a subgraph where every vertex has even degree.
When $T=\emptyset$ an (inclusion-wise) minimal odd $T$-join is an odd circuit.
When $T=\{s,t\}$ a minimal odd $T$-join is either an odd $st$-path, or it is the union of an even $st$-path $P$ and an odd circuit $C$ where $P$ and $C$ share at most one vertex. When $T=\{s,t\}$ we say that a set $B\subseteq E(G)$ is an $st$-cut (resp. an $st$-join) if it is a $T$-cut (resp. a $T$-join).

A {\em signature} of the signed graft $(G,\Sigma,T)$ is a set of the form $\Sigma\triangle \delta(U)$, where $U\subseteq V(G)$ and $|U\cap T|$ is even.\footnote{Given sets $A,B$ the set $A-B=\{a\in A:a\notin B\}$, and the set $A\triangle B=(A\cup B)-(A\cap B)$.} Observe that if $\Gamma$ is a signature, then $(G,\Sigma,T)$ and $(G,\Gamma,T)$ have the same collection of odd $T$-joins. We will need the following basic result: 

\begin{thm}\label{zas}
Let $(G,\Sigma,T)$ be a signed graft, and let $F\subseteq E(G)$. Then the following statements hold: \begin{itemize}
\item (Zaslavsky~\cite{Zaslavsky82}) Assume that $T=\emptyset$. If $F$ contains no odd cycle, then there is a signature disjoint from $F$. If $F$ contains no signature, then there is an odd cycle disjoint from $F$.
\item If $F$ does not contain a $T$-join, then there is a $T$-cut disjoint from $F$. If $F$ does not contain a $T$-cut, then there is a $T$-join disjoint from $F$.
\end{itemize}
\end{thm}

\noindent This theorem is very useful and will be applied many times without reference throughout this paper. The first application is the following:

\begin{prp}\label{coverchar}
Let $(G,\Sigma,T)$ be a signed graft. Let $B$ be a minimal set of edges that intersects every odd $T$-join. Then $B$ is either a $T$-cut or a signature. In particular, $B$ intersects every odd $T$-join with odd parity.
\end{prp}
\begin{proof}
By the minimality of $B$, it suffices to show that $B$ contains a $T$-cut or a signature, as $T$-cuts and signatures intersect every odd $T$-join.
To this end, let us assume that $B$ does not contain a $T$-cut. Then there is a $T$-join $J$ disjoint from $B$. Since $B$ intersects every odd $T$-join, it follows that $J$ is an even $T$-join. It also follows that $B$ intersects every odd cycle $C$, for if not, then $J\triangle C$ would be an odd $T$-join disjoint from $B$, which is not the case. Hence, $B$ contains a signature of $(G,\Sigma,\emptyset)$. That is, there is a cut $\delta(U)$ such that $\Sigma\triangle \delta(U)\subseteq B$. It suffices to show that $|U\cap T|$ is even. Since $B\cap J=\emptyset$, we get that $(\Sigma\triangle \delta(U))\cap J=\emptyset$, so in particular, $|(\Sigma\triangle \delta(U))\cap J|$ is even. Since $|\Sigma\cap J|$ is even, it follows that $\delta(U)\cap J$ is even, implying in turn that $|U\cap T|$ is even, as required.
\end{proof}
\noindent Given a signed graft, a {\em cover} is a set of edges that intersects every odd $T$-join {\em with odd parity}.\footnote{This definition is not standard!}
Then by proposition~\ref{coverchar} every minimal set of edges that intersects every odd $T$-join is a cover.

The maximum number of pairwise (edge) disjoint odd $T$-joins in $(G,\Sigma,T)$ is denoted $\nu(G,\Sigma,T)$.
The cardinality of a minimum cover is denoted $\tau(G,\Sigma,T)$.
Clearly, $\tau(G,\Sigma,T)\geq\nu(G,\Sigma,T)$.
We say that $(G,\Sigma,T)$ {\em packs} if equality holds.
$\widetilde{K_5}$ is the signed graft $(K_5,E(K_5),\emptyset)$ and
$F_7$ is the signed graft $(G,\Sigma,T)$ in figure~\ref{fig:fano}.
\begin{figure}[!ht]
\centering
\includegraphics[scale=0.2]{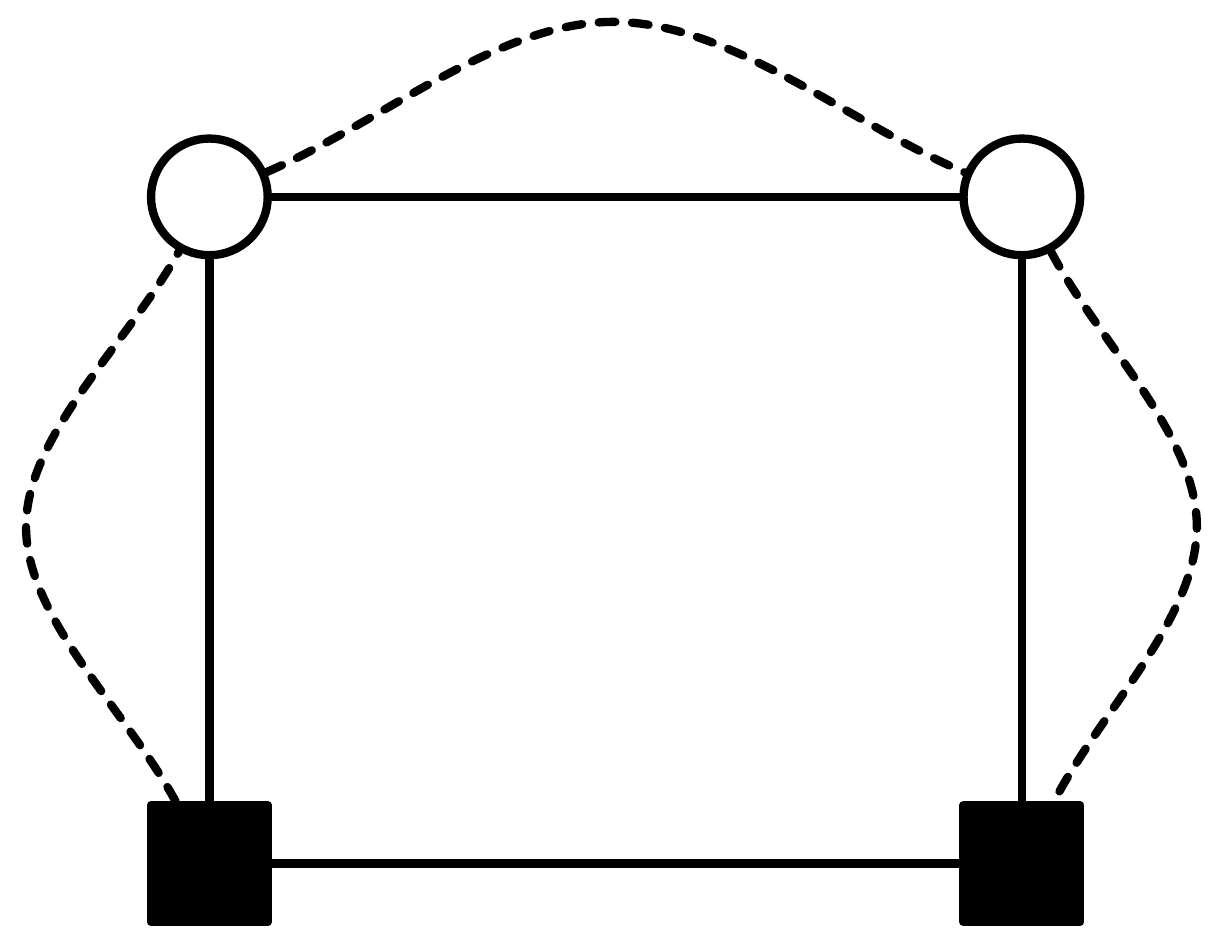}
\caption{Signed graft $F_7$. Dashed edges form the signature, square vertices are terminals.}
\label{fig:fano}
\end{figure}
Note, $4=\tau(\widetilde{K_5})>\nu(\widetilde{K_5})=2$ and $3=\tau(F_7)>\nu(F_7)=1$.
Thus $\widetilde{K_5}$ and $F_7$ do not pack.

Let $(G,\Sigma,T)$ be a signed graft.
$(G,\Gamma,T)$ is obtained by {\em resigning} $(G,\Sigma,T)$ if $\Gamma$ is a signature of $(G,\Sigma,T)$. 
For $e\in E(G)$, we say that $(G\setminus e,\Sigma-\{e\},T)$ is obtained by {\em deleting~$e$}.
For $e=uv\in E(G)-\Sigma$, we say that $(G/e,\Sigma,T')$ is obtained by {\em contracting $e$}
where $T'=T-\{u,v\}$ if both or none of $u,v$ are in $T$ and $T'=T-\{u,v\}\cup\{w\}$ if exactly one of $u,v$ is in $T$ 
where $w$ is the vertex obtained from $e$ by contracting $e$.
A signed graft is a {\em minor} of $(G,\Sigma,T)$ 
if it is obtained by sequentially deleting/contracting edges and resigning.
Note, we can always do all deletions first, resign, and then do all contractions.
We often do not distinguish between signed grafts related by resigning.
In particular we denote by $(G,\Sigma,T)/I\setminus J$ the signed graft obtained from $(G,\Sigma)$ by contracting edge set $I$ and deleting edge set $J$.
Observe that this is only well defined if $I$ does not contain an odd circuit or an odd $T$-join.

We say that a signed graft $(G,\Sigma,T)$ is {\em Eulerian} if every non-terminal vertex has even degree and 
either: every terminal has odd degree and the signature has an odd number of edges; 
or every terminal has even degree and the signature has an even number of edges. So $(G,\Sigma,\emptyset)$ is Eulerian if every vertex has even degree. Notice that resigning preserves the Eulerian property. 

We can now state the main result of the paper,
\begin{thm}\label{main2} 
If an Eulerian signed graft has at most two terminals and it does not contain either of $\widetilde{K_5}$ or $F_7$ as a minor then it packs.
\end{thm}
%
%
\noindent
Observe that the Eulerian condition cannot be omitted. 
For instance $(K_4,E(K_4),\emptyset)$ does not pack and does not contain either of $\widetilde{K_5}$ or $F_7$ as a minor.
Similarly, the signed graft obtained from $F_7$ by deleting the unique edge between the two terminal vertices
does not pack and does not contain either $\widetilde{K_5}$ or $F_7$ as a minor.
\subsection{Special cases}
We say that a graph $H$ is an {\em odd-minor} of a graph $G$ if $H$ is obtained from $G$ by first 
deleting edges and then contracting {\em all} edges on a cut. Theorem~\ref{main2} implies,
\begin{crly}[Geelen and Guenin~\cite{Geelen02}]
Let $G$ be a graph that does not contain $K_5$ as an odd minor and where every vertex has even degree.
Then the minimum number of edges needed to intersect all odd circuits 
is equal to the maximum number of pairwise disjoint odd circuits.
\end{crly}
\begin{proof}


Consider the signed graft $(G,E(G),T)$ where $T=\emptyset$.
Since $T=\emptyset$, $F_7$ is not a minor of $(G,E(G),T)$.
We claim that $\widetilde{K_5}$ is not a minor of $(G,E(G),T)$ either.
Suppose for a contradiction that $\widetilde{K_5}=(G,E(G),\emptyset)/I\setminus J$.
Let $(H,E(H),\emptyset)=(G,E(G),\emptyset)\setminus J$.
We may assume that we resign $(H,E(H),\emptyset)$ to obtain $(H,E(H)-B,\emptyset)$ where $B$ is a cut of $E(H)$, $I\subseteq B$
and that $\widetilde{K_5}=(H,E(H)-B,\emptyset)/I$.
As $\widetilde{K_5}$ has no even edge, $I=B$.
But then $K_5$ is an odd-minor of $G$, a contradiction.
Since all vertices of $G$ have even degree and since $T=\emptyset$, $(G,E(G),\emptyset)$ is Eulerian.
Thus $\tau(G,\Sigma,\emptyset)=\nu(G,\Sigma,\emptyset)$ by theorem~\ref{main2}.
Since $T=\emptyset$ each odd $T$-joins contains an odd circuit and the result follows.
\end{proof}
A {\em blocking vertex} (resp. {\em blocking pair}) in a signed graft is 
a vertex (resp. pair of vertices) that intersects every odd circuit.
\begin{prp}\label{topo}
Consider a signed graft $(G,\Sigma,T)$ where $T=\{s,t\}$.
If any of (1)-(6) hold, then $(G,\Sigma,\{s,t\})$ does not contain $\widetilde{K_5}$ or $F_7$ as a minor:
\begin{enumerate}[\;\;\;(1)]
\item 
there exists a blocking vertex,
\item 
$s,t$ is a blocking pair,
\item 
every minimal odd $st$-join is connected,
\item 
$G$ is a plane graph with at most two odd faces,
\item 
$G$ is a plane graph and $u,v$ is a blocking pair where $s,u,t,v$ appear on a facial cycle in this order,
\item 
$G$ has an embedding on the projective plane where every face is even and $s, t$ are connected by an odd edge.
\end{enumerate}
\end{prp}

\begin{proof}[Proof sketch]
Observe that (3) contains (2) and (6).
Thus it suffices to show the result for (1), (3), (4) and (5).
Suppose that $(G,\Sigma,T)$ with $T=\{s,t\}$ belongs to one of these classes, and let $(G',\Sigma',T')$ be a minor of it. Then, 
\begin{itemize}
\item
if $(G,\Sigma,T)$ belongs to one of (1), (4), then so does $(G',\Sigma',T)$,
\item 
if $(G,\Sigma,T)$ belongs to (3) and $T'=T$, then $(G',\Sigma',T')$ belongs to (3),
\item 
if $(G,\Sigma,T)$ belongs to (5) and $T'=T$, then $(G',\Sigma',T')$ belongs to (5),
\item 
if $(G,\Sigma,T)$ belongs to (3) and $T'=\emptyset$, then $(G',\Sigma',T')$ belongs to (1),
\item 
if $(G,\Sigma,T)$ belongs to (5) and $T'=\emptyset$, then $(G',\Sigma',T')$ has a blocking pair.
\end{itemize}
In all of the aforementioned cases, $(G',\Sigma',T')$ is not equal to either of $\widetilde{K_5}$ or $F_7$ (we leave this as a simple exercise), finishing the proof.
\end{proof}
\noindent 
Theorem~\ref{main2} implies that an Eulerian signed graft with two terminals that is in any of classes (1)-(6) packs.
We will now show that some of these cases lead to classical results.

Proposition~\ref{topo}(1) and theorem~\ref{main2} imply,
\begin{crly} 
Let $(H,T)$ be a graft with $|T|\leq 4$.
Suppose that every vertex of $H$ not in $T$ has even degree and that all the vertices in $T$ have degrees of the same parity. 
Then the maximum number of pairwise disjoint $T$-joins is equal to the minimum size of a $T$-cut. 
\end{crly}
\begin{proof}
Suppose that $T=\{s,t,s',t'\}$. Let $\Sigma= \delta_H(s')$ and identify $s',t'$ to obtain $G$. Denote by $v$ the vertex corresponding to $s',t'$ in $G$.
Then the signed graft $(G,\Sigma,\{s,t\})$ contains a blocking vertex $v$, 
so by proposition~\ref{topo}(1) it has no $F_7$ or $\widetilde{K}_5$ minor.
By construction $(G,\Sigma,\{s,t\})$ is Eulerian. 
Hence, theorem \ref{main2} implies that $\tau(G,\Sigma,\{s,t\})=\nu(G,\Sigma,\{s,t\})$. 
Observe that an odd $st$-join of $(G,\Sigma,\{s,t\})$ is a $T$-join of $H$, 
and that an $st$-cut or a signature of $(G,\Sigma)$ is a $T$-cut of $H$.
The result now follows.
\end{proof}
\noindent 
In fact this result holds as long as $|T|\leq 8$~\cite{Cohen97}.

Proposition~\ref{topo}(2) and theorem~\ref{main2} imply,
\begin{crly}[Hu~\cite{Hu63}, Rothschild and Whinston~\cite{Rothschild66}] 
Let $H$ be a graph and choose two pairs $(s_1,t_1)$ and $(s_2,t_2)$ of vertices, where $s_1\neq t_1$, $s_2\neq t_2$, the degrees of $s_1,t_1,s_2,t_2$ have the same parity, and all the other vertices have even degree. Then the maximum number of pairwise disjoint paths that are between $s_i$ and $t_i$ for some $i=1,2$, is equal to the minimum size of an edge subset whose deletion removes all $s_1t_1$- and $s_2t_2$-paths.
\end{crly}
\begin{proof}
 Let $\Sigma= \delta_H(s_1)\bigtriangleup \delta_H(t_2)$ and identify $s_1,s_2$ as well as $t_1,t_2$ to obtain $G$. (So all the edges between $s_1$ and $s_2$ and between $t_1$ and $t_2$ have turned into loops.)
Denote by $s$ (resp. $t$) the vertex of $G$ corresponding to $s_1,s_2$ (resp. $t_1,t_2$) in $H$.
The signed graft $(G,\Sigma,\{s,t\})$ has $\{s,t\}$ as a blocking pair, 
so by proposition~\ref{topo}(2) it has no $F_7$ or $\widetilde{K}_5$ minor.
By construction $(G,\Sigma)$ is Eulerian.
Thus, theorem \ref{main2} implies that $\tau(G,\Sigma,\{s,t\})=\nu(G,\Sigma,\{s,t\})$.
Observe that a minimal odd $st$-join of $(G,\Sigma,\{s,t\})$ is an $s_it_i$-path of $H$, for some $i=1,2$.
The result now follows.
\end{proof}

Next we shall derive corollaries using duals of plane graphs.
\begin{figure}[!ht]
\centering
\includegraphics[scale=0.15]{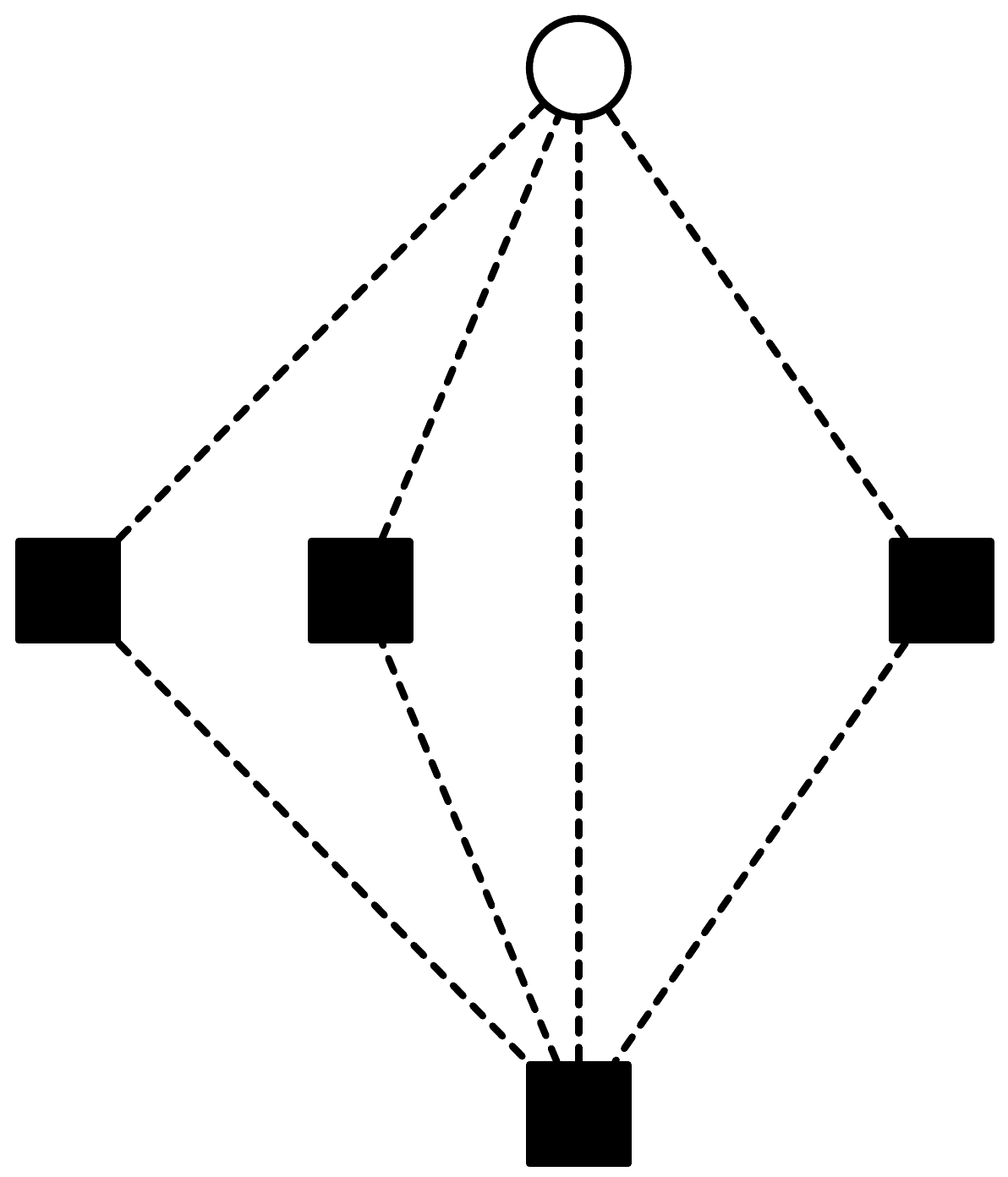}
\caption{Signed graft. All edges are in the signature and square vertices are terminals.}
\label{fig:oddtjoins}
\end{figure}

Note, in the next theorem, the {\em length} of a circuit, resp. $T$-join, is the number of its edges, and 
a circuit, resp. $T$-join, is odd, if it contains an odd number of edges in $\Sigma$.
\begin{crly}\label{complicatedT}
Let $(G,\Sigma,T)$ be a signed graft where $G$ is a plane graph with exactly two odd faces.
Suppose that $\Sigma=E(G)$ or that all $T$-joins have even length.
If $(G,\Sigma,T)$ does not contain the signed graft in figure~\ref{fig:oddtjoins} as a minor, then the maximum number of pairwise disjoint signatures is equal to the minimum of the following two quantities:
\begin{itemize}
\item the length of the shortest odd circuit, 
\item the length of the shortest odd $T$-join.
\end{itemize}
\end{crly}
\begin{proof}
Denote by $s$ and $t$ the two odd faces of $G$.
Let $G^*$ be the plane dual of $G$ and let $\Gamma$ be an odd $T$-join of $(G,\Sigma,T)$.
Then $(G^*,\Gamma,\{s,t\})$ is a signed graft. Notice that if $(G,\Sigma,T)$ is the signed graft in figure~\ref{fig:oddtjoins}, then $(G^*,\Gamma,\{s,t\})$ is $F_7$. Recall that a {\it bond} is an inclusion-wise minimal cut.
\begin{claim}
Let $B\subseteq E(G)=E(G^*)$.
\begin{enumerate}[\;\;(i)]
\item
If $B$ is an $st$-cut of $(G^*,\Gamma,\{s,t\})$ then $B$ is an odd cycle of $(G,\Sigma,T)$.
\item
If $B$ is a signature of $(G^*,\Gamma,\{s,t\})$ then $B$ is an odd $T$-join of $(G,\Sigma,T)$.
\item
If $B$ is an odd $st$-join of $(G^*,\Gamma,\{s,t\})$ then $B$ is a signature of $(G,\Sigma,T)$.
\end{enumerate}
\end{claim}
\begin{cproof}
{\bf (i)}
$B=B_1\triangle\ldots\triangle B_k$ where $B_k$ are bonds of $G^*$.
Since $B$ is an $st$-cut, an odd number of these bonds are $st$-bonds.
Thus an odd number of $B_1,\ldots,B_k$ are circuits of $G$ separating faces $s$ and $t$ and
the remainder are circuits of $G$ with faces $s$ and $t$ on the same side.
It follows that $B$ is an odd cycle of $(G,\Sigma,T)$.
{\bf (ii)}
As $B$ is a signature of $(G^*,\Gamma,\{s,t\})$, $B\triangle\Gamma=\delta_{G^*}(U)$ where $s,t\notin U$.
Denote by $u_1,\ldots,u_k$ the elements of $U$, 
then $B\triangle\Gamma=\delta_{G^*}(u_1)\triangle\ldots\triangle\delta_{G^*}(u_k)$. 
For $i\in[k]$\footnote{$[k]:=\{1,2,\ldots,k\}$}, $\delta_{G^*}(u_i)$ is a facial even
circuit of $(G,\Sigma)$ and thus $B\triangle\Gamma$ is an even cycle of $(G,\Sigma)$.
As $\Gamma$ is an odd $T$-join of $(G,\Sigma,T)$ so is $B$.
{\bf (iii)}
Since $B$ is an $st$-join of $G^*$, 
$|\delta_{G^*}(u)\cap B|$ is odd if $u=s,t$ and even otherwise.
Thus the facial circuits of $G$ that intersect $B$ with odd parity are the ones separating faces $s$ and $t$.
As the facial circuits span the cycle space of $G$, for every cycle $C$ of $G$, $|C\cap B|$ and $|C\cap\Sigma|$ have the same parity.
Hence, $B\triangle\Sigma=\delta_G(U)$ for some $U\subseteq V(G)$.
$|B\cap\Gamma|$ is odd as $B$ is an {\em odd} $st$-join of $(G^*,\Gamma,\{s,t\})$.
$|\Sigma\cap\Gamma|$ is odd as $\Gamma$ is an {\em odd} $T$-join of $(G,\Sigma,T)$.
Thus $|\delta_G(U)\cap\Gamma|=|(B\triangle\Sigma)\cap\Gamma|$ is even.
It follows that $|U\cap T|$ is even, thus $B$ is a signature of $(G,\Sigma,T)$.
\end{cproof}
\begin{claim}
$(G^*,\Gamma,\{s,t\})$ is Eulerian.
\end{claim}
\begin{cproof}
Suppose all $T$-joins of $G$ have even length.
Then any circuit of $G$ has even length.
Thus all vertices of $G^*$ have even degree.
We chose $\Gamma$ to be a $T$-join of $G$, thus $|\Gamma|$ is even.
It follows by definition that the signed graft $(G^*,\Gamma,\{s,t\})$ is Eulerian.
Suppose that $\Sigma=E(G)$.
As $s$ and $t$ are the only two odd faces of $G$, $s$ and $t$ are the only vertices of $G^*$ of odd degree.
We chose $\Gamma$ to be an odd $T$-join of $(G,\Sigma=E(G),T)$, thus $|\Gamma|$ is odd.
It follows by definition that the signed graft $(G^*,\Gamma,\{s,t\})$ is Eulerian.
\end{cproof}
Suppose now that $(G,\Sigma,T)$ does not contain the signed graft in figure~\ref{fig:oddtjoins} as a minor.
\begin{claim}
$(G^*,\Gamma,\{s,t\})$ does not contain either of $\widetilde{K_5}$ or $F_7$ as a minor.
\end{claim}
\begin{cproof}
Since $G^*$ is planar, $(G^*,\Gamma,\{s,t\})$ does not contain $\widetilde{K_5}$ as a minor.
Suppose for a contradiction that $(G^*,\Gamma,\{s,t\})/I\setminus J=F_7$.
Denote by $e_1,\ldots,e_k$ the elements of $J$ and let $(G',\Sigma',T')$ be obtained from $(G,\Sigma,T)$ 
by deleting edges in $I$ and contracting $e_1,\ldots,e_r$ for some $r\leq k$ as large as possible.
If $r=k$ then $(G',\Sigma',T')$ is given in figure~\ref{fig:oddtjoins}, a contradiction.
Otherwise, since we could not resign and contract $e_{r+1}$, $e_{r+1}$ must be in every signature of $(G',\Sigma',T')$.
Thus, by claim~1 (iii), every odd $st$-join of $(G^*,\Gamma,\{s,t\})/I\setminus\{e_1,\ldots,e_r\}$ uses $e_{r+1}$
and $(G^*,\Gamma,\{s,t\})/I\setminus J$ has no odd $st$-join, a contradiction.
\end{cproof}
By claim~2, claim~3 and theorem~\ref{main2}, $\tau=\tau(G^*,\Gamma,\{s,t\})=\nu(G^*,\Gamma,\{s,t\})$.
Thus there is a minimal cover $B$ of $(G^*,\Gamma,\{s,t\})$ with $|B|=\tau$ and
pairwise disjoint odd $st$-joins $L_1,\ldots,L_{\tau}$ of $(G^*,\Gamma,\{s,t\})$.
By proposition~\ref{coverchar} and claim~1, $B$ is either an odd circuit of $(G,\Sigma,T)$ or an odd $T$-join of $(G,\Sigma,T)$.
By claim~1, for all $i\in[\tau]$, $L_i$ is a signature of $(G,\Sigma,T)$.
\end{proof}
Next we will show that in the previous result, the case where $T$ consists of two vertices is of independent interest.
Consider $H$ obtained as follows:
\begin{itemize}
\item [($\star$)\;\;]
start from a plane graph with exactly two faces of odd length and distinct vertices $s$ and $t$, and identify $s$ and $t$.
\end{itemize}
\begin{crly}\label{cutcover}
Let $H$ be a graph as in ($\star$) and suppose that the length of the shortest odd circuit is $k$.
Then there exist cuts $B_1,\ldots,B_k$ such that every edge $e$ is in at least $k-1$ of $B_1,\ldots,B_k$.
\end{crly}
\begin{proof}
$H$ is obtained as in ($\star$) from a plane graph $G$ with exactly two faces of odd length and distinct vertices $s,t$.
The signed graft $(G,E(G),T)$ where $T=\{s,t\}$ does not contain the signed graft in figure~\ref{fig:oddtjoins} as $|T|<4$.
By corollary~\ref{complicatedT} there exists pairwise disjoint signatures 
$\Sigma_1,\ldots,\Sigma_{p}$ and $C\subseteq E(G)$ with $|C|=p$ where $C$ is an odd circuit or an odd $T$-join of $G$.
In either case $C$ is an odd circuit of $H$, thus $p\geq k$.
Since $\Sigma_1,\ldots,\Sigma_p$ are signatures of $(G,E(G),\{s,t\})$
for all $i\in[p]$, $\Sigma_i=E(G)\triangle\delta_G(U_i)=E(G)-\delta_G(U_i)$ where $s,t\notin U_i$.
Since $\Sigma_1,\ldots,\Sigma_p$ are pairwise disjoint, 
every edge of $G$ (resp. $H$) is in at least $p-1\geq k-1$ of $B_i=\delta_G(U_i)=\delta_H(U_i)$.
\end{proof}
The attentive reader may have noticed that we can also derive corollary~\ref{cutcover} 
directly from theorem~\ref{main2} and proposition~\ref{topo}(4).
Suppose that $H$ is as in ($\star$) and is loopless. 
Then by corollary~\ref{cutcover}, there exists cuts $\delta(U_1), \delta(U_2)$ 
such that every edge is in $\delta(U_1)\cup\delta(U_2)$. 
It follows that $U_1\cap U_2, U_1\cap (V(H)-U_2), (V(H)-U_1)\cap U_2, (V(H)-U_1)\cap (V(H)-U_2)$ are stable sets. 
Hence, $H$ is $4$-colourable. 

The following conjecture would generalize the 4-colour theorem,
\begin{cjr}
Let $H$ be a graph that does not contain $K_5$ as an odd minor and suppose that the length of the shortest odd circuit is $k$.
Then there exist cuts $B_1,\ldots,B_k$ such that every edge $e$ is in at least $k-1$ of $B_1,\ldots,B_k$.
\end{cjr}
Graphs in ($\star$) do not contain $K_5$ as an odd minor~\cite{Gerards93} and 
corollary~\ref{cutcover} implies the previous conjecture for these graphs.
We close this section with a sharper version of theorem~\ref{main2}.
\begin{thm}
Let $(G,\Sigma,\{s,t\})$ be an Eulerian signed graft that does not contain $\widetilde{K_5}$ or $F_7$ as a minor.
Let $k$ be the size of the smallest $st$-cut and let $\ell$ be the size of the smallest signature.
When $k\geq\ell$ one can in fact find a collection of $k$ pairwise disjoint sets, $\ell$ of which are odd $st$-join and $k-\ell$ are even $st$-paths.
\end{thm}
\begin{proof}
Let $(G',\Sigma')$ be obtained from $(G,\Sigma)$ by adding $k-\ell$ odd loops.
As $F_7$ and $\widetilde{K_5}$ have no loops, $(G',\Sigma',\{s,t\})$ does not contain $\widetilde{K_5}$ or $F_7$ as a minor.
Since $(G,\Sigma,\{s,t\})$ is Eulerian, so is $(G',\Sigma',\{s,t\})$.
It follows from theorem~\ref{main2} that $k=\tau(G',\Sigma',\{s,t\})=\nu(G',\Sigma',\{s,t\})$.
Thus there exists $k$ pairwise disjoint odd $st$-join in $(G',\Sigma',\{s,t\})$ 
and exactly $k-\ell$ must contain an odd loop that is in $(G',\Sigma')$ but not in $(G,\Sigma)$.
The result now follows.
\end{proof}
\subsection{Cycling and idealness}
A {\em clutter} $\mathcal{C}$ is a finite collection of sets, over some finite set $E(\mathcal{C})$, 
with the property that no set in $\mathcal{C}$ is contained in another set of $\mathcal{C}$.
$\mathcal{C}$ is {\em binary} if for every $S_1,S_2,S_3\in\mathcal{C}$, 
$S_1\triangle S_2\triangle S_3$ is contained in a set of $\mathcal{C}$.
A {\em cover} of a binary clutter $\mathcal C$ is a subset of $E(\mathcal{C})$
that intersects every set in $\mathcal C$ with odd parity.\footnote{This is not standard!}
An inclusion-wise minimal set of edges that intersects all sets in $\mathcal{C}$, is a cover \cite{Lehman64}.
The maximum number of pairwise disjoint sets in $\mathcal{C}$ is denoted $\nu(\mathcal{C})$.
The minimum size of a cover of $\mathcal{C}$ is $\tau(\mathcal{C})$.
$\mathcal C$ {\em packs} if $\tau(\mathcal{C})=\nu(\mathcal{C})$.
A binary clutter is {\em Eulerian} if 
all minimal covers have the same parity.

Let $\mathcal{C}$ be a clutter and $e\in E(\mathcal{C})$. 
The {\em contraction} $\mathcal{C}/e$ and {\em deletion} $\mathcal{C}\setminus e$ are clutters 
with $E(\mathcal{C}/e)=E(\mathcal{C}\setminus e)=E(\mathcal{C})-\{e\}$ where 
$\mathcal{C}/e$ is the collection of inclusion-wise minimal sets in $\{C-\{e\}: C\in \mathcal{C}\}$ 
and $\mathcal{C}\setminus e:= \{C: e\notin C\in \mathcal{C}\}$. 
A clutter obtained from $\mathcal{C}$ by a sequence of deletions and contractions is a {\em minor} of $\mathcal{C}$.
Denote 
by $\mathcal{L}_7$ the clutter of odd $T$-joins of $F_7$, 
by $\mathcal{O}_5$ the clutter of odd circuits of $K_5$, 
by $b(\mathcal{O}_5)$ the clutter of complements of cuts of $K_5$,
and
by $\mathcal{P}_{10}$ the clutter of $T$-joins of the Petersen graph where $T$ is the set of all vertices.

\begin{cjr}[Cycling Conjecture. Seymour~\cite{Seymour81}, see also Schrijver~\cite{Schrijver03}]
Eulerian binary clutters that do not contain $\mathcal{L}_7$, $\mathcal{O}_5$, $b(\mathcal{O}_5)$, or $\mathcal{P}_{10}$ as a minor, pack.
\end{cjr}
Let $(G,\Sigma,\{s,t\})$ be a signed graft and let $\mathcal{H}$ be the clutter of minimal odd $st$-joins.
Note that $\mathcal{H}$ is binary, and it can be readily checked that 
$\mathcal{H}$ is Eulerian if and only if $(G,\Sigma,\{s,t\})$ is Eulerian.
Observe also that $\mathcal{L}_7$ (resp. $\mathcal{O}_5$) is a minor of $\mathcal{H}$ if and only if
$F_7$ (resp. $\widetilde{K_5}$) is a minor of $(G,\Sigma,\{s,t\})$.
Thus theorem~\ref{main2} can be restated as,
\begin{thm}
The Cycling Conjecture holds for Eulerian clutters of minimal odd $st$-joins.
\end{thm}

Let $\mathcal{H}$ be a clutter. We define,
\begin{equation}\label{nustardef}
\nu^*(\mathcal{H})=
\max
\left\{
\sum_{S\in\mathcal{H}}\lambda_S:
\sum_{S\in\mathcal{H}:e\in S}\lambda_S\leq 1,\;\mbox{for all $e\in E(\mathcal{H})$},
\lambda_S\geq 0\;\mbox{for all $S\in\mathcal{H}$}
\right\}.
\end{equation}
$\mathcal H$ {\em fractionally packs} if $\tau(\mathcal{H})=\nu^*(\mathcal{H})$.
\begin{cjr}[Flowing Conjecture. Seymour~\cite{Seymour81,Seymour76}]
Binary clutters that do not contain $\mathcal{L}_7$, $\mathcal{O}_5$, or $b(\mathcal{O}_5)$ as a minor, fractionally pack.
\end{cjr}
\begin{crly}[Guenin~\cite{Guenin02}]
The Idealness Conjecture holds for clutters of minimal odd $st$-joins.
\end{crly}
\begin{proof}
Let $\mathcal H$ be the clutter of minimal odd $st$-joins of the signed graft $(G,\Sigma,\{s,t\})$.
Assume that ${\mathcal H}$ has no minor $\mathcal{L}_7$ or $\mathcal{O}_5$.
Then $(G,\Sigma,\{s,t\})$ has no minor $F_7$ or $\widetilde{K_5}$.
Let $(G',\Sigma',\{s,t\})$ be obtained from $(G,\Sigma,\{s,t\})$ by replacing every even (resp. odd) edge by two parallel even (resp. odd) edges.
Note that $(G',\Sigma',\{s,t\})$ also has no minor $F_7$ or $\widetilde{K_5}$.
It follows by theorem~\ref{main2} that $\tau(G',\Sigma',\{s,t\})=\nu(G',\Sigma',\{s,t\})$.
It can now be readily checked that it implies that $\tau(\mathcal{H})=\nu^*(\mathcal{H})$ as required,
where in equation~\eqref{nustardef}, $\lambda_S\in\{0,\frac12,1\}$ for all $S\in{\mathcal H}$.
\end{proof}
Applying the previous result to the case where $s=t$ we obtain,
\begin{thm}[Weakly bipartite graph theorem, Guenin~\cite{Guenin01}]
The Idealness Conjecture holds for clutters of odd circuits of graphs.
\end{thm}

\section{Organization of the proof}
%
\subsection{Extremal counterexample}
We start with the following basic result:
\begin{rem}\label{BG-min-cover}
Let $(G,\Sigma,T)$ be an Eulerian signed graft. Then the following statements hold:
\begin{enumerate}[\;\;(1)]
\item The cardinality of every signature and every $T$-cut has the same parity as $\tau(G,\Sigma,T)$.
\item Take an integer $k\geq 0$ such that $k,\tau(G,\Sigma,T)$ have different parities. If $J_1,\ldots,J_k$ are disjoint odd $T$-joins, then $E(G)-\left(\cup_{i=1}^k J_i\right)$ is also an odd $T$-join.
\end{enumerate}
\end{rem}
\begin{proof}
{\bf (1)} We leave this as an exercise.
{\bf (2)} Let $J:=E(G)-\left(\cup_{i=1}^k J_i\right)$. 
For every vertex $v\in V(G)-T$, $|\delta(v)|$ is even as the signed graft is Eulerian, so
$$|\delta(v) \cap J| \equiv |\delta(v)| - \sum_{i=1}^k |\delta(v)\cap J_i| \equiv 0 - 0 \equiv 0 \pmod{2}.$$
Moreover, for every terminal $v\in T$, $|\delta(v)|$ and $\tau(G,\Sigma,T)$ have the same parity by (1), so
$$|\delta(v) \cap J| \equiv |\delta(v)| - \sum_{i=1}^k |\delta(v)\cap J_i| \equiv \tau(G,\Sigma,T) - k \equiv 1 \pmod{2}.$$ Thus, $J$ is a $T$-join.
By (1), $|\Sigma|, \tau(G,\Sigma,T)$ have the same parity, so
$$|\Sigma\cap J| \equiv \tau(G,\Sigma,T) - 
\sum_{i=1}^k |\Sigma\cap J_i| \equiv \tau(G,\Sigma,T) - k \equiv 1 \pmod{2},$$ it follows that $J$ is an odd $T$-join, as required.
\end{proof}

A {\em counterexample} is an Eulerian signed graft with at most two terminals that does not pack and that does not contain $\widetilde{K_5}$ or $F_7$ as a minor. By remark~\ref{BG-min-cover}~(2), $\tau(G,\Sigma,T)\geq 3$ for every counterexample $(G,\Sigma,T)$.
A counterexample $(G,\Sigma,T)$ is {\em extremal} if it satisfies the following properties (in this order):
\begin{enumerate}[\;\;(M1)]
\item it minimizes $\tau(G,\Sigma,T)$,
\item it minimizes $|V(G)|$, and 
\item it maximizes $|E(G)|$.
\end{enumerate}
\begin{rem}
If there exists a counterexample then there exists an extremal counterexample.
\end{rem}
\begin{proof}
Clearly there exists a counterexample $(G,\Sigma,T)$ that minimizes (M1) and (M2) in that order.
It suffices to show that $G$ cannot have an arbitrarily large number of edges. 
For otherwise some edge $e\in E(G)$ has at least $\tau(G,\Sigma,T)$ parallel edges (all of the same parity).
But then $\tau\bigl((G,\Sigma,T)/e\bigr)=\tau(G,\Sigma,T)$, $(G,\Sigma,T)/e$ does not pack, 
it does not contain $\widetilde{K_5}$ or $F_7$ as a minor and $|V(G/e)|=|V(G)|-1$, contradicting our choice of $(G,\Sigma,T)$.
\end{proof}

Let $G$ be a graph, $U\subseteq V(G)$ and $B\subseteq E(G)$. 
We denote by $G[U]$ the graph with vertices $U$ and edges of $G$ whose ends\footnote{An {\it end} of an edge is a vertex incident to the edge.} are in $U$. 
We denote by $V_G(B)$ the set of ends of $B$ and we shall omit the subindex $G$ when there is no ambiguity.
We write $G[B]$, for the graph with edges $B$ and vertices $V(B)$. 
We say $B$ is {\it connected} if $G[B]$ is a connected graph.
Let $(G,\Sigma,T)$ be a signed graft such that $\tau(G,\Sigma,T)\geq 3$, and let $\Omega\in E(G)$. Choose $k\in [\tau(G,\Sigma,T)]-[2]$ of the same parity as $\tau(G,\Sigma,T)$. 
An {\em $(\Omega, k)$-packing} is a sequence $(L_1,\ldots,L_k)$ of odd $T$-joins where,
$\Omega\in L_1\cap L_2\cap L_3$ and $\Omega\notin L_4\cup \cdots\cup L_k$, and
$L_1,\ldots,L_k$ are pairwise $\Omega$-disjoint\footnote{Two sets $A$ and $B$ are $\Omega$-disjoint if $A\cap B\subseteq \{\Omega\}$.}.
For a subset $L\subseteq E(G)$, we say that a cover $B$ is a {\em $k$-mate} of $L$ if $|B-L|\leq k-3$ and if $B$ is either a signature or a $T$-cut. 
Moreover, $B$ is an {\it extremal $k$-mate for $L$} if, for every other $k$-mate $B'$ of $L$, $B'\cap L$ is not a proper subset of $B\cap L$.
\begin{prp}\label{BG-minimality}
Let $(G,\Sigma,T)$ be an extremal counterexample with $\tau:=\tau(G,\Sigma,T)$. Then we may assume
\begin{enumerate}[\;\;(1)]
\item $G$ is connected,
\item there exists $\Omega\in E(G)$ that is not in at least one minimum cover, \\if $T\neq\emptyset$ we can choose $\Omega\in\delta(v)$ for some $v\in T$,
\item there do not exist $\tau-1$ pairwise disjoint odd $T$-joins,
\item for every $\Omega$ as in (2), there exists an $(\Omega, \tau)$-packing,
\item every odd $T$-join has a $\tau$-mate.
\end{enumerate}
\end{prp}
\begin{proof}
{\bf (1)}
Identify a vertex of each (connected) component with an arbitrary vertex. (Neither of the obstructions $\widetilde{K_5}$, $F_7$ has a cut-vertex.)

{\bf (2)}
Let $B$ be a minimum cover. Note $B\neq E(G)$, for otherwise every edge of $B$ is an odd $T$-join and so $(G,\Sigma,T)$ packs, which is not the case.
If $T=\emptyset$ then let $\Omega\in E-B$.
Otherwise, $T=\{s,t\}$. 
Then we can pick $\Omega\in (\delta(s)\cup \delta(t))- B$.
For otherwise, $\delta(s)\cup \delta(t)\subseteq B$ and thus $\delta(s)\cup\delta(t)=\delta(s)=\delta(t)$, which by (1) implies that $E(G)=\delta(s)$, a contradiction.

{\bf (3)}
Suppose otherwise. Remove some $\tau-1$ pairwise disjoint odd $T$-join in $(G,\Sigma,T)$.
By remark~\ref{BG-min-cover}~(2), what is left is an odd $T$-join. Hence, one can actually find $\tau$ pairwise disjoint odd $T$-joins in $(G,\Sigma,T)$, contradicting the fact that $(G,\Sigma,T)$ does not pack.

{\bf (4)}
Add two parallel edges $\Omega_1,\Omega_2$ to $\Omega$ of the same parity as $\Omega$ to obtain Eulerian $(G',\Sigma',T)$.
By the choice of $\Omega$, $B$ remains a minimum cover for $(G',\Sigma',T)$, so $\tau(G',\Sigma',T)=\tau$.
Since $|V(G')|=|V(G)|$ and $|E(G')|>|E(G)|$ and since $(G,\Sigma,T)$ is an extremal counterexample,  $(G',\Sigma',T)$ packs. Hence, $(G',\Sigma',T)$ contains a set $L_1,L_2,\ldots,L_{\tau}$ of pairwise disjoint odd $T$-joins. 
All of $\Omega, \Omega_1$ and $\Omega_2$ must be used by the odd $T$-joins in $L_1,L_2,\ldots,L_{\tau}$, 
say by $L_1,L_2,L_3$, since otherwise one finds at least $\tau-1$ disjoint odd $T$-joins in $(G,\Sigma,T)$, contradicting (3).
Then $(L_1, (L_2\cup\{\Omega\})-\{\Omega_1\}, (L_3\cup\{\Omega\})-\{\Omega_2\},L_4,\ldots,L_\tau)$ is the required $(\Omega,\tau)$-packing.

{\bf (5)}
Let $L$ be an odd $T$-join.
Then the signed graft $(G,\Sigma,T)\setminus L$ packs, since $(G,\Sigma,T)$ is an extremal counterexample and $\tau\bigl((G,\Sigma,T)\setminus L\bigr)<\tau$. Let $B'$ be a minimum cover of $(G,\Sigma,T)\setminus L$. Since both $(G,\Sigma,T)$ and $(G,\Sigma,T)\setminus L$ are Eulerian, it follows that $\tau((G,\Sigma,T)\setminus L)$ and $\tau$ have different parities, and so either $\tau((G,\Sigma,T)\setminus L)\leq \tau-3$ or $\tau((G,\Sigma,T)\setminus L)=\tau-1$. However, observe that the latter is not possible, because  
$(G,\Sigma,T)$ does not pack and
$(G,\Sigma,T)\setminus L$ packs. As a result $|B'|=\tau((G,\Sigma,T)\setminus L)\leq \tau-3$. 
Let $B$ be a minimal cover contained in $B'\cup L$. Then $|B-L|\leq |B'|\leq \tau-3$. Moreover, since $B$ is a minimal cover, proposition~\ref{coverchar} implies that $B$ is either a signature or a $T$-cut. Thus $B$ is a $\tau$-mate for $L$.
\end{proof}
\subsection{$\Omega$-systems}

An edge subset of a signed graph or a signed graft is {\it bipartite} if all circuits contained in it are even.
From proposition~\ref{BG-minimality} it follows that an extremal counterexample $(G,\Sigma,T)$ has an $(\Omega,\tau)$-packing $(L_1,\ldots,L_{\tau})$. 
We distinguish between the cases where $(L_1\cup L_2\cup L_3)-\{\Omega\}$ is bipartite or non-bipartite and define the appropriate data structure in each case.

A {\em non-bipartite $\Omega$-system} consists of a pair $\bigl((G,\Sigma,T),(L_1,\ldots,L_k)\bigr)$ where $\tau(G,\Sigma,T)\geq 3$, $k\in \{3,\ldots,\tau(G,\Sigma,T)\}$, $k$ has the same parity as $\tau(G,\Sigma,T)$, and
\begin{enumerate}[\;\;\;(N1)]
\item $(G,\Sigma,T)$ is an Eulerian signed graft with $|T|\leq 2$, and if $T=\{s,t\}$, then $\Omega\in \delta(s)$,
\item $(L_1,\ldots,L_k)$ is an $(\Omega,k)$-packing where $L_1,\ldots,L_k$ are minimal odd $T$-joins,
\item $(L_1\cup L_2\cup L_3)-\{\Omega\}$ is non-bipartite, and
\item every odd $T$-join $L\subseteq L_1\cup L_2\cup L_3$ has a $k$-mate.
\end{enumerate}

To define the other data structures, we need some terminology.
Let $(G,\Sigma,T)$ be a signed graft where $|T|\leq 2$ and let $L$ be a minimal odd $T$-join. Define $C(L)$ and $P(L)$ as follows: \begin{enumerate}[\;\;]
\item if $T=\emptyset$, then $L$ is an odd circuit and we define $P(L):=\emptyset$ and $C(L):=L$,
\item if $T=\{s,t\}$ and $L$ is an odd $st$-path, we define $P(L):=L$ and $C(L):=\emptyset$,
\item otherwise, $T=\{s,t\}$ and $L$ is the disjoint union of an even $st$-path, denoted $P(L)$, and an odd circuit, denoted $C(L)$.
\end{enumerate}
We say that $L$ is {\em simple} if $C(L)=\emptyset$ (see figure~\ref{fig:simple}) \begin{figure}[!ht]
\centering
\includegraphics[scale=0.3]{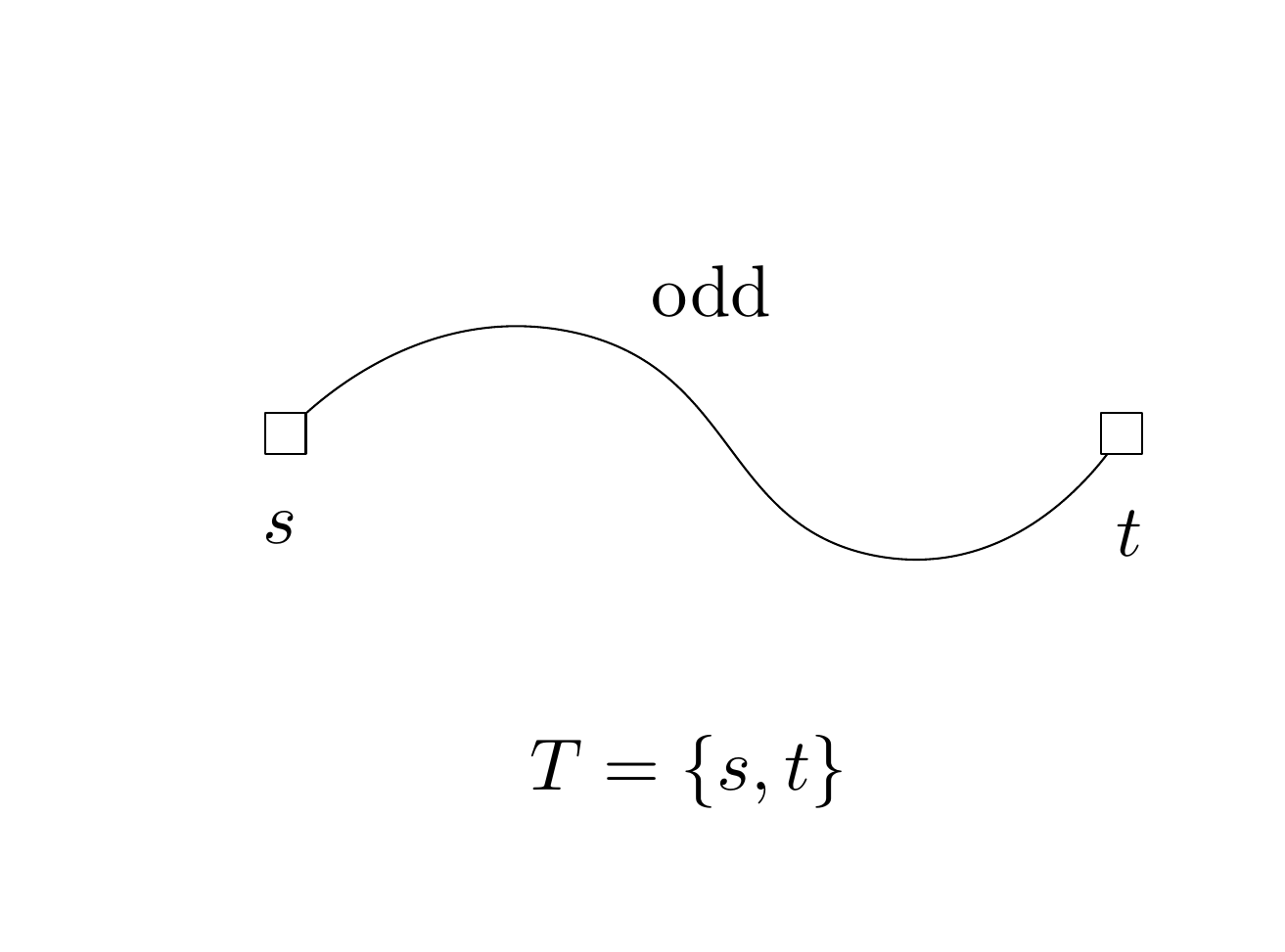}
\caption{An illustration of simple odd $T$-joins.}
\label{fig:simple}
\end{figure} and it is {\em non-simple} otherwise (see figure~\ref{fig:nonsimple}). 
\begin{figure}[!ht]
\centering
\includegraphics[scale=0.3]{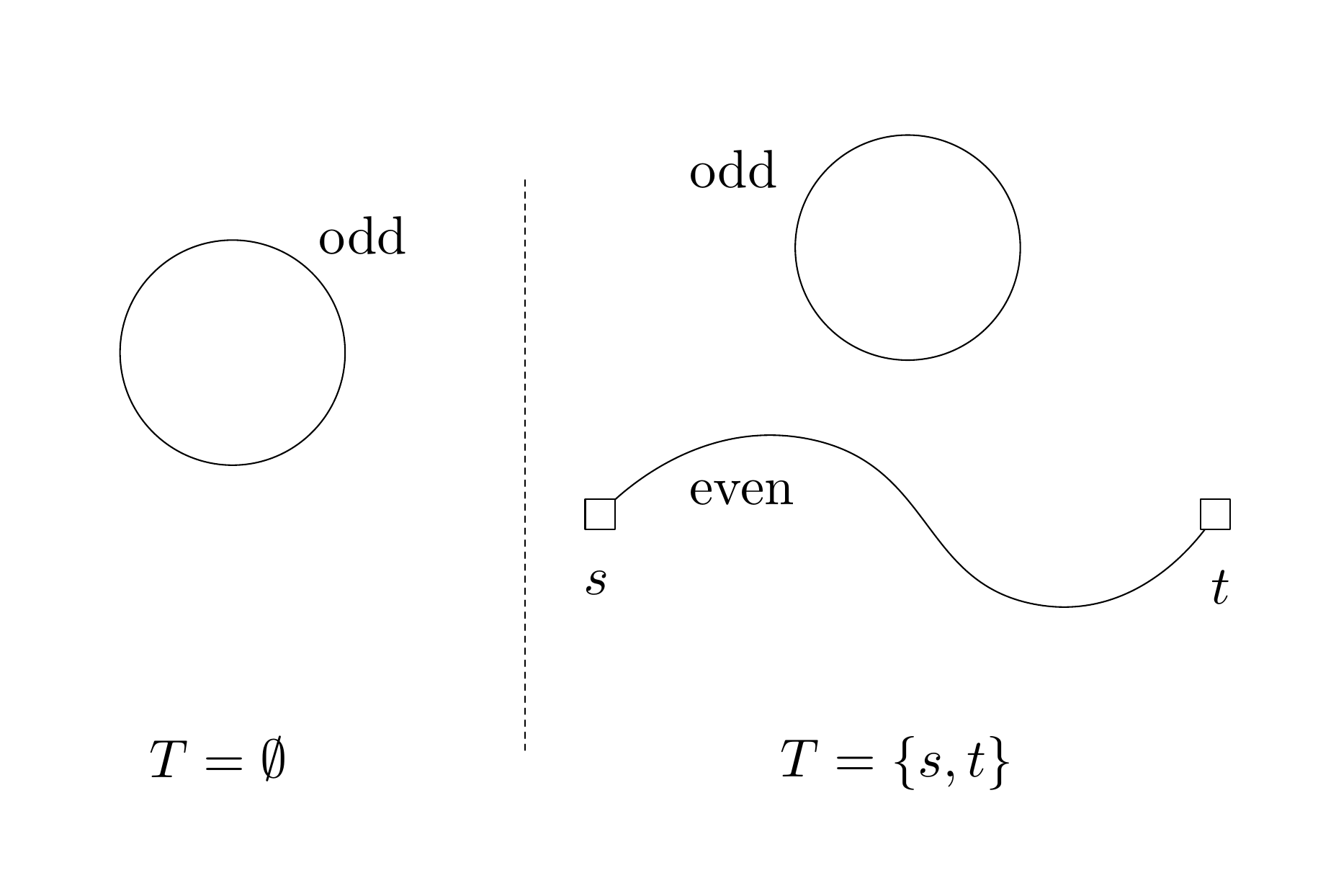}
\caption{An illustration of non-simple odd $T$-joins.}
\label{fig:nonsimple}
\end{figure}

A cycle (in a directed graph) is {\it directed} if it is the disjoint union of directed circuits.
An $st$-join is {\it directed} if it is the disjoint union of some $st$-dipaths and some directed circuits.

A {\em bipartite $\Omega$-system} consists of a tuple $\bigl((G,\Sigma,T), (L_1,\ldots,L_k), m\bigr)$ where $\tau(G,\Sigma,T)\geq 3$, $k\in \{3,\ldots,\tau(G,\Sigma,T)\}$, $k$ has the same parity as $\tau(G,\Sigma,T)$, and
\begin{enumerate}[\;\;\;(B1)]
\item $(G,\Sigma,T)$ is an Eulerian signed graft with $|T|\leq 2$, and if $T= \{s,t\}$, then $\Omega\in \delta(s)$,
\item $(L_1,\ldots,L_k)$ is an $(\Omega,k)$-packing and $m\in [k]-[2]$ where
\begin{enumerate}[\;\;]
\item if $T=\emptyset$, then $m=3$,
\item if $T=\{s,t\}$, then for each $j\in [m]-[3]$, $L_j$ contains an even $st$-path $P_j$ and an odd circuit $C_j$ that are (edge-)disjoint,
\item if $T=\{s,t\}$, then for each $j\in [k]-[m]$, $L_j$ is connected,
\end{enumerate}
\item $\Sigma\cap (L_1\cup L_2\cup L_3\cup P_4\cup \ldots \cup P_m)=\{\Omega\}$.
\end{enumerate}

A {\em non-simple bipartite $\Omega$-system} consists of a tuple $\bigl((G,\Sigma,T), (L_1,\ldots,L_k), m, \vec{H}\bigr)$ where 
\begin{enumerate}[\;\;\;(NS1)]
\item $\bigl((G,\Sigma,T), (L_1,\ldots,L_k), m\bigr)$ is a bipartite $\Omega$-system, 
\item $L_1,L_2,L_3$ are minimal odd $T$-joins, and at least one of them is non-simple,
\item \begin{enumerate} []
\item $H=G[L_1\cup L_2\cup L_3\cup P_4\cup \ldots \cup P_m]$,
\item $L_1,L_2, L_3$ are directed $T$-joins in $\vec{H}$ (if $T=\{s,t\}$ then they are directed $st$-joins),
\item if $T=\{s,t\}$, $P_4, \ldots , P_m$ are $st$-dipaths in $\vec{H}$,
\item $\vec{H}\setminus \Omega$ is acyclic,
\end{enumerate}
\item in $\vec{H}$, every odd directed $T$-join that is $\Omega$-disjoint from some odd directed circuit, has a $k$-mate.
\end{enumerate}

A {\em simple bipartite $\Omega$-system} consists of a tuple $\bigl((G,\Sigma,\{s,t\}), (L_1,\ldots,L_k), m, \vec{H}\bigr)$ where 
\begin{enumerate}[\;\;\;(S1)]
\item $\bigl((G,\Sigma,\{s,t\}), (L_1,\ldots,L_k), m\bigr)$ is a bipartite $\Omega$-system, 
\item \begin{enumerate} []
\item $H=G[L_1\cup L_2\cup L_3\cup P_4\cup \ldots \cup P_m]$,
\item $L_1,L_2, L_3$ are odd $st$-dipaths in $\vec{H}$, 
\item $P_4, \ldots ,P_m$ are $st$-dipaths in $\vec{H}$,
\item $\vec{H}$ is acyclic,
\end{enumerate}
\item in $\vec{H}$, every odd $st$-dipath has a $k$-mate.
\end{enumerate}

\begin{prp}\label{minimal-to-omega}
An extremal counterexample has a non-bipartite, non-simple bipartite, or simple bipartite $\Omega$-system.
\end{prp}
\noindent The proof of this proposition is provided in \S\ref{sec-classification}.

Given a bipartite $\Omega$-system $((G,\Sigma,\{s,t\}),(L_1,\ldots,L_k),m)$, we define two cut structures.

A {\it primary cut structure} is a sequence $(U_1,\ldots,U_n)$ where \begin{enumerate}[\;\;(PC1)]
\item $L_2,L_3$ are odd $st$-paths,
\item $n\in [m-2]$ and $s\in U_1\subset \cdots \subset U_n\subseteq V(G)-\{t\}$,
\item for each $i\in [n-1]$, there exist $q_i\in U_i$, {\it base} $Q_{3+i}$ and {\it residue} $R_{3+i}$, where $Q_{3+i}\subset L_{3+i}-C_{3+i}$ is a $q_it$-path such that $V(Q_{3+i})\cap U_i=\{q_i\}$, $R_{3+i}\subset L_{3+i}-C_{3+i}$ is a connected $sq_i$-join, and $Q_{3+i}\cap R_{3+i}=\emptyset$ (see figure~\ref{fig:cut-structure}),
\item for each $i\in [n-1]$, $\delta(U_i)$ is a $k$-mate of $R_{3+i}\cup Q_{3+i}$,
and for every proper subset $W$ of $U_i$ with $s\in W$, $\delta(W)$ is not a $k$-mate of $R_{3+i}\cup Q_{3+i}$,
\item $\delta(U_n)$ is a $k$-mate of $L_1$, and for every proper subset $W$ of $U_n$ with $s\in W$, $\delta(W)$ is not a $k$-mate of $L_1$,
\item there exist $d,q\in U_n$ and a partition of $L_1$ into {\it base} $Q$, {\it brace} $D$ and {\it residue} $R$, where 
$Q$ is a $qt$-path with $V(Q)\cap U_n=\{q\}$,
$D$ is an $sd$-path containing $\Omega$ with $V(D)\cap U_n=\{s,d\}$ that is vertex-disjoint from $Q$ outside $U_n$,
and $R$ is a connected $dq$-join (see figure~\ref{fig:cut-structure-2}).
\end{enumerate}
For $i\in [m]-[n+2]$, set $Q_i:=P_i$, $R_i:=\emptyset$, and call $Q_i$ the {\it base} of $L_i$, and for $i=2,3$, set $Q_i:= P_i=L_i$ and call $Q_i$ the {\it base} of $L_i$.

A {\it secondary cut structure} is a sequence $(U_1,\ldots,U_n)$ where \begin{enumerate}[\;\;(SC1)]
\item $L_1,L_2,L_3$ are odd $st$-paths,
\item $m\geq 4$, $n\in [m-3]$ and $s\in U_1\subset \cdots \subset U_n\subseteq V(G)-\{t\}$,
\item for each $i\in [n]$, there exist $q_i\in U_i$, {\it base} $Q_{3+i}$ and {\it residue} $R_{3+i}$, where $Q_{3+i}\subset L_{3+i}-C_{3+i}$ is a $q_it$-path such that $V(Q_{3+i})\cap U_i=\{q_i\}$, $R_{3+i}\subset L_{3+i}-C_{3+i}$ is a connected $sq_i$-join, and $Q_{3+i}\cap R_{3+i}=\emptyset$ (see figure~\ref{fig:cut-structure}),
\item for each $i\in [n]$, $\delta(U_i)$ is a $k$-mate of $R_{3+i}\cup Q_{3+i}$,
and for every proper subset $W$ of $U_i$ with $s\in W$, $\delta(W)$ is not a $k$-mate of $R_{3+i}\cup Q_{3+i}$.
\end{enumerate}
For $i\in [m]-[n+3]$, set $Q_i:=P_i$, $R_i:=\emptyset$, and call $Q_i$ the {\it base} of $L_i$, and for $i\in [3]$, set $Q_i:=P_i=L_i$ and call $Q_i$ the {\it base} of $L_i$.

\begin{figure}[!ht]
\centering
\includegraphics[scale=0.3]{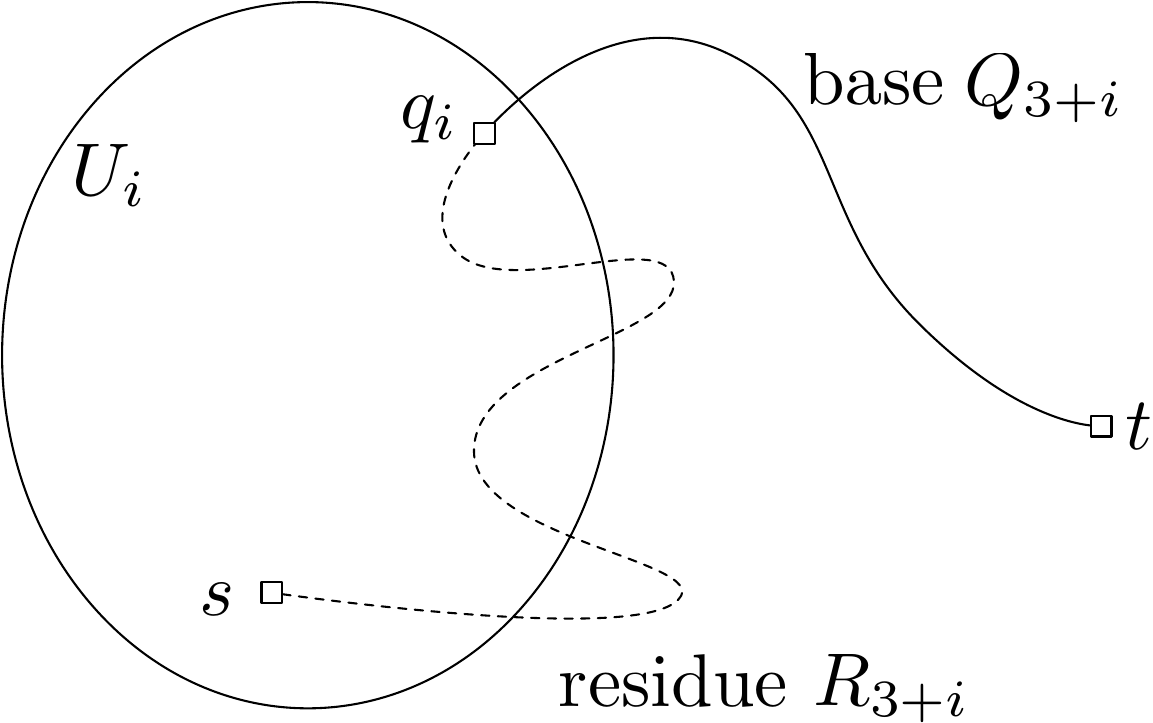}
\caption{Bases and residues of primary ($i\in [n-1]$) and secondary cut structures ($i\in [n]$).}
\label{fig:cut-structure}
\end{figure}
\begin{figure}[!ht]
\centering
\includegraphics[scale=0.3]{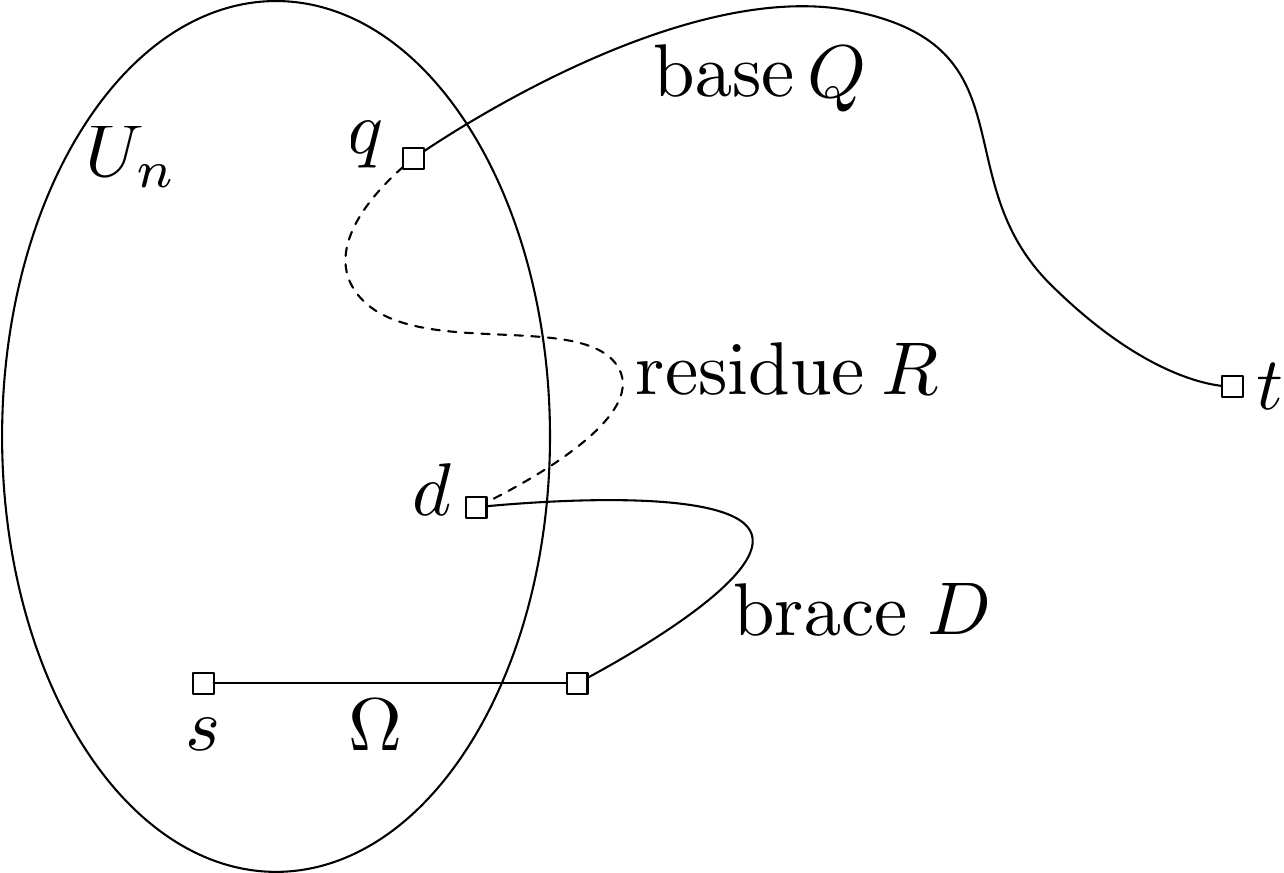}
\caption{The base, residue and brace of $U_n$ for the primary cut structure.}
\label{fig:cut-structure-2}
\end{figure}

A {\em cut $\Omega$-system} consists of a tuple $\bigl((G,\Sigma,\{s,t\}), (L_1,\ldots,L_k), m, (U_1,\ldots,U_n),\vec{H}\bigr)$ where 
\begin{enumerate}[\;\;\;(C1)]
\item $\bigl((G,\Sigma,\{s,t\}), (L_1,\ldots,L_k), m\bigr)$ is a bipartite $\Omega$-system,
\item $(U_1,\ldots,U_n)$ is a primary or a secondary cut structure,
\item \begin{enumerate}[\;\;]
\item $H$ is the union of all bases and, if it exists, the brace,
\item the brace, if it exists, is an $sd$-dipath in $\vec{H}$,
\item the bases are directed paths in $\vec{H}$ rooted towards $t$,
\item the following digraph $\vec{H}^+$ is acyclic: start from $\vec{H}$, for each $q_i$ add arc $(s,q_i)$, and if $d,q$ existed and $d\neq q$, add arc $(d,q)$,
\item $\Sigma\cap E(H)=\{\Omega\}$ and $\Sigma$ has no edge in common with any of the residues.
\end{enumerate}
\item for every odd $st$-dipath $P$ in $\vec{H}$ such that $V(P)\cap U_n=\{s\}$, there is a $k$-mate for $P$.
\end{enumerate}
Consider a non-bipartite $\Omega$-system $\bigl((G,\Sigma,T),{\mathcal L}\bigr)$.
Then ${\mathcal L}$ is the $(\Omega, k)$-packing {\em associated with} the $\Omega$-system and $(G,\Sigma,T)$ is the signed graft {\em associated with} the $\Omega$-system. 
Similarly, one defines the associated $(\Omega,k)$-packing and the associated signed graft for bipartite and cut $\Omega$-systems. 
We say that an $\Omega$-system has a particular minor when the associated signed graft does. 
Theorem~\ref{main2} follows from proposition~\ref{minimal-to-omega} and the following three results,
\begin{prp}\label{main-result-non-bipartite}
A non-bipartite $\Omega$-system has an $F_7$ minor.
\end{prp}
\begin{prp}\label{main-result-bipartite-nonsimple}
A non-simple bipartite $\Omega$-system has an $F_7$ or a $\widetilde{K_5}$ minor.
\end{prp}
\begin{prp}\label{main-result-bipartite-simple}
A simple bipartite $\Omega$-system has an $F_7$ minor.
\end{prp}
\subsection{Outline of the proof}
In this section we discuss the outline of the proofs of propositions~\ref{main-result-non-bipartite},~\ref{main-result-bipartite-nonsimple} and~ \ref{main-result-bipartite-simple}.

A non-bipartite $\Omega$-system $((G,\Sigma,T), (L_1,\ldots,L_k))$ comes in the following {\em flavours}:
\begin{enumerate}[\;\;\;(NF1)]
\item at least two of $L_1,L_2,L_3$ are non-simple, and for $i\in [3]$, if $L_i$ is non-simple then $\Omega\in P(L_i)$.
\item at most one of $L_1,L_2,L_3$ is non-simple, and for $i\in [3]$, if $L_i$ is non-simple then $\Omega\in C(L_i)$.
\end{enumerate}
Note that $T\neq \emptyset$ for both flavours (NF1) and (NF2). We will postpone the proof of the next result to Section~\ref{sec-prelim-nb}. 
\begin{prp}\label{nbflavours}
Every non-bipartite $\Omega$-system is of flavour (NF1) or (NF2).
\end{prp}
\noindent A non-bipartite $\Omega$-system $((G,\Sigma,\{s,t\}), (L_1,\ldots,L_k))$ is {\em minimal} if 
(a) there is no non-bipartite $\Omega$-system whose associated signed graft is a proper minor of $(G,\Sigma,\{s,t\})$, and 
(b) among all non-bipartite $\Omega$-systems with the same associated signed graft, $|L_1\cup L_2\cup L_3|$ is minimized. 
Note that every non-bipartite $\Omega$-system contains as a minor a minimal non-bipartite $\Omega$-system. 
Proposition~\ref{main-result-non-bipartite} will follow from the following results,
\begin{prp}\label{prp-NF1}
A minimal non-bipartite $\Omega$-system of flavour (NF1) has an $F_7$ minor.
\end{prp}
\begin{prp}\label{prp-NF2}
Consider a minimal non-bipartite $\Omega$-system of flavour (NF2) and assume that there
is no non-bipartite $\Omega$-system of flavour (NF1) with the same associated signed graft. 
Then the $\Omega$-system has an $F_7$ minor.
\end{prp}

A non-simple bipartite $\Omega$-system $((G,\Sigma,T), (L_1,\ldots,L_k),m, \vec{H})$ is {\em minimal} if there is no non-simple bipartite $\Omega$-system whose associated signed graft is a proper minor of $(G,\Sigma,T)$.
Proposition~\ref{main-result-bipartite-nonsimple} is proved for minimal non-simple bipartite $\Omega$-systems, which clearly is sufficient.

A simple bipartite $\Omega$-system $((G,\Sigma,T), (L_1,\ldots,L_k),m,\vec{H})$ comes in the following {\em flavours}:
\begin{enumerate}[\;\;\;(SF1)]
\item no odd $st$-dipath of $\vec{H}$ has an $st$-cut $k$-mate,
\item some odd $st$-dipath of $\vec{H}$ has an $st$-cut $k$-mate.
\end{enumerate}
A simple bipartite $\Omega$-system $((G,\Sigma,\{s,t\}), (L_1,\ldots,L_k),m, \vec{H})$ is {\em minimal} if there is no simple bipartite $\Omega$-system whose associated signed graft is a proper minor of $(G,\Sigma,\{s,t\})$. 
Proposition~\ref{main-result-bipartite-simple} will follow from the following results,
\begin{prp}\label{prp-SF1}
Let $((G,\Sigma,\{s,t\}), (L_1,\ldots,L_k), m, \vec{H})$ be a minimal simple bipartite $\Omega$-system of flavour (SF1) and assume that there is no non-simple bipartite $\Omega$-system whose associated signed graft is a minor of $(G,\Sigma,\{s,t\})$. Then the $\Omega$-system has an $F_7$ minor.
\end{prp}
\begin{prp}\label{prp-SF2}
Let $((G,\Sigma,\{s,t\}), (L_1,\ldots,L_k), m, \vec{H})$ be a minimal simple bipartite $\Omega$-system of flavour (SF2) and assume that there is no non-simple bipartite $\Omega$-system whose associated signed graft is a minor of $(G,\Sigma,\{s,t\})$. Then the $\Omega$-system has an $F_7$ minor.
\end{prp}
Our proof of proposition~\ref{prp-SF2} is more involved.
\begin{prp}\label{SF2-cut}
A simple bipartite $\Omega$-system of flavour (SF2) has a cut $\Omega$-system.
\end{prp}
\begin{proof}
Let $((G,\Sigma,T),(L_1,\ldots,L_k),m,\vec{H})$ be a simple bipartite $\Omega$-system of flavour (SF2). After redefining $\mathcal{L}$, if necessary, we may assume that $L_1$ has an $st$-cut $k$-mate. Choose $U_1\subseteq V(G)-\{t\}$ with $s\in U_1$ such that $\delta(U_1)$ is a $k$-mate of $L_1$, and for every proper subset $W$ of $U_1$ with $s\in W$, $\delta(W)$ is not a $k$-mate of $L_1$. It is easily seen that $(U_1)$ is a primary cut structure. Let $R$ be the residue for $L_1$, and update $\vec{H}:=\vec{H}\setminus R$.
It is easily seen that $((G,\Sigma,T),(L_1,\ldots,L_k),m,(U_1),\vec{H})$ is a cut $\Omega$-system.
\end{proof}

Let $\bigl((G,\Sigma,\{s,t\}), (L_1,\ldots,L_k), m, (U_1,\ldots,U_n),\vec{H}\bigr)$ be a cut $\Omega$-system.
The $\Omega$-system is {\em minimal} if, among all cut $\Omega$-systems whose associated signed graft is a minor of $(G,\Sigma,\{s,t\})$, $|E(\vec{H})|$ is minimized, and the size $n$ of the cut structure is maximized, in this order of priority.
The $\Omega$-system is {\em primary} (resp. {\em secondary}) if $(U_1,\ldots,U_n)$ is a primary (resp. secondary) cut structure. Proposition~\ref{prp-SF2} will follow from proposition~\ref{SF2-cut} and the following results,
\begin{prp}\label{prp-cut-primary}
Let $\bigl((G,\Sigma,\{s,t\}), (L_1,\ldots,L_k), m, (U_1,\ldots,U_{n-1},U),\vec{H}\bigr)$ be a minimal cut $\Omega$-system that is primary and assume there is no non-simple bipartite $\Omega$-system whose associated signed graft is a minor of $(G,\Sigma,\{s,t\})$.
Then the $\Omega$-system has an $F_7$ minor.
\end{prp}
\begin{prp}\label{prp-cut-secondary}
Let $\bigl((G,\Sigma,\{s,t\}), (L_1,\ldots,L_k), m, (U_1,\ldots,U_n),\vec{H}\bigr)$ be a minimal cut $\Omega$-system that is secondary and assume there is no non-simple bipartite $\Omega$-system whose associated signed graft is a minor of $(G,\Sigma,\{s,t\})$.
Then the $\Omega$-system has an $F_7$ minor.
\end{prp}
\subsection{Organization of the paper}
Section~\ref{sec-prelim-nb} develops some preliminary results for non-bipartite $\Omega$-systems.
The proof of proposition~\ref{prp-NF1} for $\Omega$-systems of flavour (NF1) is given in \S\ref{sec-NF1}.
The proof of proposition~\ref{prp-NF2} for $\Omega$-systems of flavour (NF2) is given in \S\ref{sec-NF2}.
Section~\ref{sec-prelim-bp} develops some preliminary results for bipartite $\Omega$-systems.
The proof of proposition~\ref{main-result-bipartite-nonsimple}, along with preliminaries, is given in \S\ref{sec-setup-simple}, \S\ref{sec-NS-I}, \S\ref{sec-NS-II} and \S\ref{sec-NS-III}.
Section~\ref{sec-linkage} describes another preliminary and the proof of proposition~\ref{prp-SF1} can be found in \S\ref{sec-SF1}.
Section~\ref{sec-shore} develops our last preliminary and the proofs of propositions~\ref{prp-cut-primary} and~\ref{prp-cut-secondary} can be found in \S\ref{sec-cut-primary}, \S\ref{sec-cut-secondary}, respectively. 
The outline is summarized in figure~\ref{fig:outline}.
\begin{figure}[!ht]
\centering
\includegraphics[width=\textwidth,height=\textheight,keepaspectratio]{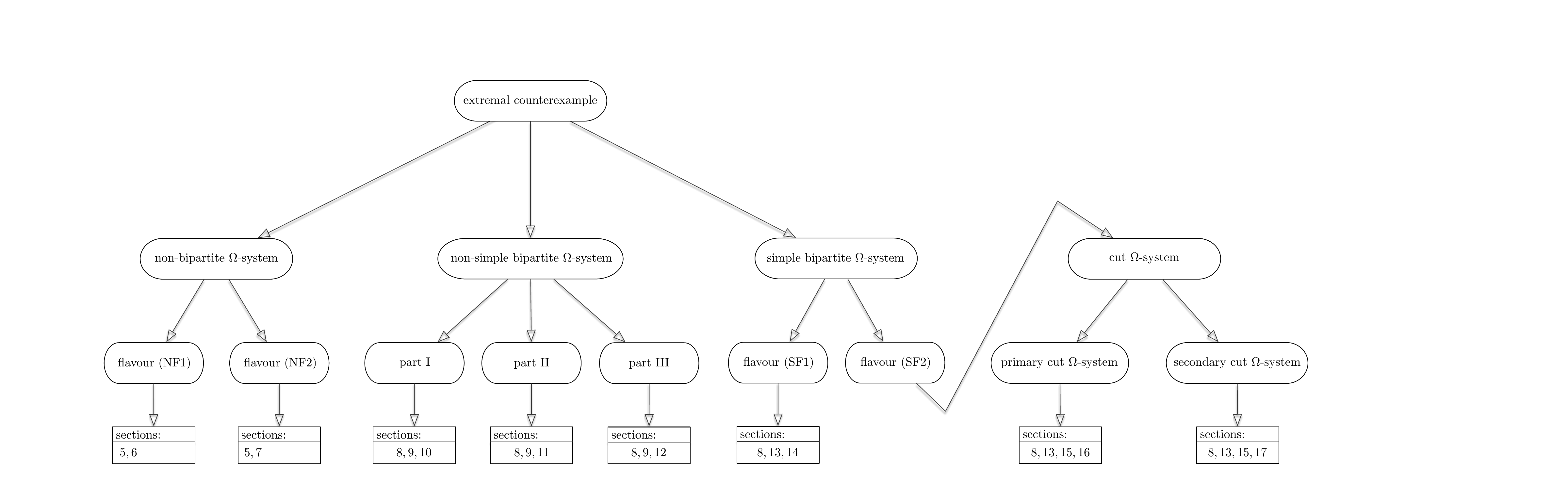}
\caption{Outline of the proof.}
\label{fig:outline}
\end{figure}
%

\section{Covers}\label{sec-covers}
In this section, we develop tools that will be helpful in dealing with covers.
\subsection{Caps and mates}
Let $(G,\Sigma,T)$ be a signed graft and let $\mathcal L=(L_1,\ldots,L_k)$ be an $(\Omega, k)$-packing.
We say that for $\ell\in[k]$ a set $B\subseteq E(G)$ is a {\em cap of $L_\ell$ in $\mathcal L$} if the following hold, 
\begin{enumerate}[\;\;(T1)]
\item $B$ is either a signature or a $T$-cut,
\item $\Omega\in B$,
\item $B\subseteq L_1\cup\ldots\cup L_k$, and
\item for all $i\in [k]-\{\ell\}$, $|B\cap L_i|=1$, and $|B\cap L_\ell|\geq 3$.
\end{enumerate}
\noindent
The next result characterizes $k$-mates of sets in an $(\Omega, k)$-packing.
\begin{prp}\label{packmate}
Let $(G,\Sigma,T)$ be a signed graft and $\mathcal L=(L_1,\ldots,L_k)$ be an $(\Omega,k)$-packing.
Then for $\ell\in[k]$, $B$ is a $k$-mate of $L_\ell$ if and only if $B$ is a cap of $L_\ell$ in $\mathcal L$.
\end{prp}
\begin{proof}
Suppose first that $B$ is a $k$-mate of $L_\ell$. By definition of $k$-mates, (T1) holds and $|B-L_\ell|\leq k-3$.
(T2) holds for otherwise, $B\cap L\neq\emptyset$ for all $L\in{\mathcal L}$ which implies $|B-L_\ell |\geq |\mathcal L|-1= k-1$, a contradiction. 
If $\ell\in[3]$, then $B-L_\ell$ intersects the $k-3$ pairwise disjoint sets $L_4,\ldots,L_k$. 
If $\ell\in [k]-[3]$, then $B-L_\ell$ intersects the $k-3$ pairwise disjoint sets in $\{L_3, L_4,\ldots,L_k\}-\{L_\ell\}$.
In either cases $|B-L_\ell|=k-3$ and (T3) and (T4) hold.

Suppose (T1)-(T4) hold.
Suppose $\ell \in [3]$ say $\ell=1$.
Then $B-L_1\subseteq L_4\cup\ldots\cup L_k$.
Moreover, $|B\cap L_i|=1$ for all $i\in\{4,\ldots,k\}$.
Thus $|B-L_1|\leq k-3$, so $B$ is a $k$-mate of $L_1$.
Suppose $\ell \notin [3]$ say $\ell=4$.
Then $B-L_4\subseteq \{\Omega\}\cup L_5\cup\ldots\cup L_k$.
Thus $|B-L_4|\leq k-3$, so $B$ is a $k$-mate of $L_4$.
\end{proof}
\begin{prp}\label{usefulparity4}
Let $(G,\Sigma,T)$ be a signed graft and let 
$L_4,\ldots,L_k$ be pairwise disjoint odd $T$-joins.
Let $L$ be a subset of $E(G)-(L_4\cup \ldots \cup L_k)$ that has a $k$-mate $B$. Then $B\subseteq L\cup L_4\cup \ldots\cup L_k$.
\end{prp}
\begin{proof} 
We have $$k-3\leq \sum_{i=4}^k |B\cap L_i|\leq |B-L|\leq k-3,$$ 
where the first inequality follows from $B\cap L_i\neq\emptyset$,
the second as $L\cap (L_4\cup \ldots \cup L_k)=\emptyset$ 
and the third because $B$ is a $k$-mate of $L$.
Hence, equality holds throughout, so $|B-L|=k-3$ and the result follows.
\end{proof}
\begin{prp}\label{usefulparity3}
Let $(G,\Sigma,T)$ be a signed graft and take two $(\Omega,k)$-packings
$$\mathcal L=(L_1, L_2,L_3,L_4,\ldots,L_k) \quad \mbox{and}\quad
\mathcal L'=(L'_1, L'_2,L_3,L_4,\ldots,L_k).$$
Let $B_1, B'_1$ be $k$-mates of $L_1, L'_1$, respectively.
Let $B\subseteq B_1\cup B'_1$ be a cover that is either a signature or a $T$-cut. Then,
\begin{enumerate}[\;\;(1)]
\item $\Omega\in B$, 
\item $B\subseteq L_1\cup L'_1\cup L_4\cup\ldots\cup L_k$,
\item $|B\cap L_i|=1$ for all $i\in\{3,\ldots,k\}$, 
\item $B$ is a $k$-mate of $L_1\cup L'_1$,
\item $|B\cap L_1|\geq 3$ or $|B\cap L'_1|\geq 3$,
\item if $B\cap (L'_1-L_1)=\emptyset$ then $B$ is a $k$-mate of $L_1$,
\item if $B\cap (L'_1-L_1)=B\cap (L_1-L'_1)=\emptyset$ then $B$ is a $k$-mate of $L_1\cap L'_1$.\end{enumerate}
\end{prp}
\begin{proof}
By proposition~\ref{packmate} $B_1$ (resp. $B'_1$) is a cap of $L_1$ (resp. $L'_1$) in $\mathcal L$ (resp. $\mathcal{L}'$). 
Thus,
\begin{align}
B_1\cup B'_1 &\subseteq L_1\cup L'_1\cup L_4\cup\ldots\cup L_k,\tag{a}\\
|B_1 \cap L_i|=|B'_1\cap L_i|&=1\quad\mbox{for all}\quad i\in\{4,\ldots,k\},\tag{b}
\end{align}
Since $B\subseteq B_1\cup B'_1$, (a) implies that (2) holds.
As $B$ is a cover and $B\cap L_3\neq \emptyset$, (1) must hold as well.
Let $i\in\{4,\ldots,k\}$.  Then by (b)
\[
|B\cap L_i|\leq|B_1\cap L_i|+|B'_1\cap L_i|\leq 2.
\]
Hence, as $B$ is a cover, $|B\cap L_i|=1$ so (3) holds. Combining this with (a) yields $$|B-(L_1\cup L'_1)|\leq \sum_{i=4}^k |B\cap L_i|=k-3$$ and so $B$ is a $k$-mate of $L_1\cup L'_1$ so (4) holds. 
It follows (as every cover has cardinality at least $\tau(G,\Sigma)\geq k$) that $|B\cap (L_1\cup L'_1)|\geq 3$.
Hence, for some $L\in \{L_1,L'_1\}$, $|B\cap L|>1$ and so $|B\cap L|\geq 3$ thus (5) holds.
(6) and (7) trivially follow from (4).
\end{proof}
The following are immediate corollaries.
\begin{prp}\label{usefulparity}
Let $(G,\Sigma,T)$ be a signed graft and $\mathcal L=(L_1,\ldots,L_k)$ be an $(\Omega, k)$-packing.
Suppose for $i=1,2$, $B_i$ is a $k$-mate of $L_i$ and let $B\subseteq B_1\cup B_2$ be a cover that is either a signature or a $T$-cut. Then
\begin{enumerate}[\;\;(1)]
\item $\Omega\in B$, 
\item $B\subseteq L_1\cup L_2\cup L_4\cup\ldots\cup L_k$,
\item $|B\cap L_i|=1$ for all $i\in\{3,\ldots,k\}$, 
\item $|B\cap L_1|\geq 3$ or $|B\cap L_2|\geq 3$,
\item for $i=1,2$, if $|B\cap L_i|=1$ then $B$ is a $k$-mate of $L_{3-i}$.
\end{enumerate}
\end{prp}
\begin{proof}
Choose $\mathcal L'=(L_2,L_1,L_3,\ldots,L_k)$ and apply proposition~\ref{usefulparity3} parts (5) and (6).
\end{proof}
\begin{prp}\label{usefulparity2}
Let $(G,\Sigma,T)$ be a signed graft and $\mathcal L=(L_1,\ldots,L_k)$ be an $(\Omega, k)$-packing.
Suppose $B_1$ and $B'_1$ are $k$-mates of $L_1$ and let $B\subseteq B_1\cup B'_1$ be a cover that is either a signature or a $T$-cut. Then $B$ is also a $k$-mate of $L_1$.
\end{prp}
\begin{proof}
Choose $\mathcal L'=\mathcal L$ and apply proposition~\ref{usefulparity3}(6).
\end{proof}
\subsection{Signatures versus $T$-cuts}
\begin{prp}\label{matescutcut}
Let $(G,\Sigma,T)$ be a signed graft with $|T|\leq 2$ and let $(L_1,\ldots,L_k)$ be an $(\Omega, k)$-packing.
Suppose that $L_1,L_2$ are minimal odd $T$-joins and, for $i=1,2$ $L_i$ is simple or $\Omega\in C(L_i)$.
Suppose further that for $i=1,2$ there exists a $k$-mate $B_i$ of $L_i$.
Then one of $B_1,B_2$ is a signature.
\end{prp}
\begin{proof}
By proposition~\ref{packmate}, for each $i=1,2$, $B_i$ is a cap of $L_i$ in $\mathcal{L}$. Thus, $B_1\cap L_2=B_2\cap L_1=\{\Omega\}$. Hence, if $\Omega\in C(L_1)$ then 
$B_2\cap C(L_1)=\{\Omega\}$, implying that $B_2$ is a signature.
Similarly, if $\Omega\in C(L_2)$ then $B_1$ is a signature. 
Otherwise, $T=\{s,t\}$ and $L_1,L_2$ are simple.
Suppose for a contradiction that for $i=1,2$, $B_i=\delta(U_i)$ where $U_i\subseteq V(G)-\{t\}$.
Let $B=\delta(U_1\cap U_2)\subseteq B_1\cup B_2$.
By proposition~\ref{packmate} $\{\Omega\}=L_2\cap B_1=L_2\cap\delta(U_1)$.
Since $L_2$ is simple and since $U_1\cap U_2\subset U_1$, $\delta(U_1\cap U_2)\cap L_2=\{\Omega\}$, 
it follows that $L_2\cap B=\{\Omega\}$ (recall $\omega\in \delta(s)$).
Similarly, we have $L_1\cap B=\{\Omega\}$, contradicting proposition~\ref{usefulparity} part (4).
\end{proof}
\begin{prp}\label{matessimplenonsimple}
Let $(G,\Sigma,T)$ be a signed graft with $T=\{s,t\}$ and let $(L_1,\ldots,L_k)$ be an $(\Omega, k)$-packing, where $L_1,L_2,L_3$ are minimal odd $T$-joins.
Suppose that $L_1$ is non-simple and that $L_2, L_3$ are simple.
Suppose that for $i=2,3$ there exists a $k$-mate $B_i$ of $L_i$.
Then $\Omega\in C(L_1)$.
\end{prp}
\begin{proof}
By proposition~\ref{matescutcut} one of $B_2,B_3$ is a signature, say $B_2$.
Thus $B_2\cap C(L_1)\neq\emptyset$. But proposition~\ref{packmate} implies that $B_2\cap L_1=\{\Omega\}$ and the result follows.
\end{proof}
\begin{prp}\label{nonsimplestcutsign}
Let $(G,\Sigma,T)$ be a signed graft with $|T|\leq 2$ and let $(L_1,\ldots,L_k)$ be an $(\Omega, k)$-packing.
Suppose that $L_2$ is a non-simple minimal odd $T$-join and that there exists a $k$-mate $B_1$ of $L_1$. Then, 
\begin{enumerate}[\;\;(1)]
\item if $\Omega\in P(L_2)$ then $B_1$ is a $T$-cut,
\item if $\Omega\in C(L_2)$ then $B_1$ is a signature.
\end{enumerate}
\end{prp}
\begin{proof}
{\bf (1)}
By proposition~\ref{packmate}, $B_1\cap L_2=\{\Omega\}$.
Since $\Omega\in P(L_2)$, $B_1\cap C(L_2)=\emptyset$.
Since $C(L_2)$ is an odd circuit, $B_1$ is not a signature.
It follows from the definition of $k$-mate that $B_1$ is a $T$-cut.
{\bf (2)}
Proceeding as above we have $B_1\cap P(L_2)=\emptyset$.
If $T=\emptyset$, then we are done.
Otherwise, $T=\{s,t\}$ and $P(L_2)$ is an $st$-path, so $B_1$ is not an $st$-cut.
It follows that $B_1$ is a signature.
\end{proof}

\section{Non-bipartite, non-simple and simple bipartite $\Omega$-systems}\label{sec-classification}
%
In this section, we prove proposition~\ref{minimal-to-omega}, stating that every extremal counterexample has a non-bipartite, non-simple bipartite, or simple bipartite $\Omega$-system.

\begin{proof}[Proof of proposition~\ref{minimal-to-omega}]
Let $(G,\Sigma,T)$ be an extremal counterexample with $\tau:=\tau(G,\Sigma,T)$.
By proposition~\ref{BG-minimality} parts (2) and (4) there exists an $(\Omega, \tau)$-packing ${\mathcal L}=(L_1,\ldots,L_{\tau})$ of odd $T$-joins.
By proposition~\ref{BG-minimality} part (5) every odd $T$-join has a $\tau$-mate.
If $(L_1\cup L_2\cup L_3)-\{\Omega\}$ is non-bipartite, then $\bigl((G,\Sigma,T),{\mathcal L}\bigr)$ is a non-bipartite $\Omega$-system.
Otherwise, $(L_1\cup L_2\cup L_3)-\{\Omega\}$ is bipartite.
We will show that $(G,\Sigma,T)$ has a non-simple bipartite or simple bipartite $\Omega$-system.

We can rearrange the elements of the sequence $\mathcal{L}$ to ensure (B2) is satisfied for some $m\in [\tau]-[2]$.
For each $i\in [3]$, let $B_i$ be a $\tau$-mate of $L_i$. 
Since $(L_1\cup L_2\cup L_3)-\{\Omega\}$ is bipartite, it follows that, for each $i\in [3]$, either $L_i$ is simple or $\Omega\in C(L_i)$.
Therefore, by proposition~\ref{matescutcut}, at least two of $B_1,B_2,B_3$, say $B_1$ and $B_2$, are signatures.
By proposition~\ref{packmate}, $B_1$ (resp. $B_2$) is a cap of $L_1$ (resp. $L_2$) in $\mathcal{L}$.
Let $U$ be the subset of $V(L_1)-T$ for which $L_1\cap \delta(U) = (L_1\cap B_1)-\{\Omega\}$, and let $\Gamma:= B_1\triangle \delta(U)$.
It is clear that $\Gamma$ is a signature for $(G,\Sigma,T)$.
We will show that $((G,\Gamma,T),\mathcal{L},m)$ is a bipartite $\Omega$-system. 
It is clear that (B1) and (B2) hold.
To prove (B3), we need to show that, for $i\in [3]$, $\Gamma\cap L_i=\{\Omega\}$, and for $i\in [m]-[3]$, $\Gamma\cap P_i=\emptyset$.
By definition, $\Gamma\cap L_1=\{\Omega\}$.

\begin{claim} 
For $i=2,3$, $B_1\cap P_i=\emptyset$ and $\delta(U)\cap L_i=\emptyset$.
\end{claim}
\begin{cproof}
Since $B_1\cap L_i=\{\Omega\}$ and $\Omega\notin P_i$, it follows that $B_1\cap P_i=\emptyset$.
To prove the next equation, choose vertices $s,s',t$ as follows: $\Omega$ has ends $s,s'$, if $T\neq \emptyset$ then $T=\{s,t\}$, and if $T=\emptyset$ then $t:=s$. Notice that $s,s',t\notin U$ and $Q_i:=L_i-\{\Omega\},Q_1:=L_1-\{\Omega\}$ are $s't$-paths.
Suppose for a contradiction that $\delta(U)\cap L_i\neq \emptyset$.
Then our choice of $U$ implies that $L_i$ and $L_1$ have a vertex $u\in U$ in common.
Consider the cycle $C:=Q_i[u,t]\cup Q_1[u,t]$.\footnote{Given a path $P$ and vertices $a,b\in V(P)$, $P[a,b]$ denotes the subpath between $a$ and $b$.} Since $B_1\cap L_i=\{\Omega\}$ and $(B_1\cap L_1)-\{\Omega\} = \delta(U)\cap L_1$, it follows that $B_1\cap C = \delta(U)\cap Q_1[u,t]$, implying in turn that $|B_1\cap C|$ is odd. As $B_1$ is a signature, it follows that $C\subseteq (L_1\cup L_2)-\{\Omega\}$ is an odd cycle, a contradiction as $(L_1\cup L_i)-\{\Omega\}$ is bipartite.
\end{cproof}

\noindent Thus, for $i=2,3$ $$\Gamma\cap L_i = (B_1\triangle \delta(U))\cap L_i = (B_1\cap L_i) \triangle (\delta(U)\cap L_i)= \{\Omega\}.$$

\begin{claim} 
For $i\in [m]-[3]$, $\delta(U)\cap P_i=\emptyset$. 
\end{claim}
\begin{cproof}
As $B_1,B_2$ are signatures and $|B_1\cap L_i|=|B_2\cap L_i|=1$, it follows that $B_1\cap P_i=B_2\cap P_i=\emptyset$. Hence, $B_2\cap (L_1\cup P_i)=\{\Omega\}$, implying that $(L_1\cup P_i)-\{\Omega\}$ is bipartite.
Suppose for a contradiction that $\delta(U)\cap P_i\neq \emptyset$.
Then $L_1$ and $P_i$ have a vertex of $U$ in common, and so $(L_1\cup P_i)-\{\Omega\}$ is non-bipartite, a contradiction.
\end{cproof}

\noindent Hence, for $i\in [m]-[3]$ $$\Gamma\cap P_i = (B_1\triangle \delta(U))\cap P_i = (B_1\cap P_i) \triangle (\delta(U)\cap P_i)= \emptyset.$$ Therefore, (B3) holds and $((G,\Gamma,T),\mathcal{L},m)$ is a bipartite $\Omega$-system. Among all bipartite $\Omega$-systems whose associated signed graft is $(G,\Gamma,T)$, we may assume that the $(\Omega,\tau)$-packing $\mathcal{L}$ of odd $T$-joins has the smallest total number of edges.

Let $H:=G[L_1\cup L_2\cup L_3\cup P_4\cup \cdots \cup P_m]$. 
Orient the edges of $H$ so that each of $L_1,L_2,L_3$ is a directed $T$-join, and if $T=\{s,t\}$ and $\Omega\in \delta(s)$, each of $P_4,\ldots,P_m$ is an $st$-dipath; call this digraph $\vec{H}$.

\begin{claim} 
$\vec{H}\setminus \Omega$ is acyclic.
\end{claim}
\begin{cproof}
Suppose otherwise.
Let $C$ be a directed circuit in $\vec{H}\setminus \Omega$. We assume that $\Omega=(s,s')$ and that either $T=\emptyset$ or $T=\{s,t\}$. When $T=\emptyset$, set $t:=s$. 
Create $m-3$ copies $\overline{\Omega}_4,\ldots, \overline{\Omega}_m$ of the arc $(s',s)$. For each $i\in [3]$, let $Q_i:=L_i-\{\Omega\}$ and for each $i\in [m]-[3]$, let $Q_i:=\{\overline{\Omega}_i\}\cup P_i$. Notice that $Q_1,\ldots,Q_m$ are pairwise arc-disjoint directed $s't$-joins, and $Q_1,Q_2,Q_3$ are $s't$-dipaths.
We can now decompose $(Q_1\cup \cdots \cup Q_m)- C$ into pairwise arc-disjoint directed $s't$-joins $Q'_1\cup \cdots \cup Q'_m$, where \begin{itemize}
\item $Q'_1, Q'_2, Q'_3$ are $s't$-dipaths, and
\item for $i\in [m]-[3]$, $\overline{\Omega}_i\in Q'_i$. \end{itemize}
For $i\in [3]$, let $L'_i:=Q'_i\cup \{\Omega\}$, and for $i\in [m]-[3]$, let $P'_i$ be an $st$-dipath contained in $Q'_i-\{\overline{\Omega}_i\}$.
Then $L'_1, L'_2, L'_3$ are directed odd $st$-joins and $P'_4 ,\ldots, P'_m$ are even $st$-dipaths in $\vec{H}$.
Let $\mathcal{L'}=(L'_1,L'_2,L'_3,C_4\cup P'_4, \ldots, C_m\cup P'_m, L_{m+1},\ldots,L_\tau)$.
It can now be readily checked that $((G,\Gamma,T), \mathcal{L'}, m)$ is a bipartite $\Omega$-system, a contradiction as $\mathcal{L'}$ has fewer edges than $\mathcal{L}$.~\end{cproof}

It is now easily seen that $((G,\Gamma,T),\mathcal{L},m,\vec{H})$ is either a non-simple bipartite or simple bipartite $\Omega$-system, finishing the proof.
\end{proof}

   
\section{Preliminaries for non-bipartite $\Omega$-systems}\label{sec-prelim-nb}
In this section we prove results required for the proofs of propositions~\ref{prp-NF1} and \ref{prp-NF2}.
We also prove proposition~\ref{nbflavours}, namely, that every non-bipartite $\Omega$-system is of flavour (NF1) or (NF2).
\subsection{The two flavours (NF1) and (NF2)}
Let us start with the following:
\begin{prp}\label{2bipartite}
Let $(G,\Sigma)$ be a signed graph whose edges can be partitioned for some distinct vertices $x,y$ into $xy$-paths $Q_1, Q_2, \ldots, Q_n$. If, for every distinct $i,j\in [n]$, $Q_i\cup Q_j$ is bipartite, then $(G,\Sigma)$ is bipartite.
\end{prp}
\begin{proof}
We will proceed by induction on $n$. For $n=1$ this is obvious. Suppose $n>1$. By the induction hypothesis, $Q_1\cup \ldots \cup Q_{n-1}$ is bipartite, and so by theorem~\ref{zas}, there is a signature $\Gamma$ of $(G,\Sigma,\emptyset)$ disjoint from $Q_1\cup \cdots \cup Q_{n-1}$, so $\Gamma\subseteq Q_n$. As $Q_1\cup Q_n$ is an even cycle, it follows that $|\Gamma|$ is even.
Let $U$ be the vertex subset of $V(Q_n)-\{x,y\}$ for which $\delta(U)\cap Q_n=\Gamma\cap Q_n$. We claim that $\delta(U)=\Gamma$, and this will imply that $(G,\Sigma,\emptyset)$, and therefore $(G,\Sigma)$, is bipartite. 

Suppose, for a contradiction, that $\Gamma\subsetneq \delta(U)$. Take an edge $\{v,u\}\in \delta(U)-\Gamma$ with $u\in U$. Then $\{v,u\}$ belongs to some $Q_j\in \{Q_1,\ldots,Q_{n-1}\}$. We may assume that $\{v,u\}\in Q_j[x, u]$. Let $C=Q_1[x,u]\cup Q_j[x,u]$. Then $|C\cap \Gamma|=|Q_1[x,u]\cap \delta(U)|$, which is odd as $x\notin U$ and $u\in U$. Hence, $C$ is an odd cycle, but $C\subseteq Q_1\cup Q_j$, which is a contradiction. Therefore, $\Gamma=\delta(U)$, and this completes the proof.
\end{proof}
Next we prove that every non-bipartite $\Omega$-system is of flavour (NF1) or (NF2).
\begin{proof}[Proof of proposition~\ref{nbflavours}]
Let $((G,\Sigma,T), (L_1,\ldots,L_k))$ be a non-bipartite $\Omega$-system that is not of flavour (NF2). 
We will show (NF1) holds.

Proposition~\ref{matessimplenonsimple} implies that at least two of $L_1,L_2,L_3$ are non-simple.
It remains to show that $\Omega\in P(L_1)\cap P(L_2)\cap P(L_3)$. Suppose otherwise. Then, for some $i\in [3]$, $L_i$ is non-simple and $\Omega\in C(L_i)$. 
By proposition~\ref{nonsimplestcutsign}, $B_1,B_2,B_3$ are signatures, and whenever $L_i\in \{L_1,L_2,L_3\}$ is non-simple, $\Omega\in C(L_i)$.

For each $j\in [3]$, let $Q_j=L_j-\{\Omega\}$. 
Suppose $s,s'$ are the ends of $\Omega$.
When $T=\emptyset$, $Q_1,Q_2$ and $Q_3$ are $s's$-paths, and
when $T=\{s,t\}$, $Q_1,Q_2$ and $Q_3$ are all $s't$-paths.
Moreover, for every permutation $i,j,k$ of $1,2,3$, $(Q_i\cup Q_j)\cap B_k=\emptyset$, implying that $Q_i\cup Q_j$ is bipartite. Therefore, from proposition~\ref{2bipartite} we conclude that $Q_1\cup Q_2\cup Q_3=(L_1\cup L_2\cup L_3)-\{\Omega\}$ is bipartite, which is a contradiction.
\end{proof}
\subsection{A disentangling lemma}
\begin{lma}\label{nbdisentangle}
Let $((G,\Sigma,\{s,t\}), (L_1,\ldots,L_k))$ be a minimal non-bipartite $\Omega$-system. For $i=1,2$, let $R_i\cup Q_i$ be a non-trivial partition of $L_i$ such that $\Omega\in Q_1\cap Q_2$, $R_1\cup Q_2$ is a minimal odd $st$-join and $R_1\cup R_2$ is an even cycle. Let $Q_3$ be a minimal subset of $L_3$ such that $Q_3\cup R_1$ contains a minimal odd $st$-join. Then one of the following does not hold: \begin{enumerate}[\;\;(i)]

\item $(L_1\cup Q_2\cup Q_3)-\{\Omega\}$ is non-bipartite,

\item $R_2\cup \{\Omega\}$ does not have a $k$-mate,

\item $R_1$ is a path whose internal vertices all have degree two in $G[L_1\cup Q_2\cup Q_3]$.
\end{enumerate}
\end{lma}
\begin{proof}
Suppose otherwise. We will show that $((G,\Sigma,\{s,t\}), (L_1,\ldots,L_k))$ is not a minimal non-bipartite $\Omega$-system, which will yield a contradiction.
Let $(G',\Sigma'):=(G,\Sigma)\setminus R_2/R_1$ and define $L'_1,\ldots,L'_k$ as follows: for $i\in [3]$ $L'_i:=Q_i$, and for $i\in \{4,\ldots,k\}$ $L'_i$ is a minimal odd $st$-join of $(G',\Sigma')$ contained in $L_i$. We claim that $((G',\Sigma',\{s,t\}), (L'_1,\ldots,L'_k))$ is a non-bipartite $\Omega$-system.

{\bf (N1)}
Since $R_1\cup R_2$ is an even cycle, every minimal cover of $(G,\Sigma,\{s,t\})$ disjoint from $R_1$ has an even number of edges in common with $R_2$. Hence, $(G',\Sigma',\{s,t\})$ is Eulerian and $\tau(G',\Sigma',\{s,t\}),$ $\tau(G,\Sigma,\{s,t\})$ have the same parity.
{\bf (N3)}
Observe that (i) implies $(L'_1\cup L'_2\cup L'_3)-\{\Omega\}$ is non-bipartite.
{\bf (N4)}
Let $L'\subseteq L'_1\cup L'_2\cup L'_3$ be a minimal odd $st$-join of $(G',\Sigma',\{s,t\})$. By (iii) one of $L', L'\cup R_1$ is a minimal odd $st$-join of $(G,\Sigma,\{s,t\})$. 
In the former case, let $B'$ be a $k$-mate of $L'$ in $(G,\Sigma,\{s,t\})$. By definition, $|B'-L'|\leq k-3$ and so $B'-L'\subseteq L_4\cup \cdots \cup L_k$, implying that $B'\cap R_1=\emptyset$. Thus $B'$ is still a $k$-mate for $L'$ in $(G',\Sigma',\{s,t\})$.
In the latter case, when $L'\cup R_1$ is a minimal odd $st$-join of $(G,\Sigma,\{s,t\})$, $L'\cup R_2$ also contains a minimal odd $st$-join $L$. Let $B$ be a $k$-mate of $L$ in $(G,\Sigma,\{s,t\})$. Once again, $|B-L|\leq k-3$ and so $B-L\subseteq L_4\cup \cdots\cup L_k$, implying that $B\cap R_1=\emptyset$. As a result, $B-R_2$ is a $k$-mate for $L'$ in $(G',\Sigma',\{s,t\})$.
{\bf (N2)}
As $\tau(G',\Sigma',\{s,t\}), \tau(G,\Sigma,\{s,t\})$ have the same parity, $\tau(G',\Sigma',\{s,t\}), k$ have the same parity.
We need to show $\Omega\in L'_3$ and $\tau(G',\Sigma',\{s,t\})\geq k$. 
By (N4) $L'_1$ has a $k$-mate $B'$ in $(G',\Sigma',\{s,t\})$. Then $|B'-L'_1|\leq k-3$ and so $B'-L'_1\subseteq L'_4\cup \cdots \cup L'_k$. Since $B'\cap L'_3\neq \emptyset$, $B'\cap L'_3=\{\Omega\}$, and so $\Omega\in L'_3$.
Suppose for a contradiction that $\tau(G',\Sigma',\{s,t\})<k$. The parity condition implies that $\tau(G',\Sigma',\{s,t\})\leq k-2$. Let $B'$ be a minimum cover in $(G',\Sigma',\{s,t\})$. For $|B'|\leq k-2$ and $L'_1,L'_4,\ldots,L'_k$ are $k-2$ pairwise disjoint odd $st$-joins, we have $|B'|= k-2$, and as $B'\cap L'_2\neq \emptyset$, $\Omega\in B'$. 
Let $B$ be a minimal cover of $(G,\Sigma,\{s,t\})$ contained in $B'\cup R_2$ and containing $B'$.
By proposition~\ref{coverchar}, $B$ is either a signature or an $st$-cut. However, $|B-(R_2\cup \{\Omega\})|=|B'-\{\Omega\}|=k-3,$ implying that $B$ is a $k$-mate of $R_2\cup \{\Omega\}$ in $(G,\Sigma,\{s,t\})$, contradicting (ii). 
\end{proof}
\subsection{Mates and connectivity}
Recall that if $(G,\Sigma,\{s,t\})$ is a signed graft with signatures $\Sigma_1,\Sigma_2$ then by definition $\Sigma_1\triangle\Sigma_2$ is a cut where both $s,t$ are on the same shore. We will require the following easy remark,
\begin{rem}\label{trivialc}
Let $G$ be a graph with distinct vertices $s,t$.
For $i=1,2$ let $W_i\subseteq V(G)-\{t\}$ where $s\in W_1\subseteq W_2$.
Let $P$ be an $st$-path and let $\Omega$ be the edge of $P$ incident to $s$.
If $P\cap\delta(W_2)=\{\Omega\}$ then $P\cap\delta(W_1)=\{\Omega\}$.
\end{rem}
\begin{prp}\label{matescutsign}
Let $(G,\Sigma,\{s,t\})$ be a signed graft and $(L_1,\ldots,L_k)$ be an $(\Omega, k)$-packing, where $L_2$ is an odd $st$-path.
Suppose there exist an $st$-cut $B_1$ that is a $k$-mate of $L_1$ and a signature $B_2$ that is a $k$-mate of $L_2$.
Choose $U_1\subseteq V(G)-\{t\}$ such that $B_1=\delta(U_1)$ and let $W=(V(L_1)\cap U_1)-\{s\}$.
Then there exists a path in $G[U_1]$ between $s$ and $W$ that is disjoint from $B_2$.
\end{prp}
\begin{proof}
Suppose for a contradiction there is no such path.
Then there exists $U'\subset U_1$ such that $s\in U'$ and $W\subseteq U_1-U'$ and 
all edges with one end in $U'$ and one end in $U_1-U'$ are in $B_2$.
Then the $st$-cut $B=\delta(U')\subseteq B_1\cup B_2$ and by construction $L_1\cap B=\{\Omega\}$.
By proposition~\ref{packmate} $L_2\cap B_1=L_2\cap \delta(U_1)=\{\Omega\}$.
Since $L_2$ is an odd $st$-path, and since $U'\subset U_1$ by remark~\ref{trivialc}, $\delta(U')\cap L_2=\{\Omega\}$.
But then $|B\cap L_1|=|B\cap L_2|=1$, contradicting proposition~\ref{usefulparity} part~(4).
\end{proof}

\section{Non-bipartite $\Omega$-system of flavour (NF1)}\label{sec-NF1}

In this section we prove proposition~\ref{prp-NF1}, namely that a minimal non-bipartite $\Omega$-system of flavour (NF1) has an $F_7$ minor. 
For convenience, whenever $L_i$ is non-simple, we write $P_i:=P(L_i)$ and $C_i:=C(L_i)$.
Let $(G,\Sigma,\{s,t\})$ be a signed graft and let $\delta(U)$ be an $st$-cut that is a $k$-mate of a minimal odd $st$-join $L$.
We say that $U\subseteq V(G)-\{t\}$ is {\em shore-wise minimal} if among all $k$-mates of $L$ of the form $\delta(U')$ where $U'\subseteq V(G)-\{t\}$, $U'$ is not a proper subset of $U$.
\begin{prp}\label{NF1mates}
Let $((G,\Sigma,\{s,t\}), \mathcal L=(L_1,\ldots,L_k))$ be a non-bipartite $\Omega$-system of flavour (NF1), where $\Omega\in \delta(s)$. Then, \begin{enumerate}
\item[\;\;(1)]
for $i\in [3]$, every $k$-mate of $L_i$ is an $st$-cut.\end{enumerate} Furthermore, for $i\in [3]$, let $\delta(U_i)$ be a $k$-mate of $L_i$ where $U_i$ is shore-wise minimal. Then \begin{enumerate}
\item[\;\;(2)]  
for $i\in [3]$, if $L_i$ is non-simple, then $P_i\cap \delta(U_i)=\{\Omega\}$ and $C_i\cap \delta(U_i)\neq \emptyset$,

\item[\;\;(3)]  
for distinct $i,j\in [3]$, if $L_i,L_j$ are non-simple, then $U_i\subset U_j$ or $U_j\subset U_i$,

\item[\;\;(4)] 
for distinct $i,j\in [3]$, if $L_i$ is non-simple and $L_j$ is simple, then $U_i\subset U_j$.
\end{enumerate}
\end{prp}
\begin{proof} 
{\bf (1)}
Let $i\in [3]$ and let $B$ be a $k$-mate of $L_i$. By (NF1) one of $\{L_1,L_2,L_3\}-\{L_i\}$, say $L_j$, is non-simple and $\Omega\in P_j$. Proposition~\ref{nonsimplestcutsign} then implies that $B$ is an $st$-cut. 

Now for $i\in [3]$, let $B_i=\delta(U_i)$ be a $k$-mate of $L_i$ where $U_i\subseteq V(G)-\{t\}$ is shore-wise minimal. We need to prove (2)-(4). We may assume $L_1$ and $L_2$ are non-simple. 
By proposition~\ref{packmate}, for $i\in [3]$, $B_i$ is a cap of $L_i$ in $\mathcal L$.

{\bf (2)}
We may assume $i=1$. Consider the $(\Omega,k)$-packing 
\[
\mathcal{L'}=(C_1\cup P_2,C_2\cup P_1,L_3,\ldots,L_k).
\]
As $((G,\Sigma,\{s,t\}), \mathcal L)$ is a non-bipartite $\Omega$-system, $C_1\cup P_2$ has a $k$-mate $B'_1$.
By proposition~\ref{packmate}, $B'_1$ is a cap of $C_1\cup P_2$ in $\mathcal{L}'$, implying that $B'_1\cap (C_2\cup P_1)=\{\Omega\}$ and so $B'_1\cap C_2=\emptyset$.
Thus, by proposition~\ref{nonsimplestcutsign}, $B'_1$ is an $st$-cut $\delta(U)$ where $U\subseteq V(G)-\{t\}$.
Consider the $st$-cut $B=\delta(U_1\cap U)\subseteq B_1\cup B'_1$.
Since $B_1$ is a cap of $L_1$ in $\mathcal L$,
and $B'_1$ is a cap of $C_1\cup P_2$ in $\mathcal L'$,
$P_2\cap\delta(U_1)=P_1\cap \delta(U)=\{\Omega\}$.
Thus $B\cap P_2=\delta(U_1\cap U)\cap P_2=\{\Omega\}$ 
and $B\cap P_1=\delta(U_1\cap U)\cap P_1=\{\Omega\}$
(see remark~\ref{trivialc}).
It follows by proposition~\ref{usefulparity3} that $B$ is a $k$-mate of $L_1\cap (C_1\cup P_2)=C_1\cup \{\Omega\}$.
In particular, $B$ is a $k$-mate of $L_1$.
Since $U_1$ is shore-wise minimal, $U_1\subseteq U$. 
Hence, as $P_1\cap \delta(U)=\{\Omega\}$, we have $P_1\cap\delta(U_1)=\{\Omega\}$. Also, since $B_1$ is a cap of $L_1$ in $\mathcal L$, $C_1\cap \delta(U_1)\neq \emptyset$.

{\bf (3)}
Since $\delta(U_i),\delta(U_j)$ are, respectively, caps of $L_i,L_j$ in $\mathcal L$,
\[
\delta(U_i)\cap C_j=\emptyset
\quad\mbox{and}\quad
\delta(U_j)\cap C_i=\emptyset.
\] 
Thus, either $V(C_i)\subseteq U_j$ or $V(C_i)\cap U_j=\emptyset$, and either $V(C_j)\subseteq U_i$ or $V(C_j)\cap U_i=\emptyset$.
By (2), $P_i\cap\delta(U_i)=P_j\cap\delta(U_j)=\{\Omega\}$, and so $\delta(U_i)\cap C_i\neq\emptyset$ and
$\delta(U_j)\cap C_j\neq\emptyset$.
By proposition~\ref{usefulparity}~(4), 
\[
\delta(U_i\cap U_j)\cap (C_i\cup C_j)\neq\emptyset
\quad\mbox{and}\quad
\delta(U_i\cup U_j)\cap (C_i\cup C_j)\neq\emptyset.
\]
It therefore follows that, after possibly interchanging the role of $i,j$, we have that
$V(C_i)\subseteq U_j$ and $V(C_j)\cap U_i=\emptyset$.
But then proposition~\ref{usefulparity}~(5) implies that $\delta(U_i\cap U_j)$ is a $k$-mate of $L_i$.
Hence, as $U_i$ is shore-wise minimal, $U_i\subset U_j$ as required.

{\bf (4)}
Since $\delta(U_i)$ is a cap of $L_i$ in $\mathcal L$, $\delta(U_i)\cap L_j=\{\Omega\}$, and as $L_j$ is simple, $L_j\cap \delta(U_i\cap U_j)=\{\Omega\}$ (see remark~\ref{trivialc}).
Therefore, by proposition~\ref{usefulparity}~(5), $\delta(U_i\cap U_j)$ is a $k$-mate of $L_i$. Since $U_i$ is shore-wise minimal, $U_i\subset U_j$ as required.
\end{proof}

\begin{lma}\label{NF1intersection}
Let $((G,\Sigma,\{s,t\}), \mathcal L=(L_1,\ldots,L_k))$ be a minimal non-bipartite $\Omega$-system of flavour (NF1), where $\Omega\in \delta(s)$ and among all non-bipartite $\Omega$-systems with the same associated signed graft, the number of non-simple minimal odd $st$-joins among $L_1,L_2,L_3$ is maximum. Suppose, for $i\in [3]$, $B_i=\delta(U_i)$ is a $k$-mate of $L_i$ where $U_i$ is shore-wise minimal and where $U_1\subset U_2\subset U_3$. Then the following hold: \begin{enumerate}[\;\;(1)]

\item For distinct $i,j\in [3]$, if $L_i$ and $L_j$ are non-simple, then $C_i$ and $C_j$ have at most one vertex in common.

\item For distinct $i,j\in [3]$, if $L_i$ is non-simple and $L_j$ is simple, then $C_i$ and $L_j$ have at most one vertex in common.

\item Suppose $L_3$ is simple. If $L$ is a minimal odd $st$-join contained in $P_2\cup L_3$, then $L\cap \delta(U_3)=L_3\cap \delta(U_3)$.

\item Let $L_0$ be the path with a single vertex $s$ and let $U_0:=\emptyset$. For some $j\in [3]$, take $v\in V(L_j)\cap (U_j- U_{j-1})$. Let $U$ be the component of $G[U_j- U_{j-1}]$ containing $v$. Then $V(L_{j-1})\cap U\neq \emptyset$. 

\item Suppose $L_3$ is non-simple. Then there is a path in $G[\overline{U_3}]$ between $V(C_3)$ and $t$, where $\overline{U_3}=V(G)-U_3$.
\end{enumerate}
\end{lma}
\begin{proof} Observe that $L_1$ and $L_2$ are non-simple. By proposition~\ref{NF1mates}, for each $i\in [3]$, if $L_i$ is non-simple then $P_i\cap \delta(U_i)=\{\Omega\}$ and $C_i\cap \delta(U_i)\neq \emptyset$. Thus $V(C_1)\subseteq U_2$, $V(C_2)\subseteq U_3-U_1$, and if $L_3$ is non-simple, $V(C_3)\cap U_2=\emptyset$. Moreover, for $i=1,2$, $V(P_i)\cap U_3=\{s\}$, and if $L_3$ is non-simple, $V(P_3)\cap U_3=\{s\}$.

{\bf (1)} 
We will first prove that $C_1$ and $C_2$ have at most one vertex in common. Suppose otherwise.
We will obtain a contradiction by proving that $((G,\Sigma,\{s,t\}), \mathcal L)$ is not a minimal non-bipartite $\Omega$-system.

Choose distinct vertices $u,v\in V(C_1)\cap V(C_2)$. Notice that $u,v\in U_2-U_1$. Let $R_1$ be a $uv$-path contained in $C_1$ that avoids vertex $s$. 
Let $R_2$ be the $uv$-path contained in $C_2$ such that $R_1\cup R_2$ is an even cycle (notice that $C_2$ is an odd circuit).
For $i=1,2$, let $Q_i=L_i-R_i$, and let $Q_3=L_3$.
Observe that $V(R_1)\subset V(C_1)\subseteq U_2$, that $R_1$ is internally vertex-disjoint from $C_1-R_1$ as $C_1$ is a circuit, and that $R_1$ is vertex-disjoint from $P_1\cup P_2\cup Q_3$ as $V(P_1)\cap U_2=V(P_2)\cap U_2=V(Q_3)\cap U_2=\{s\}$.
Notice further that $R_1$ is internally vertex-disjoint from $C_2-R_2$. For if not, $C_2\triangle (R_1\cup R_2)$ can be partitioned into non-empty parts $C'_2,X$ where $C'_2$ is an odd circuit and $X$ is an even cycle.
But then $((G,\Sigma,\{s,t\})\setminus X, (L_1\triangle (R_1\cup R_2), C'_2\cup P_2, L_3,\ldots,L_k))$ is another non-bipartite $\Omega$-system, contradicting the minimality of the $\Omega$-system $((G,\Sigma,\{s,t\}), \mathcal L)$.
It therefore follows that the internal vertices of $R_1$ all have degree two in $G[L_1\cup Q_2\cup Q_3]$. 
Observe that $(Q_1\cup Q_2\cup Q_3\cup R_1)-\{\Omega\}$ is non-bipartite as it contains the odd cycle $C_1$. Lemma~\ref{nbdisentangle} therefore implies $R_2\cup \{\Omega\}$ has a $k$-mate $B$.
Observe that $B$ is also a $k$-mate of $L_2$ and of $L_1\triangle (R_1\cup R_2)$, as $R_2\cup \{\Omega\}\subset L_2$ and $R_2\cup \{\Omega\}\subset L_1\triangle (R_1\cup R_2)$. Thus by proposition~\ref{NF1mates} $B$ is an $st$-cut, so $B=\delta(U)$ for some $U\subseteq V(G)-\{t\}$. 
Then $\delta(U_2\cap U)$ is a cover contained in $B_2\cup B$, and so by proposition~\ref{usefulparity2}, it is a $k$-mate of $L_2$. Thus the shore-wise minimality of $U_2$ implies that $U_2\subseteq U$. As $\delta(U)$ is a $k$-mate of $L_1\triangle (R_1\cup R_2)$, it follows that $\delta(U)\cap (L_2\triangle (R_1\cup R_2))=\{\Omega\}$. In particular, $\delta(U)\cap (C_2-R_2)=\emptyset$ and as $u,v\in U_2\subseteq U$, we get that $V(C-R_2)\subseteq U$.

We claim that $s\in V(C_1-R_1)$. For if not, similarly as above, $(C_2-R_2)\cup \{\Omega\}$ also has a $k$-mate $\delta(W)$, where $W\subseteq V(G)-\{t\}$ and $U_2\subseteq W$ and $V(R_2)\subseteq W$. 
Since $\delta(U\cup W)$ is contained in $\delta(U)\cup \delta(W)$, and $\delta(U),\delta(W)$ are $k$-mates for $L_2$, proposition~\ref{usefulparity2} implies that $\delta(U\cup W)$ is also a $k$-mate for $L_2$. Hence, $\delta(U\cup W)\cap C_2\neq \emptyset$ and so $V(C_2)\not\subseteq U\cup W$. However, $V(C_2-R_2)\subseteq U$ and $V(R_2)\subseteq W$, and so $V(C_2)\subseteq U\cup W$, which is not the case.

Hence, $s\in V(C_1-R_1)$. Let $\tilde{C}_1=(C_1-R_1)\cup R_2$ and $\tilde{C}_2=(C_2-R_2)\cup R_1$. Consider the $(\Omega,k)$-packing  
$$\mathcal{\tilde{L}}=(\tilde{L}_1=\tilde{C}_1\cup P_1, \tilde{L}_2=\tilde{C}_2\cup P_2, L_3,\ldots,L_k).$$ 
The minimality of the non-bipartite $\Omega$-system $((G,\Sigma,\{s,t\}), \mathcal{L})$ implies that $\tilde{C}_1$ and $\tilde{C}_2$ are odd circuits, and since $V(\tilde{C}_1\cup \tilde{C}_2)\subseteq U_3$ and $V(P_1\cup P_2)\cap U_3=\{s\}$, it follows that $\mathcal{\tilde{L}}$ is an $(\Omega,k)$-packing.
By proposition~\ref{NF1mates}, for $i=1,2$, there is a $k$-mate $\delta(\tilde{U}_i)$ for $\tilde{L}_i$, where $\tilde{U}_i\subseteq V(G)-\{t\}$ is shore-wise minimal. Since $s\in V(\tilde{C}_1)$ and $u,v\in V(\tilde{C}_1)\cap V(\tilde{C}_2)$, it follows from proposition~\ref{NF1mates} that $\tilde{U}_1\subset \tilde{U}_2 \subset U_3$. Hence, in particular, $V(\tilde{C}_1)\subseteq \tilde{U}_2$ and in turn $V(R_2)\subset \tilde{U}_2$, so $R_2$ is vertex-disjoint from $C_3$. Thus, similarly as above, $R_1\cup \{\Omega\}$ has a $k$-mate $\delta(U'), U'\subseteq V(G)-\{t\}$.

Note that $\delta(U)$ is a $k$-mate of $\tilde{L}_1$ and $\delta(U')$ is a $k$-mate of $\tilde{L}_2$. Since $s\in V(C_1-R_1)$ and $(C_1-R_1)\cap \delta(U)=(C_1-R_1)\cap \delta(U')=\emptyset$, we have $V(C_1-R_1)\subseteq U\cap U'$ and in particular, $u,v\in U\cap U'$. Consider $\delta(U\cup U')$ which is contained in $\delta(U)\cup \delta(U')$. Since $R_1\cap \delta(U)=\emptyset$, it follows that $R_1\cap \delta(U\cup U')=\emptyset$, and so by proposition~\ref{usefulparity}, $\delta(U\cup U')$ is a $k$-mate of $L_2$, implying that $R_2\cap \delta(U\cup U') \neq \emptyset$, a contradiction as $R_2\cap \delta(U')=\emptyset$. Hence, $C_1$ and $C_2$ have at most one vertex in common.

Suppose now that $L_3$ is non-simple. 
Notice first that by proposition~\ref{NF1mates}~(2), $P_3\cap \delta(U_3)=\{\Omega\}$, so $V(P_3)\cap U_3=\{s\}$.
Since $V(C_1)\subseteq U_2$ and $V(C_3)\cap U_2=\emptyset$, it follows that $C_1$ and $C_3$ are vertex-disjoint. It remains to show that $C_2$ and $C_3$ have at most one vertex in common. 
Suppose otherwise.
We will once again obtain a contradiction by proving that $((G,\Sigma,\{s,t\}), \mathcal L)$ is not a minimal non-bipartite $\Omega$-system.
As we just showed, $C_1$ and $C_2$ have at most one vertex in common. 
Choose distinct vertices $u,v\in V(C_2)\cap V(C_3)$ and let $R_2$ be a $uv$-path contained in $C_2$ that is vertex-disjoint from $C_1$. 
Let $R_3$ be the $uv$-path contained in $C_3$ such that $R_2\cup R_3$ is an even cycle.
As before, the minimality of the $\Omega$-system implies that the internal vertices of $R_2$ all have degree two in $G[L_1\cup L_2\cup (L_3-R_3)]$ (recall that $V(P_3)\cap U_3=\{s\}$).
Lemma~\ref{nbdisentangle} therefore implies $R_3\cup \{\Omega\}$ has a $k$-mate $B$.
As $B$ is also a $k$-mate of $L_3$, proposition~\ref{NF1mates} implies that $B$ is an $st$-cut, so $B=\delta(U)$ for some $U\subseteq V(G)-\{t\}$. 
Then $\delta(U_3\cap U)$ is a cover contained in $B_3\cup B$, and so by proposition~\ref{usefulparity2}, it is $k$-mate of $L_3$. Thus the shore-wise minimality of $U_3$ implies that $U_3\subseteq U$.

We claim $C_2-R_2$ has a vertex in common with $C_1$. For if not, similarly as above, $(C_3-R_3)\cup \{\Omega\}$ also has a $k$-mate $\delta(W)$, where $W\subseteq V(G)-\{t\}$ and $U_3\subseteq W$. Since $\delta(U\cup W)$ is contained in $\delta(U)\cup \delta(W)$, and $\delta(U),\delta(W)$ are $k$-mates for $L_3$, proposition~\ref{usefulparity2} implies that $\delta(U\cup W)$ is also a $k$-mate for $L_3$. Hence, $\delta(U\cup W)\cap C_3\neq \emptyset$ and so $V(C_3)\not\subseteq U\cup W$. However, $u,v\in U_3\subseteq U\cup W$, forcing $V(R_3)\subseteq W$ and $V(C_3-R_3)\subseteq U$, and so $V(C_3)\subseteq U\cup W$, which is not the case.

Hence, $C_2-R_2$ has a vertex in common with $C_1$. Let $\tilde{C}_2=(C_2-R_2)\cup R_3$ and $\tilde{C}_3=(C_3-R_3)\cup R_2$. Consider the $(\Omega,k)$-packing 
$$\mathcal{\tilde{L}}=(L_1,\tilde{L}_2=\tilde{C}_2\cup P_2, \tilde{L}_3=\tilde{C}_3\cup P_3, L_4,\ldots,L_k).$$ 
The minimality of the non-bipartite $\Omega$-system $((G,\Sigma,\{s,t\}), \mathcal{L})$ implies that $\mathcal{\tilde{L}}$ is an $(\Omega,k)$-packing.
By proposition~\ref{NF1mates}, for $i=2,3$, there is a $k$-mate $\delta(\tilde{U}_i)$ for $\tilde{L}_i$, where $\tilde{U}_i\subseteq V(G)-\{t\}$ is shore-wise minimal. Since $\tilde{C}_2$ has vertices in common with the both of $C_1,\tilde{C}_3$, it follows from proposition~\ref{NF1mates} that either $U_1\subset \tilde{U}_2 \subset \tilde{U}_3$ or $\tilde{U}_3\subset \tilde{U}_2\subset U_1$. Hence, in particular, $V(R_3)\subset U_1\cup \tilde{U}_3$ and so the internal vertices of $R_3$ have degree two in $G[L_1\cup (L_2-R_2)\cup L_3]$. Thus, similarly as above, $R_2\cup \{\Omega\}$ has a $k$-mate $\delta(U'), U'\subseteq V(G)-\{t\}$. Note $\delta(U_2\cap U')$ is a cover contained in $\delta(U_2)\cup \delta(U')$, and so by proposition~\ref{usefulparity2}, it is a $k$-mate of $L_2$. Thus the shore-wise minimality of $U_2$ implies that $U_2\subseteq U'$.

Note that $\delta(U)$ is a $k$-mate of $L_3$ and $\delta(U')$ is a $k$-mate of $L_2$. Since $C_2-R_2$ has a vertex $x$ in common with $C_1$, $(C_2-R_2)\cap \delta(U)=(C_2-R_2)\cap \delta(U')=\emptyset$, and $x\in U_2\subset U\cap U'$, we must have $V(C_2-R_2)\subseteq U\cap U'$ and in particular, $u,v\in U\cap U'$. Consider $\delta(U\cup U')$ which is contained in $\delta(U)\cup \delta(U')$. Since $R_2\cap \delta(U')=\emptyset$, it follows that $R_2\cap \delta(U\cup U')=\emptyset$, and so by proposition~\ref{usefulparity}, $\delta(U\cup U')$ is a $k$-mate of $L_3$, implying that $R_3\cap \delta(U\cup U') \neq \emptyset$, a contradiction as $R_3\cap \delta(U)=\emptyset$. Hence, $C_2$ and $C_3$ have at most one vertex in common, thereby finishing the proof.

{\bf (2)}
Suppose that $L_3$ is simple. It is clear that $C_1$ and $L_3$ have at most one vertex (in particular, $s$) in common. We will show that $C_2$ and $L_3$ have at most one vertex in common. Suppose otherwise. Choose distinct $u,v\in V(C_2)\cap V(L_3)$, and let $R_3$ be the $uv$-path contained in $L_3$. Let $R_2$ be the $uv$-path contained in $C_2$ such that $R_2\cup R_3$ is an even cycle.

We claim that $R_2$ is vertex-disjoint from $C_1$. Let $\tilde{C}_2:=(C_2-R_2)\cup R_3$ and $\tilde{L}_3:=(L_3-R_3)\cup R_2$. The minimality of our non-bipartite $\Omega$-system implies $\tilde{L}_3$ is still simple. Consider the $(\Omega,k)$-packing 
$$\mathcal{\tilde{L}}:=(L_1, \tilde{L}_2=\tilde{C}_2\cup P_2, \tilde{L}_3, L_4,\ldots,L_k).$$ 
The minimality of the non-bipartite $\Omega$-system $((G,\Sigma,\{s,t\}), \mathcal{L})$ implies that $\mathcal{\tilde{L}}$ is an $(\Omega,k)$-packing.
By proposition~\ref{NF1mates}, for $i=2,3$, there exists a $k$-mate $\delta(\tilde{U}_i)$ of $\tilde{L}_i$, where $\tilde{U}_i\subseteq V(G)-\{t\}$ is shore-wise minimal, and $U_1\subset \tilde{U}_2\subset \tilde{U}_3$. In particular, $V(R_2)\cap \tilde{U}_2=\emptyset$ and $V(C_1)\subseteq \tilde{U}_2$, so $R_2$ is vertex-disjoint from $C_1$.

As a result, the internal vertices of $R_2$ all have degree two in $G[L_1\cup L_2\cup (L_3-R_3)]$. Thus lemma~\ref{nbdisentangle} implies that $R_3\cup \{\Omega\}$ has a $k$-mate $B$. As $B$ is also a $k$-mate of $L_3$, proposition~\ref{NF1mates} implies that $B=\delta(U)$ for some $U\subseteq V(G)-\{t\}$. However, since $\delta(U)\cap C_2=\emptyset$ and $u,v\notin U$, it follows that $V(C_2)\cap U=\emptyset$. Consider $\delta(U_2\cap U)$, which is contained in $\delta(U_2)\cup \delta(U)$. Since $\delta(U_2)$ is a $k$-mate of $L_2$, $\delta(U)$ is a $k$-mate of $L_3$, and $C_2\cap \delta(U_2\cap U)=\emptyset$, it follows from proposition~\ref{usefulparity} that $\delta(U_2\cap U)\cap L_3\neq \emptyset$, implying in turn that $\delta(U_2)\cap L_3\neq \{\Omega\}$, a contradiction. Thus, $C_2$ and $L_3$ have at most one vertex in common.

{\bf (3)}
Among all non-bipartite $\Omega$-systems with the same associated signed graft, the number of non-simple minimal odd $st$-joins among $L_1,L_2,L_3$ is maximum. Hence, $L$ must be a simple minimal odd $st$-join, and the minimality of the $\Omega$-system implies that $P:=L\bigtriangleup P_2\bigtriangleup L_3$ is an even $st$-path.
Consider the $(\Omega,k)$-packing 
$$\mathcal{\tilde{L}}=(L_1, \tilde{L}_2:=C_2\cup P, \tilde{L}_3:=L, L_4,\ldots, L_k).$$
The minimality of the non-bipartite $\Omega$-system $((G,\Sigma,\{s,t\}), \mathcal{L})$ implies that $\mathcal{\tilde{L}}$ is an $(\Omega,k)$-packing.
By proposition~\ref{NF1mates}, for $i=2,3$, there exists a $k$-mate $\delta(\tilde{U}_i)$ of $\tilde{L}_i$ where $\tilde{U}_i$ is shore-wise minimal, and $\tilde{U}_2\subset \tilde{U}_3$.
We claim that $U_3=\tilde{U}_3$, thereby finishing the proof of (3).  
Let $B:=\delta(U_3\cap \tilde{U}_3)$. Since $L_3, \tilde{L}_3$ are simple, $\delta(U_3)\cap (\tilde{L}_3-L_3)=\emptyset$ and $\delta(\tilde{U}_3)\cap (L_3-\tilde{L}_3)=\emptyset$, it follows that $B\cap (\tilde{L}_3-L_3)=B\cap (L_3-\tilde{L}_3)=\emptyset$. Therefore, proposition~\ref{usefulparity3} implies that $B$ is a $k$-mate for the both of $L_3$ and $\tilde{L}_3$, and so the shore-wise minimality of $U_3,\tilde{U}_3$ implies that $U_3\subset U_3\cap \tilde{U}_3$ and $\tilde{U}_3\subset U_3\cap \tilde{U}_3$. Hence, $U_3=\tilde{U}_3$, as claimed.

{\bf (4)}
Suppose otherwise. 
Assume first that $j=1$.
Observe that $\delta(U)\subseteq \delta(U_1)$. Since $\delta(U_1-U)=\delta(U_1)\triangle \delta(U)$, it follows that $\delta(U_1-U)\subseteq \delta(U_1)$, implying in turn that $\delta(U_1-U)$ is also a $k$-mate of $L_1$, contradicting the shore-wise minimality of $U_1$.
Assume next that $j\neq 1$.
Observe that $\delta(U)\subseteq \delta(U_{j-1})\cup \delta(U_j)$ and $\delta(U)\cap L_{j-1}=\emptyset$.
However, since $\delta(U_j-U)=\delta(U_j)\triangle \delta(U)$ and $\delta(U_j)\cap L_{j-1}=\{\Omega\}$,
$$\delta(U_j-U)\subseteq \delta(U_{j-1})\cup \delta(U_j) \quad \mbox{and}\quad \delta(U_j-U)\cap L_{j-1}=\{\Omega\}.$$
Hence, proposition~\ref{usefulparity} implies that $\delta(U_j-U)$ is a $k$-mate of $L_j$, contradicting the shore-wise minimality of $U_j$.

{\bf (5)}
By proposition~\ref{NF1mates}~(2), $P_3\cap \delta(U_3)=\{\Omega\}$, so $V(P_3)\cap U_3=\{s\}$.
Suppose for a contradiction that (5) does not hold. 
Then there is a subset $U\subset \overline{U_3}$ containing $t$ such that $U\cap V(C_3)=\emptyset$, and such that there is no edge of $G[\overline{U_3}]$ with one end in $U$ and one end not in $U$.
Let $\overline{U}=V(G)-U$. Then $\delta(\overline{U})\subset \delta(U_3)$ and so $|\delta(\overline{U})-L_3|\leq k-3$. However, $\delta(\overline{U})\cap L_3=\{\Omega\}$, and so $|\delta(\overline{U})|\leq k-2$, a contradiction as $k\leq \tau(G,\Sigma)$.
\end{proof}

We are now ready to prove proposition~\ref{prp-NF1}.

\begin{proof}[Proof of proposition~\ref{prp-NF1}]
Let $((G,\Sigma,\{s,t\}),\mathcal{L}=(L_1,\ldots,L_k))$ be a minimal non-bipartite $\Omega$-system of flavour (NF1), where $\Omega$ has ends $s,s'$.
Recall that at least two, say $L_1$ and $L_2$, of $L_1,L_2,L_3$ are non-simple, and $\Omega\in P_1\cap P_2\cap P_3$.

By proposition~\ref{NF1mates}, for each $i\in [3]$, there exists a $k$-mate $B_i=\delta(U_i)$ where $U_i\subseteq V(G)-\{t\}$ is shore-wise minimal, and we may assume $U_1\subset U_2\subset U_3$. Moreover, for $i\in [3]$, if $L_i$ is non-simple then $B_i\cap P_i=\{\Omega\}$ and $B_i\cap C_i\neq \emptyset$. Let $U_0=\emptyset$.

In the first case, assume that $L_3$ is non-simple. Let $U_4:=V(G)$ and let $C_0$ (resp. $C_4$) be the path of single vertex $s$ (resp. $t$). Then by lemma~\ref{NF1intersection}, \begin{enumerate}[\;\;(a)]

\item for $j\in [4]$, there exists a shortest path $Q_j$ in $G[U_j-U_{j-1}]$ between $V(C_{j-1})$ and $V(C_j)$, and

\item for $j\in [2]$, $C_j$ and $C_{j+1}$ have at most one vertex in common.
\end{enumerate} Moreover, let $P'_3$ be the shortest path contained in $P_3$ connecting $s'$ to $V(C_3\cup Q_4)-U_3$. It is now clear that $C_1\cup C_2\cup C_3\cup Q_1\cup Q_2\cup Q_3\cup Q_4\cup P'_3$ has an $F_7$ minor.

In the remaining case, $L_3$ is simple. As above, lemma~\ref{NF1intersection} implies \begin{enumerate}[\;\;(a')]

\item for $j\in [2]$, there exists a shortest path $Q_j$ in $G[U_j-U_{j-1}]$ between $V(C_{j-1})$ and $V(C_j)$, 

\item there exists a shortest path $Q_3$ in $G[U_3-U_2]$ between $V(C_2)$ and $V(L_3)$,

\item $C_1$ and $C_2$ have at most one vertex in common, and $C_2$ and $L_3$ have at most one vertex in common, and

\item if $P_2$ and $L_3$ share a vertex $w$ other than $s,s',t$, then either (a) $V(L_3[s',w])\subseteq V(G)-U_3$ and $L_3[s',w]\cup P_2[s',w]$ is an even cycle, or (b) $V(L_3[w,t])\subseteq V(G)-U_3$ and $L_3[w,t]\cup P_2[w,t]$ is an even cycle.
\end{enumerate} It is now clear that $C_1\cup C_2\cup L_3\cup Q_1\cup Q_2\cup Q_3\cup P_2$ has an $F_7$ minor.
\end{proof}


\section{Non-bipartite $\Omega$-system of flavour (NF2)}\label{sec-NF2}

In this section we prove proposition~\ref{prp-NF2}, namely that a minimal non-bipartite $\Omega$-system of flavour (NF2) has an $F_7$ minor, as long as there is no non-bipartite $\Omega$-system of flavour (NF1) with the same associated signed graft.
Observe that $L_1,L_2$ and $L_3$ are connected.
For convenience, whenever $L_i$ is non-simple, we write $P_i:=P(L_i)$ and $C_i:=C(L_i)$.
\begin{prp}\label{evencircuits}
Let $(G,\Sigma,\{x,y\})$ be a non-bipartite signed graft whose edges can be partitioned into odd $xy$-paths $Q_1,Q_2$. For each $i=1,2$, direct the edges of $Q_i$ from $x$ to $y$, and assume that every directed circuit in $Q_1\cup Q_2$ is even. Let $\vec{H}$ be the directed signed graft obtained by contracting all arcs that belong to at least one directed circuit. Then $\vec{H}$ is a non-bipartite and acyclic directed signed graft whose edges can be partitioned into two odd $xy$-dipaths.
\end{prp}
\begin{proof} Let $A$ be the set of all arcs that belong to at least one directed circuit. It is clear by construction that $\vec{H}$ is acyclic and can be partitioned as the union of two $xy$-dipaths $Q'_1, Q'_2$ where for $i=1,2,$ $Q'_i= Q_i-A$ ($Q'_i$ is equal to $Q_i/A$). Since every directed circuit is even, it follows that $Q'_1,Q'_2$ are odd $xy$-dipaths. To show $\vec{H}$ is non-bipartite, let $C$ be an odd circuit of $Q_1\cup Q_2$. Clearly, $C-A$ is a cycle of $\vec{H}$, and again, since every directed circuit is even, it follows that $C-A$ is an odd cycle of $\vec{H}$. In particular, $\vec{H}$ is non-bipartite.
\end{proof}
\begin{prp}\label{NF2mates}
Let $((G,\Sigma, \{s,t\}),\mathcal{L}=(L_1,\ldots,L_k))$ be a non-bipartite $\Omega$-system of flavour (NF2), where $\Omega$ has ends $s,s'$.
For $i\in [3]$, let $B_i$ be a $k$-mate of $L_i$. Then, \begin{enumerate}[\;\;(1)]
\item exactly one of $B_1,B_2,B_3$, say $B_3$, is an $st$-cut,

\item $L_1$ and $L_2$ are simple,

\item $(L_1\cup L_2)-\{\Omega\}$ is non-bipartite and $(L_1\cup L_3)-\{\Omega\}, (L_2\cup L_3)-\{\Omega\}$ are bipartite.
\end{enumerate} Furthermore, choose $U\subseteq V(G)-\{s,t\}$ such that $B_1\triangle B_2=\delta(U)$. Then, \begin{enumerate}
\item[\;\;(4)] for every $L\subseteq L_1\cup L_2\cup L_3$, $(L\cap B_1)-\{\Omega\}=(L\cap L_1)\cap \delta(U)$,

\item[\;\;(5)] $L_1$ and $L_2$ have at least one vertex of $U$ in common. 
\end{enumerate}
\end{prp}
\begin{proof}
{\bf (1)}
Proposition~\ref{matescutcut} implies that at least two of $B_1,B_2,B_3$ are signatures.
Suppose for a contradiction that each of $B_1,B_2,B_3$ is a signature.
For $i\in [3]$, note that $L_i-\{\Omega\}$ is an $s't$-path (recall that if $C(L_i)\neq \emptyset$, then $\Omega\in C(L_i)$ and the only vertex common to $C(L_i),P(L_i)$ is $s$), so let $Q_i:=L_i-\{\Omega\}$.
Since $$B_1\cap (Q_2\cup Q_3)=B_2\cap (Q_3\cup Q_1)=B_3\cap (Q_1\cup Q_2)=\emptyset$$ and $B_1,B_2,B_3$ are signatures, it follows that $Q_1\cup Q_2, Q_2\cup Q_3$ and $Q_3\cup Q_1$ are bipartite. Thus by proposition~\ref{2bipartite}, $Q_1\cup Q_2\cup Q_3=(L_1\cup L_2\cup L_3)-\{\Omega\}$ is bipartite, a contradiction.
{\bf (2)}
Suppose, for $j\in [3]$, $L_j$ is non-simple. Then $\Omega\in C_j$ and so by proposition~\ref{nonsimplestcutsign}, the covers in $\{B_1,B_2,B_3\}-\{B_j\}$ are signatures, and so by (1), $j=3$.
{\bf (3)}
Since $B_1$ and $B_2$ are signatures, it follows that $Q_2\cup Q_3$ and $Q_1\cup Q_3$ are bipartite. Then by proposition~\ref{2bipartite}, $Q_1\cup Q_2$ must be non-bipartite.
{\bf (4)}
By proposition~\ref{usefulparity}, $B_1\subseteq L_1\cup L_4\cup \ldots\cup L_k$. Thus, $L\cap B_1\subseteq L_1\cap B_1$, and so $L\cap B_1= L\cap (L_1\cap B_1)$. Hence, it suffices to show that $(L_1\cap B_1)-\{\Omega\}= L_1\cap \delta(U)$. Again, by proposition~\ref{usefulparity}, $L_1\cap B_2=\{\Omega\}$ and $\Omega\in L_1\cap B_1$, so $$L_1\cap \delta(U)=L_1\cap (B_1\triangle B_2)= (L_1\cap B_1)\triangle (L_1\cap B_2)= (L_1\cap B_1)-\{\Omega\},$$ as required.
{\bf (5)}
By (3) $(L_1\cup L_2)-\{\Omega\}$ contains an odd circuit $C$. Since $B_1$ is a signature, $|B_1\cap C|$ is odd. By (4) $C\cap B_1= (C\cap L_1)\cap \delta(U)$. Decompose $C\cap L_1$ into pairwise vertex-disjoint paths $Q_1,\ldots,Q_\ell$. Then, for some $i\in [\ell]$, $|Q_i\cap \delta(U)|$ is odd, and so $Q_i$ has one end, say $y$, in $U$ and the other in $V(G)-U$. Since $y\in V(L_1)\cap V(L_2)$, the result follows.
\end{proof}
\begin{prp}\label{NF2toNF1}
Let $((G,\Sigma,\{s,t\}),\mathcal{L}=(L_1,\ldots,L_k))$ be a non-bipartite $\Omega$-system of flavour (NF2), where $\Omega\in \delta(s)$.
Suppose there exist $C'_1,P'_1,L'_2$ and $L'_3$ such that \begin{enumerate}[\;\;(1)]

\item $C'_1\cup P'_1\cup L'_2\cup L'_3\subseteq L_1\cup L_2\cup L_3$,

\item $C'_1$ is an odd cycle, $P'_1$ is an even $st$-join, and $L'_2, L'_3$ are odd $st$-joins,

\item $\Omega\in P'_1\cap L'_2\cap L'_3$ and $\Omega\notin C'_1$,

\item the four sets $C'_1,P'_1,L'_2,L'_3$ are pairwise $\Omega$-disjoint.
\end{enumerate} Let $L'_1:=C'_1\cup P'_1$, and for each $j\in [3]$, let $\tilde{L}_j$ be a minimal odd $st$-join contained in $L'_j$. Then $((G,\Sigma,\{s,t\}), \mathcal{\tilde{L}}=(\tilde{L}_1,\tilde{L}_2,\tilde{L}_3,L_4,\ldots,L_k))$ is a non-bipartite $\Omega$-system of flavour (NF1). 
\end{prp}
\begin{proof} We will first show that $\Omega\in \tilde{L}_1\cap \tilde{L}_2\cap \tilde{L}_3$. For $j\in [3]$, let $\tilde{B}_j$ be a $k$-mate of $\tilde{L}_j$. 
By proposition~\ref{usefulparity4}, for $j\in [3]$, $\tilde{B}_j\subseteq \tilde{L}_j\cup L_4\cup \ldots\cup L_k$. Hence, for distinct $i,j\in [3]$, $\tilde{L}_i\cap \tilde{B}_j\subseteq \{\Omega\}$ and so $\tilde{L}_i\cap \tilde{B}_j= \{\Omega\}$, implying that $\Omega\in \tilde{L}_1\cap \tilde{L}_2\cap \tilde{L}_3$.

As $C'_1\cap (\tilde{L}_2\cup \tilde{L}_3\cup L_4\cup \cdots \cup L_k)=\emptyset$, we have $\tilde{B}_2\cap C'_1=\tilde{B}_3\cap C'_1=\emptyset$. Since $C'_1$ is an odd cycle, $\tilde{B}_2, \tilde{B}_3$ are $st$-cuts. So by proposition~\ref{matescutcut} one of $\tilde{L}_2, \tilde{L}_3$, say $\tilde{L}_2$, is non-simple and $\Omega\in P(\tilde{L}_2)$. Hence, $((G,\Sigma,\{s,t\}), \mathcal{\tilde{L}})$ is a non-bipartite $\Omega$-system of flavour (NF1) (because it is not of flavour (NF2)).
\end{proof}
\begin{lma}\label{NF2intersection}
Let $((G,\Sigma,\{s,t\}),\mathcal{L}=(L_1,\ldots,L_k))$ be a minimal non-bipartite $\Omega$-system of flavour (NF2), where $\Omega$ has ends $s,s'$, and assume there is no non-bipartite $\Omega$-system of flavour (NF1) with the same associated signed graft. Suppose that $L_1,L_2$ are simple and $(L_1\cup L_2)-\{\Omega\}$ is non-bipartite. Then the following hold: \begin{enumerate}[\;\;(1)]

\item For $i=1,2$, the only vertices $L_i$ and $L_3$ have in common are $s,s',t$.

\item For $i=1,2$, direct the edges of $L_i$ from $s$ to $t$. Then every directed circuit in $L_1\cup L_2$ is even.
\end{enumerate}
\end{lma}
\begin{proof}
For $i\in [3]$, let $B_i$ be a $k$-mate of $L_i$.
Since $(L_1\cup L_2)-\{\Omega\}$ is non-bipartite, proposition~\ref{NF2mates} implies that for $i=1,2$, $(L_i\cup L_3)-\{\Omega\}$ is bipartite, $B_3$ is an $st$-cut and $B_1,B_2$ are signatures. Thus there exists $U\subseteq V(G)-\{s,t\}$ such that $B_1\triangle B_2=\delta(U)$.
By proposition~\ref{NF2mates}, $L_1$ and $L_2$ have a vertex $y$ in common in $U$, and the two cycles $L_1[s',y]\cup L_2[s',y]$, $L_1[y,t]\cup L_2[y,t]$ are odd.

{\bf (1)}
In the first case, assume $L_3$ is simple.
Suppose for a contradiction that $L_3$ has a vertex other than $s,s',t$ in common with one of $L_1,L_2$.
Let $v_1$ (resp. $v_2$) be the closest vertex to $s$ (resp. furthest vertex from $s$) of $L_3$ different from $s,s',t$ that also belongs to one of $L_1,L_2$. We may assume that $v_2\in V(L_2)\cap V(L_3)$, and choose $j\in \{1,2\}$ so that $v_1\in V(L_j)\cap V(L_3)$.

\begin{claim} There exists an odd cycle $C$ in $(L_1\cup L_2)-\{\Omega\}$ that is disjoint from either $L_j[s',v_1]$ or $L_2[v_2,t]$.\end{claim}

\begin{cproof} Suppose otherwise. Then $j=1$ and $y$ must belong to the interior of the both of $L_1[s',v_1], L_2[v_2,t]$. Let \begin{align*}
P'_1&= L_1[s,y]\cup L_2[y,t]\\
C'_1&= L_1[y,v_1]\cup L_3[v_1,v_2]\cup L_2[v_2,y]\\
L'_1&= C'_1\cup P'_1\\
L'_2&= L_3[s,v_1]\cup L_1[v_1,t]\\
L'_3&= L_2[s,v_2]\cup L_2[v_2,t].\end{align*} By proposition~\ref{NF2mates}, $P'_1$ is an even $st$-join, $C'_1$ is an odd cycle, and for $j\in [3]$, $L'_j$ is an odd $st$-join. Therefore, for $j\in [3]$, there is a minimal odd $st$-join $\tilde{L}_j$ contained in $L'_j$. Proposition~\ref{NF2toNF1} implies that $((G,\Sigma,\{s,t\}), (\tilde{L}_1,\tilde{L}_2,\tilde{L}_3,L_4,\ldots,L_k))$ is a non-bipartite $\Omega$-system of flavour (NF1), contrary to our hypothesis.
\end{cproof}

Observe that $L_3[s',v_1]$ and $L_3[v_2,t]$ are paths whose internal vertices by definition have degree two in $G[L_1\cup L_2\cup L_3]$, and the two cycles $L_3[s',v_1]\cup L_j[s',v_1]$, $L_3[v_2,t]\cup L_2[v_2,t]$ are even.
Lemma~\ref{nbdisentangle} implies that either $L_j[s',v_1]\cup \{\Omega\}$ or $L_2[v_2,t]\cup \{\Omega\}$ has a $k$-mate $B$. Since $B\cap C=\emptyset$, it follows that $B$ is an $st$-cut. However, $B$ is also a $k$-mate for one of $L_j, L_2$. Hence, since $B_3$ is also an $st$-cut, proposition~\ref{matescutcut} implies that one of $L_j,L_2,L_3$ is non-simple and $\Omega$ lies in its even $st$-path, a contradiction.

In the remaining case, $L_3$ is non-simple and $\Omega\in C_3$.
We will first show that $C_3$ has no vertex other than $s,s'$ in common with either of $L_1,L_2$. Suppose otherwise.
Choose a vertex $v\in V(C_3)-\{s,s'\}$ that also belongs to one of $L_1,L_2$, and such that all the internal vertices of the subpath $C_3[s',v]$ in $C_3-\{\Omega\}$ have degree two in $G[L_1\cup L_2\cup L_3]$. Let $C_3[s,v]:=\{\Omega\}\cup C_3[s',v]$ and $C_3[v,s]:=C_3-C_3[s,v]$.
By symmetry between $L_1$ and $L_2$, we may assume that $v\in V(L_1)\cap V(C_3)$.

\begin{claim} There exists an odd cycle $C$ in $(L_1\cup L_2)-\{\Omega\}$ that is disjoint from $L_1[s',v]$.\end{claim}

\begin{cproof} Suppose otherwise. Then $y$ must belong to the interior of $L_1[s',v]$. Let \begin{align*}
C'_3&= L_1[s,v]\cup C_3[v,s]\\
L'_3&= C'_3\cup P_3\\
L'_1&= C_1[s,v]\cup L_1[v,t]\\
C'&= L_1[s',y]\cup L_2[s',y].\end{align*} 
By proposition~\ref{NF2mates}, $C'_3,C'$ are odd cycles and $L'_1, L'_3$ are odd $st$-joins. Therefore, $L'_1$ has a $k$-mate $B$. Since $L'_1\cap C'=\emptyset$, it follows that $B\cap C'=\emptyset$ and so $B$ is an $st$-cut. However, $B\cap L'_3=\{\Omega\}$, implying that $B\cap C'_3=\{\Omega\}$, a contradiction.
\end{cproof}

Recall that $C_3[s',v]$ is a path whose internal vertices have degree two in $G[L_1\cup L_2\cup L_3]$, and the cycle $C_3[s',v]\cup L_1[s',v]$ is even.
Lemma~\ref{nbdisentangle} therefore implies that $L_1[s,v]=L_1[s',v]\cup \{\Omega\}$ has a $k$-mate $B$. Since $B\cap C=\emptyset$, it follows that $B$ is an $st$-cut. However, $B\cap (L_1[s,v]\cup C_3[v,s])=\{\Omega\}$, a contradiction (as $L_1[s,v]\cup C_3[v,s]$ is an odd cycle).

We next show that $P_3$ has no vertex other than $s,t$ in common with either of $L_1,L_2$. Suppose otherwise.
Choose a vertex $v\in V(P_3)-\{s,t\}$ that also belongs to one of $L_1,L_2$, and such that all the internal vertices of the subpath $P_3[v,t]$ have degree two in $G[L_1\cup L_2\cup L_3]$.
By symmetry between $L_1$ and $L_2$, we may assume that $v\in V(L_1)\cap V(P_3)$.

\begin{claim} There exists an odd cycle $C$ in $(L_1\cup L_2)-\{\Omega\}$ that is disjoint from $L_1[v,t]$.\end{claim}

\begin{cproof} Suppose otherwise. Then $y$ must belong to the interior of $L_1[v,t]$. Let \begin{align*}
L'_1&= L_1[s,v]\cup P_3[v,t]\\
C'&= L_1[y,t]\cup L_2[y,t].
\end{align*} 
By proposition~\ref{NF2mates}, $C'$ is an odd cycle, and $L'_1$ is an odd $st$-join. Therefore, $L'_1$ has a $k$-mate $B$. Since $L'_1\cap C'=\emptyset$, it follows that $B\cap C'=\emptyset$ and so $B$ is an $st$-cut. However, $B\cap C_3=\{\Omega\}$, a contradiction.
\end{cproof}

Recall that $P_3[v,t]$ is a path whose internal vertices have degree two in $G[L_1\cup L_2\cup L_3]$, and the cycle $P_3[v,t]\cup L_1[v,t]$ is even.
Lemma~\ref{nbdisentangle} therefore implies that $L_1[v,t]\cup \{\Omega\}=L_1-L_1[s',v]$ has a $k$-mate $B$. Since $B\cap C=\emptyset$, it follows that $B$ is an $st$-cut. However, $B\cap C_3=\{\Omega\}$, a contradiction.

{\bf (2)}
Suppose otherwise. Let $C$ be a directed odd circuit contained in $L_1\cup L_2$, and let $P'_1\cup P'_2$ be two $st$-joins in $(L_1\cup L_2)-C$ such that $P'_1\cup P'_2=(L_1\cup L_2)-C$ and $P'_1\cap P'_2=\{\Omega\}$. Then one of $P'_1,P'_2$ is odd and the other is even, say $P'_1$ is even and $P'_2$ is odd. Let $L'_1:=C\cup P'_1$, $L'_2:=P'_2$ and $L'_3:=L_3$. For $j\in [3]$, let $\tilde{L}_j$ be a minimal odd $st$-join contained in $L'_i$. Then proposition~\ref{NF2toNF1} implies that $((G,\Sigma,\{s,t\}), (\tilde{L}_1, \tilde{L}_2,\tilde{L}_3, L_4,\ldots,L_k))$ is a non-bipartite $\Omega$-system of flavour (NF1), contrary to our hypothesis.
\end{proof}

We are now ready to prove proposition~\ref{prp-NF2}.

\begin{proof}[Proof of proposition~\ref{prp-NF2}]
Let $((G,\Sigma,\{s,t\}),\mathcal{L}=(L_1,\ldots,L_k))$ be a minimal non-bipartite $\Omega$-system of flavour (NF2), where $\Omega$ has ends $s,s'$, and assume there is no non-bipartite $\Omega$-system of flavour (NF1) with the same associated signed graft. Proposition~\ref{NF2mates} allows us to assume $L_1, L_2$ are simple and $(L_1\cup L_2)-\{\Omega\}$ is non-bipartite, and in turn, lemma~\ref{NF2intersection} implies that, for $i=1,2$, the only vertices $L_i$ and $L_3$ have in common are $s,s',t$. For $i\in \{2,3\}$, let $B_i$ be a $k$-mate of $L_i$. By proposition~\ref{NF2mates}, $L_3$ is an $st$-cut $\delta(U), U\subseteq V(G)-\{t\}$.

If $L_3$ is non-simple, then it is easily follows from proposition~\ref{evencircuits} and lemma~\ref{NF2intersection} that $L_1\cup L_2\cup L_3$ has an $F_7$ minor. Otherwise, when $L_3$ is simple, proposition~\ref{matescutsign} implies the existence of a shortest path $P$ in $G[U]$ between $s$ and some vertex, say $v$, of $(V(L_3)\cap U)-\{s\}$ that is disjoint from $B_2$. Note that $L_3[s,v]\cup P$ is an odd cycle. It now easily follows from proposition~\ref{evencircuits} and lemma~\ref{NF2intersection} that $L_1\cup L_2\cup L_3\cup P$ has an $F_7$ minor.
\end{proof}


\section{Preliminaries for bipartite $\Omega$-systems}\label{sec-prelim-bp}
%
\subsection{Basic properties}\label{sec-bpbasic}
\begin{rem}\label{bpsimplenonsimple}
Let $((G,\Sigma,T), (L_1,\ldots,L_k), m)$ be a bipartite $\Omega$-system, where $L_1,L_2,L_3$ are minimal odd $T$-joins. 
Since $(L_1\cup L_2\cup L_3)-\{\Omega\}$ is bipartite,
for each $i\in [3]$, either $L_i$ is simple or $\Omega\in C(L_i)$.
\end{rem}
\begin{prp}\label{bpmates}
Let $((G,\Sigma,T), (L_1,\ldots,L_k), m)$ be a bipartite $\Omega$-system, where $L_1,L_2,L_3$ are minimal odd $T$-joins.
For $i\in [3]$, let $B_i$ be a $k$-mate of $L_i$. Then at least two of $B_1,B_2,B_3$ are signatures.
\end{prp}
\begin{proof}
By remark~\ref{bpsimplenonsimple}, for every $i\in [3]$, $L_i$ is either simple or $\Omega\in C(L_i)$. The result now follows immediately from proposition~\ref{matescutcut}.
\end{proof}
\begin{prp}\label{bpsignature}
Let $((G,\Sigma,T), (L_1,\ldots,L_k), m)$ be a bipartite $\Omega$-system. Suppose $L\subseteq L_1\cup L_2\cup L_3\cup P(L_4)\cup \cdots \cup P(L_m)$ has a signature $k$-mate $B$. Then $B\cap \big(L_1\cup L_2\cup L_3\cup P_4\cup \cdots \cup P_m\big)=B\cap L$.
\end{prp}
\begin{proof}
As $B$ is a signature, it intersects each of $C_4,\ldots,C_m, L_{m+1},\ldots,L_k$. Hence, $$k-3\geq |B-L|\geq \sum_{j=4}^m |B\cap C_j|+ \sum_{j=m+1}^k |B\cap L_j|\geq k-3,$$ so equality holds throughout, implying that $B-L\subseteq C_4\cup \cdots \cup C_m \cup L_{m+1}\cup \cdots \cup L_k$, implying the result.
\end{proof}
%
\subsection{The mate proposition}\label{sec-mate}
\begin{prp}\label{mateprop}
Let $((G,\Sigma,\{s,t\}), \mathcal{L}=(L_1,\ldots,L_k), m)$ be a bipartite $\Omega$-system, where $\Omega\in \delta(s)$.
For each $i\in [m]$, let $\widetilde{P_i}\subseteq L_i$ be a connected $st$-join such that $\widetilde{P_i}\cap \Sigma\subseteq \{\Omega\}$, and if $\Omega\in \widetilde{P_i}$, then $\widetilde{P_i}\cap \delta(s)=\{\Omega\}$.
Suppose, for each $i\in [m]$, there exists a $k$-mate $B_i$ of $\widetilde{P_i}\cup \{\Omega\}$. 
Then one of $B_1,\ldots,B_m$ is not a signature.
\end{prp}
To prove this proposition, we will need a lemma, for which we introduce some notations. For $i\in [m]$, let $Q_i:= \widetilde{P_i}\cup \{\Omega\}$.
Given two signatures $B_i,B_\ell$, we choose $U_{i\ell}\subseteq V(G)-\{s,t\}$ such that $\delta(U_{i\ell})=B_i\triangle B_\ell$. 
For each $i\in [m]$, define $\widetilde{C_i}$ as follows: if $\widetilde{P_i}$ is odd then $\widetilde{C_i}:=\emptyset$, and otherwise $\widetilde{C_i}$ is an odd circuit contained in the odd cycle $L_i\triangle \widetilde{P_i}=L_i- \widetilde{P_i}$.
\begin{lma}\label{mate-shore}
Let $((G,\Sigma,\{s,t\}), \mathcal{L}=(L_1,\ldots,L_k), m)$ be a bipartite $\Omega$-system, where $G$ is connected and $\Omega$ has ends $s,s'$.
Let $J\subseteq [m]$ be an index subset of size at least three.
Suppose, for each $i\in J$, there exists a signature $k$-mate $B_i$ for $Q_i$. Then, for each $i\in J$, the following hold: \begin{enumerate}[\;\;(1)]
\item $B_i$ is a $k$-mate of $L_i$, and so $B_i$ is a cap of $L_i$ in $\mathcal{L}$,

\item for $\ell \in [m]$ such that $\widetilde{C_\ell}\neq \emptyset$, $|B_i\cap \widetilde{C_\ell}|=1$,

\item for $\ell \in [m]-\{i\}$, $B_i\cap Q_\ell=\{\Omega\}$.
\end{enumerate}
Now pick $j\in J$ and let $S:=\bigcap (U_{ij}: i\in J, i<j)$. Then, \begin{enumerate}[\;\;(1)]
\setcounter{enumi}{3}

\item $\Omega\notin \delta(S)$,

\item $\delta(S)\subseteq \bigcup (B_i: i\in J, i\leq j)$,

\item for distinct $i,\ell \in J-\{j\}$, $S\cap U_{i\ell}=\emptyset$,

\item $Q_j\cap \delta(S)= (Q_j\cap B_j)-\{\Omega\}$,

\item for $\ell \in [m]-\{j\}$, $Q_\ell\cap \delta(S)=\emptyset$.
\end{enumerate} 
Next take $L\in \{L_{m+1},\ldots,L_k\}$ and $C\in \{\widetilde{C_1},\ldots,\widetilde{C_m}\}$. Then, \begin{enumerate}[\;\;(1)]
\setcounter{enumi}{8}

\item if $L\cap \delta(S)\neq \emptyset$, then $|L\cap \delta(S)|=2$ and $|L\cap \delta(S)\cap B_j|=1$,

\item if $C\cap \delta(S)\neq \emptyset$, then $|C\cap \delta(S)|=2$,

\item if $C\cap \delta(S)\neq \emptyset$ and, for some $i,\ell\in J$ such that $i<\ell<j$, $C\cap \delta(S)\subseteq B_i\cup B_\ell$, then $V(C)\subseteq U_{ij}\cup U_{\ell j}$.
\end{enumerate}
\end{lma}
\begin{proof}

{\bf (1)}
If $i\in J\cap [3]$, then $Q_i\subseteq L_i$, and so $B_i$ is clearly a $k$-mate of $L_i$.
Otherwise, when $i\in J-[3]$, $B_i\cap \widetilde{C_i}\neq \emptyset$ as $B_i$ is a signature, and so $$|B_i-L_i|\leq |B_i-\widetilde{P_i}|-|B_i\cap \widetilde{C_i}|\leq (k-2)-1=k-3,$$ implying that $B_i$ is a $k$-mate of $L_i$. 
Hence, by proposition~\ref{packmate}, $B_i$ is a cap of $L_i$ in $\mathcal{L}$.
{\bf (2)}
Thus, if $\ell\neq i$ then $|B_i\cap \widetilde{C_\ell}|=1$ (note $B_i$ is a signature and $\widetilde{C_\ell}$ is an odd circuit). If $\ell=i$ and $i\notin [3]$, we have $$k-3\leq 
|B_i\cap \widetilde{C_4}|+\cdots+|B_i\cap \widetilde{C_m}|+|B_i\cap L_{m+1}|+\cdots+|B_i\cap L_k|
\leq |B_i-Q_i|\leq k-3,$$ so equality holds throughout, in particular, $|B_i\cap \widetilde{C_i}|= 1$. Otherwise, when $\ell=i$ and $i\in [3]$, then $$k-3\leq |B_i\cap L_4|+\cdots+|B_i\cap L_k|\leq |B_i-Q_i|\leq k-3,$$ so equality holds throughout, in particular, the middle equality implies that $B_i\cap \widetilde{C_i}= \{\Omega\}$.

{\bf (3)}
Note $|B_i\cap L_\ell|=1$. 
If $\ell\in [3]$, then $B_i\cap L_\ell=\{\Omega\}$ and so $B_i\cap Q_\ell=\{\Omega\}$.
Otherwise, $\ell\in [m]-[3]$. By (2), $|B_i\cap \widetilde{C_\ell}|=1$ and so $B_i\cap P_\ell=\emptyset$, implying that $B_i\cap Q_\ell=\{\Omega\}$.

{\bf (4)}
Note $\Omega\in B_i, i\in J$. In particular, for all $i\in J$ such that $i<j$, $\Omega\notin \delta(U_{ij})$ and so $s'\notin U_{ij}$. Thus $s'\notin S$, and since $s\notin S$, it follows that $\Omega\notin \delta(S)$.

{\bf (5)}
We have $$\delta(S)\subseteq \bigcup (\delta(U_{ij}): i\in J, i<j)\subseteq \bigcup (B_i: i\in J, i\leq j).$$

{\bf (6)}
Observe that $$\delta(U_{i\ell} \triangle U_{\ell j} \triangle U_{ji})= \delta(U_{i\ell})\triangle \delta(U_{\ell j})\triangle \delta(U_{ji}) =(B_i\triangle B_\ell) \triangle (B_\ell \triangle B_j)\triangle (B_j\triangle B_i)=\emptyset.$$ As $G$ is connected, it follows that $U_{i\ell} \triangle U_{\ell j} \triangle U_{ji}$ is either $\emptyset$ or $V(G)$. However, as $s,t\notin U_{i\ell} \triangle U_{\ell j} \triangle U_{ji}$, it must be that $U_{i\ell} \triangle U_{\ell j} \triangle U_{ji}=\emptyset$. Hence, $U_{i\ell}\cap U_{\ell j} \cap U_{ji}=\emptyset$, and so in particular, $U_{i\ell}\cap S=\emptyset$.

{\bf (7)}
Since $\Omega\in Q_j\cap B_j$, we have $$Q_j\cap \delta(U_{ij})=Q_j\cap (B_j\triangle B_i)= (Q_j\cap B_j)\triangle (Q_j\cap B_i)= (Q_j\cap B_j)\triangle \{\Omega\}= (Q_j\cap B_j)- \{\Omega\}.$$
Thus, $$Q_j\cap \delta(S) \subseteq \bigcup (Q_j\cap \delta(U_{ij}): i\in J, i<j) = (Q_j\cap B_j)-\{\Omega\}.$$ Since $s,t\notin U_{ij}$ for all $i\in J$ with $i<j$ and since $Q_1,\ldots,Q_m$ are all connected, equality holds above.

{\bf (8)}
As $|J|\geq 3$, there exists $i\in J-\{j,\ell\}$. By (4) $B_i\cap Q_\ell = B_j\cap Q_\ell=\{\Omega\}$, and so as $Q_\ell$ is connected, $V(Q_\ell)\cap U_{ij}=\emptyset$. In particular, $V(Q_\ell)\cap S=\emptyset$, so $Q_\ell\cap \delta(S)=\emptyset$.

{\bf (9)}
Since $L$ is connected, we can traverse its vertices in some order $s=v_0,v_1,v_2,\ldots, v_p=t$, where $L=\{e_x:=\{v_{x-1},v_x\}: 1\leq x\leq p\}$.
Choose $1\leq  x< y \leq p$ such that $e_x, e_y\in \delta(S)$ with $v_x, v_{y -1}\in S$. 
Either $B_j\cap L[s,v_x]=\emptyset$ or $B_j\cap L[v_{y-1}, t]=\emptyset$ (as $|B_j\cap L|=1$).
We assume $B_j\cap L[s,v_x]=\emptyset$, and the other case can be dealt with similarly.
For $i\in J$ such that $i<j$, as $v_x\in U_{ij}$ and $s\notin U_{ij}$, it follows that $\delta(U_{ij})\cap L[s, v_x]\neq \emptyset$, but $B_j\cap L[s, v_x]=\emptyset$, implying that $B_i\cap L[s, v_x]\neq \emptyset$.
We claim that $e_y\in B_j$. As $v_y\notin S$, there exists $i\in J$ such that $i<j$ and $v_y\notin U_{ij}$ and so $e_y\in \delta(U_{ij})$. 
However, as $|B_i\cap L|=1$ and $B_i\cap L[s, v_x]\neq \emptyset$, we get $B_i\cap L[v_{y-1}, t]= \emptyset$. In particular, $e_y\notin B_i$ and so $e_y\in B_j$. 
Since for all $i\in J$ such that $i\leq j$, $|B_i\cap L|=1$, it follows that $L\cap \delta(S)= \{e_x, e_y\}$ and $L\cap \delta(S)\cap B_j=\{e_y\}$.

{\bf (10)}
As above, we traverse the vertices of $C$ in some order $v_0,v_1,\ldots,$ $v_{p-1}, v_p=v_0$, where $v_0\in S$ and $C=\{e_x:=\{v_{x-1}, v_x\}: 1\leq x\leq p\}$.
Assume there exist $1\leq  x< y\leq p$ such that $e_x, e_y\in \delta(S)- B_j$ with $v_x, v_{y-1}\notin S$. 
Then, for some $i\in J$ such that $i<j$, $v_x\notin U_{ij}$ and $e_x\in \delta(U_{ij})$. 
Since $e_x\notin B_j$, it follows that $e_x\in B_i$.
Thus, as $|C\cap B_i|=1$ and $e_y\notin B_j$, $e_y\notin \delta(U_{ij})$ and $v_{y-1}\in U_{ij}$. Let $C[v_x, v_{y-1}]$ be the $v_x v_{y-1}$-subpath of $C$ not containing either of $e_x, e_{y-1}$. Then $C[v_x, v_{y-1}]\cap \delta(U_{ij})\neq \emptyset$. 
Since $C\cap B_i=\{e_x\}$, we get that $C[v_x, v_{y-1}]\cap B_j\neq \emptyset$.
To summarize, if there exist $1\leq  x< y\leq p$ such that $e_x, e_y\in \delta(S)- B_j$ with $v_x, v_{y-1}\notin S$, then $C[v_x, v_{y-1}]\cap B_j\neq \emptyset$. Therefore, as $|C\cap B_j|=1$, we get that $|C\cap \delta(S)|=2$.

{\bf (11)}
By (10) $C\cap \delta(S)=\{e_x,e_y\}$ where $e_x\in B_i$ and $e_y\in B_\ell$.
If $e_x\in B_j$ then $C\cap \delta(U_{ij})=\emptyset$, but $V(C)\cap S\neq \emptyset$ and $S\subseteq U_{ij}$, implying that $V(C)\subseteq U_{ij}\subseteq U_{ij}\cup U_{\ell j}$, and we are done. 
Similarly, if $e_y\in B_j$ then $V(C)\subseteq U_{\ell j}\subseteq U_{ij}\cup U_{\ell j}$, and we are again done.
Otherwise, $\{e_x, e_y\}\cap B_j=\emptyset$.
As $e_x\in B_i- B_j$, it follows that $e_x\in \delta(U_{ij})$, and since $v_{x-1}\in S\subseteq U_{ij}$, we get $v_x\notin U_{ij}$. Also, as $|C\cap B_i|=1$, we have $e_y\notin B_i$. This, together with the facts that $e_y\notin B_j$ and $v_y\in S\subseteq U_{ij}$, implies that $v_{y-1}\in U_{ij}$. 
Since $C\cap B_i=\{e_x\}$ and $|C\cap B_j|=1$, there exists $z\in [y-1]-[x]$ such that $$C\cap B_j =\{e_z\}\quad\mbox{and}\quad v_z, v_{z+1}, \ldots, v_{y-1}\in U_{ij}.$$
Similarly, we have $$
C\cap B_\ell =\{e_z\}\quad\mbox{and}\quad v_x,v_{x+1},\ldots,v_{z-1}\in U_{\ell j}.$$
As a result, since $v_0,v_1,\ldots,v_{x-1}, v_y,v_{y+1},\ldots,v_{p-1}\in S\subseteq U_{ij}\cap U_{\ell j}$, it follows that $V(C)\subseteq U_{ij}\cup U_{\ell j}$.
\end{proof}

We are now ready to prove the mate proposition~\ref{mateprop}.

\begin{proof}[Proof of proposition~\ref{mateprop}]
We assume that $\Omega$ has ends $s,s'$.
By identifying a vertex of each component with $s$, if necessary, we may assume that $G$ is connected.
Suppose, for a contradiction, that $B_1,\ldots,B_m$ are all signatures.
We will be applying lemma~\ref{mate-shore} to the index set $[m]$.
Notice first that as a corollary of parts (1)-(3), we have that $B_j\cap L_i \subseteq \widetilde{C_i}\cup \widetilde{P_i}$ for all $i,j\in [m]$.
For distinct $i,j\in [m]$, choose $U_{ij}\subseteq V(G)-\{s,t\}$ such that $\delta(U_{ij})=B_i\triangle B_j$. For each $j\in \{3,\ldots,m\}$, let $$S_j:= \bigcap (U_{ij}: 1\leq i<j).$$
Let $C\in \{\widetilde{C_1},\ldots,\widetilde{C_m}\}$ and $S_j\in \{S_3,\ldots,S_m\}$. We say {\it $C$ is bad for~$S_j$} if $$|C\cap \delta(S_j)|=2 \quad\mbox{and}\quad C\cap \delta(S_j)\cap B_j=\emptyset.$$

\begin{claim}
One of $S_3,\ldots,S_m$ has no bad circuit.
\end{claim}
\begin{cproof}
Let $C$ be a bad circuit for some $S_j, 3\leq j\leq m$.
Then by lemma~\ref{mate-shore} parts (2) and (5),
$$C\cap \delta(S_j)\subseteq B_i\cup B_\ell, \quad \text{ for some } 1\leq i<\ell< j.$$
Therefore, by lemma~\ref{mate-shore}(11), $V(C)\subseteq U_{ij}\cup U_{\ell j}$.
In particular, $s\notin V(C)$ and $$V(C)\cap S_{j+1}= V(C)\cap S_{j+2}= \cdots= V(C)\cap S_m=\emptyset,$$ since by lemma~\ref{mate-shore}(6), $(U_{ij}\cup U_{\ell j})\cap  S=\emptyset$, for all $S\in \{S_{j+1},\ldots,S_m\}$.
As a result, $C\notin \{\widetilde{C_1},\widetilde{C_2},\widetilde{C_3}\}$ and $C$ is not bad for any of $S_{j+1},\ldots, S_m$.
Thus every circuit is bad for at most one of $S_3,\ldots,S_m$ and every bad circuit is one of $\widetilde{C_4},\ldots,\widetilde{C_m}$. Thus, one of $S_3,\ldots,S_m$ has no bad circuit.
\end{cproof}

Choose $j\in \{3,\ldots,m\}$ so that $S_j$ has no bad circuit, and let $B:= B_j\bigtriangleup \delta(S_j)$. Notice that for each $i\in [m]$, $B\cap L_i\subseteq \widetilde{C_i}\cup \widetilde{P_i}$.

\begin{claim} 
$B$ is a cover of size $k-2$.
\end{claim}
\begin{cproof}
It is clear that $B$ is a cover. It remains to show that $|B|=k-2$.
By lemma~\ref{mate-shore},
$$B\subseteq \bigcup (B_i: 1\leq i\leq j)\subseteq \bigcup_{i=1}^{k} L_i.$$ The first inclusion follows from part (5) and the second inclusion follows from part (1).
Therefore, as $\Omega\in B$, it suffices to show that, for all $i\in [k]$, $|B\cap L_i| = 1$. Observe that, for all $i\in [k]-\{j\}$, $|B_j\cap L_i|=1$.

Take $i\in [k]-[m]$. If $L_i\cap \delta(S_j)=\emptyset$, then $|L_i\cap B|=|L_i\cap B_j|=1$. Otherwise, when $L_i\cap \delta(S_j)\neq \emptyset$, lemma~\ref{mate-shore}(9) implies $|L_i\cap \delta(S_j)|=2$ and $|L_i\cap \delta(S_j)\cap B_j|=1$, so $|L_i\cap B|= |L_i\cap (B_j\bigtriangleup \delta(S_j))|=1$.

Next take $i\in [m]$.
We will first consider $\widetilde{C_i}\cap B$, given that $\widetilde{C_i}\neq \emptyset$. If $\widetilde{C_i}\cap \delta(S_j)=\emptyset$, then $|\widetilde{C_i}\cap B|=|\widetilde{C_i}\cap B_j|=1$. Otherwise, $\widetilde{C_i}\cap \delta(S_j)\neq \emptyset$. Then, by lemma~\ref{mate-shore}(10), $|\widetilde{C_i}\cap \delta(S_j)|=2$. By our choice of $S_j$, $\widetilde{C_i}$ is not bad for $S_j$, and so $|\widetilde{C_i}\cap \delta(S_j)\cap B_j|=1$. Thus, $|\widetilde{C_i}\cap B|= |\widetilde{C_i}\cap (B_j\bigtriangleup \delta(S_j))|=1$.
We next consider $(\{\Omega\}\cup \widetilde{P_i})\cap B$. If $i\neq j$, then by lemma~\ref{mate-shore}, \begin{align*}
(\{\Omega\}\cup \widetilde{P_i})\cap B &=(\{\Omega\}\cup \widetilde{P_i})\cap (B_j\bigtriangleup \delta(S_j))\\
&= (\{\Omega\}\cup \widetilde{P_i})\cap B_j \quad \text{ by part (8)}\\
&=\{\Omega\}\quad \text{ by part (3).}
\end{align*} On the other hand, if $i=j$, then by lemma~\ref{mate-shore}, \begin{align*}
(\{\Omega\}\cup \widetilde{P_j})\cap B &=(\{\Omega\}\cup \widetilde{P_j})\cap (B_j\bigtriangleup \delta(S_j))\\
&=[(\{\Omega\}\cup \widetilde{P_j})\cap B_j] \triangle [(\{\Omega\}\cup \widetilde{P_j})\cap \delta(S_j)]\\
&=\{\Omega\}\quad \text{ by part (7).}
\end{align*} 
Since whenever $\Omega\in \widetilde{P_i}$ then $\widetilde{C_i}=\emptyset$, $$|L_i\cap B|= |\widetilde{C_i}\cap B|+|\widetilde{P_i}\cap B|= 1,$$ as $L_i\cap B\subseteq \widetilde{C_i}\cup \widetilde{P_i}$.
\end{cproof}

By claim~2, $|B|=k-2$. However, $B$ is cover and so $|B|\geq \tau(G,\Sigma)\geq k$, a contradiction.\end{proof}
\subsection{The odd-$K_5$ lemma}
 The following lemma is essentially due to Schrijver~\cite{Schrijver02}, and the presentation follows Geelen and Guenin~\cite{Geelen02}.
\begin{lma}[\cite{Schrijver02,Geelen02}]\label{oddK5}
Let $G=(V,E)$ be a graph and let $\Omega$ be an edge of $G$ with ends $s,s'$.
Let $U_0,U_1,U_2,U_3$ be a partition of $V(G)$, and let $P_1,P_2,P_3$ be internally vertex-disjoint $ss'$-paths in $G\setminus \Omega$ such that \begin{enumerate}[\;\;(i)]
\item $s,s'\in U_0$, and for $i\in \{0,1,2,3\}$, $U_i$ is a stable set in $G\setminus \Omega$,
\item for $i\in [3]$, $V(P_i)\subseteq U_0\cup U_i$, and
\item for distinct $i,j\in [3]$, there is a path between $P_i$ and $P_j$ in $G[U_i\cup U_j]$.
\end{enumerate}
Then $(G,E(G))$ has a $\widetilde{K_5}$ minor.
\end{lma}
\subsection{Mates and connectivity}
\begin{prp}\label{matessignsign}
Let $(G,\Sigma,\{s,t\})$ be a signed graft and $(L_1,\ldots,L_k)$ be an $(\Omega, k)$-packing, where $\Omega\in \delta(s)$.
Suppose that for $i=1,2$ there exists a signature $B_i$ that is a $k$-mate of $L_i$.
Let $U\subseteq V(G)-\{s,t\}$ such that $B_1\triangle B_2=\delta(U)$.
For $i=1,2$ let $W_i=V(L_i)\cap U$.
Then there exists a path $P$ in $G[U]$ between a vertex in $W_1$ and a vertex in $W_2$ such that $P\cap (B_1\cup B_2)=\emptyset$.
\end{prp}
\begin{proof}
Suppose first that there is no path in $G[U]$ between $W_1$ and $W_2$.
Then there exists $U'\subset U$ such that $W_1\subseteq U'$, 
$W_2\subseteq U-U'$ and there is no edge of $G$ with one end in $U'$ and one end in $U-U'$.
Then $B=B_1\triangle\delta(U')$ is a signature of $(G,\Sigma,\{s,t\})$ where $B\subseteq B_1\cup B_2$ and $B\cap (L_1\cup L_2)=\{\Omega\}$,
contradicting proposition~\ref{usefulparity} part (4).

Thus there exists a path $P$ in $G[U]$ between $W_1$ and $W_2$ with minimum number of edges in $B_1\cup B_2$.
Suppose $P$ has an edge $e\in B_i$ for some $i\in[2]$.
Then $e\in B_1\cap B_2$ as $e\notin\delta(U)$.
Since $s\notin U$, $e\neq\Omega$.
Proposition~\ref{packmate} implies that for some $j\in [k]-[3]$, $e\in L_j$ and $B_1\cap L_j=B_2\cap L_j=\{e\}$.
Hence, since $e\in E(G[U])$ and $s,t\notin U$, 
$e$ must belong to an odd circuit $C$ contained in $L_j\cap E(G[U])$.
But then replacing $P$ by $P\triangle C$ we obtain a new walk in $G[U]$ between $W_1$ and $W_2$ with fewer edges in $B_1\cup B_2$, contradicting our choice of $P$.
\end{proof}
\subsection{Acyclicity and flows}
\begin{prp}\label{intersection}
Consider an acyclic digraph whose edges can be written as the union of dipaths $Q_1,\ldots,Q_n$ rooted from some vertex $x$. 
Suppose that $Q_1,\ldots,Q_n$ use distinct arcs incident with $x$.
Consider the following partial ordering defined on the vertices: for vertices $u,v$, $u\leq v$ if there is a $uv$-dipath.
For every $i\in [n]$, let $v_i$ be the second smallest vertex of $Q_i$ that also lies on a dipath in $\{Q_1,\ldots,Q_n\}-\{Q_i\}$ (assuming $v_i$ exists).
Then there exists an index subset $I\subseteq [n]$ of size at least two such that, for each $i\in I$, the following hold: \begin{itemize}
\item $v_i\leq v_1$, and there is no $j\in [n]$ such that $v_j<v_i$, and
\item for each $j\in [n]$, $v_i=v_j$ if and only if $j\in I$.
\end{itemize}
\end{prp}
\begin{proof}
Suppose such an index subset does not exist. 
In particular, for any index $i\in [n]$ such that $v_i\leq v_1$, there exists $\pi(i)\in [n]-\{i\}$ such that $v_i\in V(Q_{\pi(i)})$ and $v_i> v_{\pi(i)}$.
Then one can construct the infinite chain $v_1> v_{\pi(1)} > v_{\pi(\pi(1))} > \ldots$, a contradiction as $>$ is a partial ordering on the vertices of the acyclic digraph.
\end{proof}
\begin{rem}\label{flows}
Let $(\vec{H}, \{\Omega\}, \{s,t\})$ be a directed signed graft, where $\Omega\in \delta(s)$ and $\vec{H}\setminus \Omega$ is acyclic. Suppose $E(\vec{H})$ can be written as the union of pairwise $\Omega$-disjoint edge sets $L_1,L_2,L_3$, $P_4,\ldots,P_m$ where $m\geq 3$, $L_1,L_2,L_3$ are directed minimal odd $st$-joins and $P_4,\ldots,P_m$ are even $st$-dipaths. Let $L$ be a directed minimal odd $st$-join. Then the following hold: \begin{enumerate}[\;\;(1)]
\item there exist pairwise $\Omega$-disjoint edge sets $L'_1,L'_2,L'_3,P'_4,\ldots,P'_m$ where $L'_1=L$, $L'_2,L'_3$ are directed minimal odd $st$-joins, $P'_4,\ldots,P'_m$ are even $st$-dipaths, and the number of non-simple minimal odd $st$-joins among $L'_1,L'_2,L'_3$ is equal to that of $L_1,L_2,L_3$,
\item if exactly one of $L_1,L_2,L_3$ is non-simple, then $L$ is simple if and only if $L$ is $\Omega$-disjoint from a directed odd circuit,
\item if at least two of $L_1,L_2,L_3$ are non-simple, then $L$ is $\Omega$-disjoint from a directed odd circuit.
\end{enumerate}
\end{rem}

\section{Preliminaries for non-simple bipartite $\Omega$-systems}\label{sec-setup-simple}
In this section, we lay the groundwork to prove proposition~\ref{main-result-bipartite-nonsimple}, namely that a minimal non-simple bipartite $\Omega$-system has an $F_7$ or $\widetilde{K_5}$ minor.

\subsection{Signature mates}
\begin{prp}\label{ns-signmate}
Let $((G,\Sigma,\{s,t\}), \mathcal{L}=(L_1,\ldots,L_k), m, \vec{H})$ be a non-simple bipartite $\Omega$-system.
Let $L\subseteq E(\vec{H})$ be a directed minimal odd $st$-join that is $\Omega$-disjoint from a directed circuit $C\subseteq E(\vec{H})$.
Let $B$ be a $k$-mate of $L$.
Then $B$ is not an $st$-cut and $B\cap E(\vec{H})= B\cap L$.
\end{prp}
\begin{proof}
Since $\vec{H}\setminus \Omega$ is acyclic, we can write $E(\vec{H})$ as the union of  $L'_1,L'_2,L'_3,P'_4,\ldots,P'_m$ such that, for $$\mathcal{L'}=(L'_1,L'_2,L'_3, L'_4:=P'_4\cup C_4, \ldots, L'_m:=P'_m\cup C_m, L_{m+1},\ldots,L_k),$$ $((G,\Sigma,\{s,t\}), \mathcal{L'}, m, \vec{H})$ is a non-simple bipartite $\Omega$-system, $L'_1=L$ and $C(L'_2)=C$.
By proposition~\ref{usefulparity4}, $B\subseteq L\cup L'_4\cup \cdots\cup L'_m\cup L_{m+1} \cup \cdots \cup L_k$.
Since $B\cap L'_2\neq \emptyset$ and $B\cap L'_2\subseteq \{\Omega\}$, it follows that $B\cap L'_2=\{\Omega\}$, so $B\cap C=\{\Omega\}$. Hence, $B$ is not an $st$-cut, so it is a signature.
Moreover, by proposition~\ref{bpsignature}, $B\cap E(\vec{H})= B\cap L$.
\end{proof}
\begin{prp}\label{ns-signmate2}
Let $((G,\Sigma,\{s,t\}), \mathcal{L}=(L_1,\ldots,L_k), m, \vec{H})$ be a non-simple bipartite $\Omega$-system.
Choose an even $st$-dipath $P$ of $\vec{H}$ such that $P\cup \{\Omega\}$ has a $k$-mate $B$.
Then $B$ is not an $st$-cut and $B\cap E(\vec{H})= \{\Omega\}\cup (B\cap P)$.
\end{prp}
\begin{proof}
Since $\vec{H}\setminus \Omega$ is acyclic, we can write $E(\vec{H})$ as the union of  $L'_1,L'_2,L'_3,P'_4,\ldots,P'_m$ such that, for $$\mathcal{L'}=(L'_1,L'_2,L'_3, L'_4:=P'_4\cup C_4, \ldots, L'_m:=P'_m\cup C_m, L_{m+1},\ldots,L_k),$$ $((G,\Sigma,\{s,t\}), \mathcal{L'}, m, \vec{H})$ is a non-simple bipartite $\Omega$-system and $P(L'_1)=P$.
By proposition~\ref{usefulparity4}, $B\subseteq \{\Omega\}\cup P\cup L'_4\cup \cdots\cup L'_m\cup L_{m+1} \cup \cdots \cup L_k$, and $\Omega\in B$ as $B$ intersects $L'_2$.
Then $B\cap C(L'_1) = \{\Omega\}$, implying that $B$ is not an $st$-cut, so it is a signature.
Moreover, by proposition~\ref{bpsignature} and the fact that $\Omega\in B$, it follows that $B\cap E(\vec{H})= \{\Omega\}\cup (B\cap P)$.
\end{proof}
\subsection{Two disentangling lemmas}

\begin{lma}\label{ns-disentangle}
Let $\bigl((G,\Sigma,\{s,t\}), \mathcal{L}=(L_1,\ldots,L_k), m, \vec{H}\bigr)$ be a minimal non-simple bipartite $\Omega$-system. Take disjoint subsets $I_d, I_c\subseteq E(\vec{H}\setminus \Omega)$ and $T'\subseteq \{s,t\}$ where
\begin{enumerate}[\;\;(1)]
\item $I_c$ is non-empty, if $I_c$ contains an $st$-path then $T'=\emptyset$, and if not then $T'=\{s,t\}$,
\item every signature or $st$-cut disjoint from $I_c$ intersects $I_d$ in an even number of edges,
\item if $T'=\emptyset$, there is a directed subgraph $\vec{H'}$ of $\vec{H}/I_c\setminus I_d$ that is the union of directed odd circuits $L'_1, L'_2, L'_3$ where \begin{enumerate}[\;\;]
\item $\Omega\in L'_1\cap L'_2\cap L'_3$ and $L'_1, L'_2, L'_3$ are pairwise $\Omega$-disjoint,
\item $\vec{H'}\setminus \Omega$ is acyclic.
\end{enumerate}
\item if $T'=\{s,t\}$, there is a directed subgraph $\vec{H'}$ of $\vec{H}/I_c\setminus I_d$ that is the union of directed minimal odd $st$-joins $L'_1, L'_2, L'_3$ and even $st$-dipaths $P'_4, \ldots, P'_m$, where \begin{enumerate}[\;\;]
\item $\Omega\in L'_1\cap L'_2\cap L'_3$, $\Omega\notin P'_4\cup \ldots \cup P'_m$ and
$L'_1, L'_2, L'_3,  P'_4, \ldots , P'_m$ are pairwise $\Omega$-disjoint,
\item one of $L'_1,L'_2,L'_3$ is non-simple,
\item $\vec{H'}\setminus \Omega$ is acyclic.
\end{enumerate}
\end{enumerate}
Then one of the following does not hold: 
\begin{enumerate}[\;\;(i)]
\item $I_d\cup \{\Omega\}$ does not have a $k$-mate,
\item for every directed odd $T'$-join $L'$ of $\vec{H'}$ $\Omega$-disjoint from a directed odd circuit, either $L'\cup I_d$ contains a directed odd $st$-join of $\vec{H}$ $\Omega$-disjoint from a directed odd circuit or $L'\cup I_d$ has a $k$-mate in $(G,\Sigma,\{s,t\})$ disjoint from $I_c$.
\end{enumerate} 
\end{lma}
\begin{proof}
Suppose otherwise.
Let $(G',\Sigma',T'):=(G,\Sigma,\{s,t\})/I_c\setminus I_d$ where $\Sigma'=\Sigma$; this signed graft is well-defined by (1).
For $i\in [m]-[3]$, let $L'_i:=L_i-P_i$ if $T'= \emptyset$, and let $L'_i:=(L_i-P_i) \cup P'_i$ otherwise.
Let $\mathcal{L'}:=(L'_1,\ldots,L'_m, L_{m+1},\ldots,L_k)$.
If $T'=\emptyset$, let $m':=3$, and if not, let $m':=m$.
We will show that $((G',\Sigma',T'), \mathcal{L'}, m', \vec{H'})$ is a non-simple bipartite $\Omega$-system, and this will yield a contradiction with the minimality of the original non-simple bipartite $\Omega$-system, thereby finishing the proof.

{\bf (NS1)} We first show that $((G',\Sigma',T'), \mathcal{L'}, m')$ is a bipartite $\Omega$-system.
{\bf (B1)} By (2) every signature or $T'$-cut of $(G',\Sigma',T')$ has the same parity as $\tau(G,\Sigma,\{s,t\})$, implying that $(G',\Sigma',T')$ is an Eulerian signed graft.
{\bf (B2)} It also implies that $k, \tau(G,\Sigma,\{s,t\}), \tau(G',\Sigma',T')$ have the same parity, so every minimal cover of $(G',\Sigma',T')$ has the same size parity as $k$.
We claim that $\tau(G',\Sigma',T')\geq k$.
Let $B'$ be a minimal cover of $(G',\Sigma',T')$. 
If $\Omega\notin B'$, then $$|B'|\geq \sum \left( |B'\cap L'|: L'\in \mathcal{L'}\right)\geq k.$$ Otherwise, $\Omega\in B'$. In this case, $B'\cup I_d$ contains a cover $B$ of $(G,\Sigma,\{s,t\})$. By (i), $I_d\cup \{\Omega\}$ does not have a $k$-mate, so $$k-2\leq |B-(I_d\cup \{\Omega\})|\leq |B-I_d|-1\leq |B'|-1,$$ and since $|B'|,k$ have the same parity, it follows that $|B'|\geq k$. Thus, $\mathcal{L'}$ is an $(\Omega,k)$-packing. When $T'=\emptyset$, $m'=3$. When $T'=\{s,t\}$, then $m'=m$ and for $j\in [m']-[3]$, $L'_j$ contains even $st$-path $P'_j$ and some odd circuit in $L'_j-P'_j$, and for $j\in [k]-[m']$, $L_j$ remains connected in $G'$.
{\bf (B3)} is clear from construction.

{\bf (NS2)} and {\bf (NS3)} follow from (3) and (4).
{\bf (NS4)} Let $L'$ be a directed odd $T'$-join of $\vec{H'}$ that is $\Omega$-disjoint from a directed odd circuit.
We claim that $L'\cup I_d$ has a $k$-mate in $(G,\Sigma,\{s,t\})$ disjoint from $I_c$.
By (ii), we may assume that $L'\cup I_d$ contains a directed odd $st$-join $L$ of $\vec{H}$ that is $\Omega$-disjoint from a directed odd circuit. Since $((G,\Sigma,\{s,t\}),\mathcal{L},m,\vec{H})$ is a non-simple bipartite $\Omega$-system, it follows that $L$ has a $k$-mate $B$. By proposition~\ref{ns-signmate}, $B\cap E(\vec{H})=B\cap L$, implying that $B\cap I_c=\emptyset$, as claimed.
So $B$ is a $k$-mate of $L'\cup I_d$ disjoint from $I_c$.
$B-I_d$ contains a minimal cover $B'$ of $(G',\Sigma',T')$, and since $$|B'-L'|\leq |(B-I_d)-L'|\leq |B-L|\leq k-3,$$ it follows that $B'$ is a $k$-mate of $L'$. 
\end{proof}

We will need an analogue of this lemma for the case $T=\emptyset$. As the proof is almost the same (and less intricate), we leave the proof as an exercise:

\begin{lma}\label{ns-disentangle-empty}
Let $\bigl((G,\Sigma,\emptyset), \mathcal{L}=(L_1,\ldots,L_k), 3, \vec{H}\bigr)$ be a minimal non-simple bipartite $\Omega$-system, where $\Omega\in \delta(s)$. Take disjoint subsets $I_d, I_c\subseteq E(\vec{H}\setminus \Omega)$ where
\begin{enumerate}[\;\;(1)]
\item $I_c$ is non-empty, 
\item every signature disjoint from $I_c$ intersects $I_d$ in an even number of edges,
\item there is a directed subgraph $\vec{H'}$ of $\vec{H}/I_c\setminus I_d$ that is the union of directed odd circuits $L'_1, L'_2, L'_3$ where \begin{enumerate}[\;\;]
\item $\Omega\in L'_1\cap L'_2\cap L'_3$ and $L'_1, L'_2, L'_3$ are pairwise $\Omega$-disjoint,
\item $\vec{H'}\setminus \Omega$ is acyclic.
\end{enumerate}
\end{enumerate}
Then one of the following does not hold: 
\begin{enumerate}[\;\;(i)]
\item $I_d\cup \{\Omega\}$ does not have a signature $k$-mate,
\item for every directed odd cycle $L'$ of $\vec{H'}$ $\Omega$-disjoint from a directed odd circuit, either $L'\cup I_d$ contains a directed odd cycle of $\vec{H}$ $\Omega$-disjoint from a directed odd circuit or $L'\cup I_d$ has a signature $k$-mate in $(G,\Sigma,\emptyset)$ disjoint from $I_c$.
\end{enumerate} 
\end{lma}

\subsection{Setup for the proof of proposition~\ref{main-result-bipartite-nonsimple}}

Let $((G,\Sigma,T), \mathcal{L}=(L_1,\ldots,L_k), m, \vec{H})$ be a minimal non-simple bipartite $\Omega$-system. We know that $\vec{H}\setminus \Omega$ is acyclic, and by (B3), every odd circuit in $\vec{H}$ contains $\Omega$ and no even $st$-path in $\vec{H}$ contains $\Omega$. Hence,

\begin{rem}\label{ns-disjoint}
Let $C$ be a directed odd circuit and let $P$ be an even $st$-dipath in $\vec{H}$. Then $C$ and $P$ share exactly one vertex, namely $s$.
\end{rem}

There are three possibilities: \begin{enumerate}[\;\;I:]
\item all three of $L_1,L_2,L_3$ are non-simple (see \S\ref{sec-NS-I}),
\item exactly two of $L_1,L_2,L_3$ are non-simple (see \S\ref{sec-NS-II}),
\item exactly one of $L_1,L_2,L_3$ is non-simple (see \S\ref{sec-NS-III}).
\end{enumerate}

\noindent We will assume throughout this section that $\Omega$ has ends $s,s'$.
\section{Non-simple bipartite $\Omega$-system - part I}\label{sec-NS-I}
Here we prove proposition~\ref{main-result-bipartite-nonsimple} when all of $L_1,L_2,L_3$ are non-simple. By remark~\ref{ns-disjoint}, for $i\in [3]$ and $j\in [m]$, $C_i$ and $P_j$ share exactly one vertex, namely $s$.

\begin{claims} 
There exists $j\in [m]$ such that $P_j\cup \{\Omega\}$ has no $k$-mate. 
\end{claims}
\begin{cproof}
Suppose otherwise. 
Then $T=\{s,t\}$, as $\tau(G,\Sigma,T)\geq k$ (so $\{\Omega\}$ has no $k$-mate).
Then by the mate proposition~\ref{mateprop} there exists $i\in [m]$ such that the $k$-mate of $P_i\cup \{\Omega\}$ is an $st$-cut, contradicting proposition~\ref{ns-signmate2}.
\end{cproof}

By swapping the roles of $P_1$ and $P_j$ in $\mathcal{L}$, if necessary, we may assume that $j=1$.

\begin{claims} 
$T=\emptyset$. 
\end{claims}
\begin{cproof}
Suppose for a contradiction that $T=\{s,t\}$.
Let $I_d:=P_1$ and $I_c:= P_2\cup \ldots \cup P_m$.
Let $T':=\emptyset$, and for $j\in [3]$ let $L'_j:=C_j$, and let $\vec{H'}\subseteq \vec{H}\setminus I_d/I_c$ be the union of $L'_1,L'_2,L'_3$. It is clear that (1)-(4) of the disentangling lemma~\ref{ns-disentangle} hold.
By claim~1, $P_1\cup \{\Omega\}=I_d\cup \{\Omega\}$ has no $k$-mate, so (i) holds.
Let $L'$ be a directed odd cycle of $\vec{H'}$.
Then it is clear that $L'\cup I_d$ contains a directed minimal odd $st$-join $L$ of $\vec{H}$ such that $L'\subseteq L$. By remark~\ref{flows}(3), $L$ and so $L'$ is $\Omega$-disjoint from a directed odd circuit, and since $I_d$ is $\Omega$-disjoint from every directed odd circuit by remark~\ref{ns-disjoint}, we get that $L'\cup I_d$ is $\Omega$-disjoint from a directed odd circuit, so (ii) holds as well, a contradiction with the disentangling lemma~\ref{ns-disentangle}.~\end{cproof}

The rest of this part is dedicated to finding a $\widetilde{K_5}$ minor in $(G,\Sigma,T=\emptyset)$, and our arguments are very similar to the treatment of Geelen and Guenin~\cite{Geelen02}, except for our use of Menger's theorem in claim~4.

We may assume that in $\vec{H}$, $\Omega$ is directed from $s$ to $s'$, and for $i\in [3]$, $L_i-\{\Omega\}$ is an $s's$-dipath.
Consider the following partial ordering defined on the vertices of $\vec{H}$: for $u,v\in V(\vec{H})$, $u\leq v$ if there is a $uv$-dipath in $\vec{H}\setminus \Omega$; this partial ordering is well-defined as $\vec{H}\setminus \Omega$, by (NS3). 
For each $i\in [3]$, let $v_i$ be the second smallest vertex of $L_i-\{\Omega\}$ that lies on a dipath in $\{L_1, L_2, L_3\}-\{L_i\}$
By proposition~\ref{intersection}, there exists an index subset $I\subseteq [3]$ of size at least two such that, for each $i\in I$ and $j\in [3]$, $v_j=v_i$ if and only if $j\in I$. We may assume that $1\in I$.

\begin{claims} 
For each $i\in I$, $L_i[s',v_i]\cup \{\Omega\}$ has a signature $k$-mate.
\end{claims}
\begin{cproof}
Suppose otherwise. Let $I_d:=L_i[s',v_i]$ and $I_c:=\bigcup (L_j[s',v_j]:j\in I, j\neq i)$. For $i\in [3]$ let $L'_i:=L_i-(I_c\cup I_d)$, and let $\vec{H'}\subseteq \vec{H}\setminus I_d/I_c$ be the union of $L'_1,L'_2,L'_3$.
It is easily seen that (1)-(3) of the disentangling lemma~\ref{ns-disentangle-empty} hold. By our hypothesis, (i) holds.
Let $L'$ be a directed odd cycle of $\vec{H'}$. Then $L'\cup I_c$ contains a directed odd circuit of $\vec{H}$, implying that $L'\cup I_d$ also contains a directed odd circuit of $\vec{H}$, which by remark~\ref{flows}(3) is $\Omega$-disjoint from a directed odd circuit. 
Hence, (ii) holds as well, a contradiction to the disentangling lemma~\ref{ns-disentangle-empty}.
\end{cproof}

\begin{claims} 
There exist an $s's$-dipath $P$ and a $v_1s$-dipath $Q$ in $\vec{H}\setminus \{\Omega\}$ that are internally vertex-disjoint.
\end{claims}
\begin{cproof}
Suppose otherwise. Then $s\neq v_1$ and there exists a vertex $v\in V(\vec{H})- \{s,s'\}$ such that there is no $s's$-dipath in $\vec{H}\setminus v$. 
One of the following holds: \begin{enumerate}[\;\;(a)]
\item there exists an $s'v$-dipath $R$ in $\vec{H}$ such that $R\cup \{\Omega\}$ has no $k$-mate: 
\begin{quote}
Let $I_d:=R$, $I_c:= \bigcup (L_i[s',v]: i\in [3]) - R$, for $i\in [3]$ let $L'_i:=L_i-(I_c\cup I_d)$,
and let $\vec{H'}\subseteq \vec{H}\setminus I_d/I_c$ be the union of $L'_1,L'_2,L'_3$. 
\end{quote}

\item for every $s'v$-dipath $R$ in $\vec{H}$, $R\cup \{\Omega\}$ has a (signature) $k$-mate: 
\begin{quote}
Let $I_d:=\emptyset$, $I_c:= \bigcup (L_i[v,s]: i\in [3])$,
for $i\in [3]$ let $L'_i:=L_i[s',v]\cup \{\Omega\}$,
and let $\vec{H'}\subseteq \vec{H}\setminus I_d/I_c$ be the union of $L'_1,L'_2,L'_3$. 
\end{quote}
\end{enumerate}
It is not difficult to check that in either of the cases above, (1)-(3) and (i)-(ii) of the disentangling lemma~\ref{ns-disentangle-empty} hold, a contradiction.
\end{cproof}

After redefining $\mathcal{L}$, if necessary, we may assume that $\{1,2\}\subseteq I$ and $P=L_3-\{\Omega\}$.

\begin{claims} 
$(L_i-\{\Omega\}:i\in [3])$ are pairwise internally vertex-disjoint.
\end{claims}
\begin{cproof}
It suffices to prove that $Q=\emptyset$. Suppose not. Let $I_c:=Q$, $I_d:=\emptyset$, for $i\in [2]$ let $L'_i:=L_i[s',v_i]\cup \{\Omega\}$, and let $L'_3:=L_3$. Let $\vec{H'}\subseteq \vec{H}\setminus I_d/I_c$ be the union of $L'_1,L'_2,L'_3$.
Note that $(L'_i-\{\Omega\}: i\in [3])$ are pairwise internally vertex-disjoint.
By our hypothesis, claim~3, (NS4), and proposition~\ref{bpsignature}, (1)-(3) and (i)-(ii) of the disentangling lemma~\ref{ns-disentangle-empty} hold, a contradiction.~\end{cproof}

\begin{claims} 
$(G,\Sigma,T=\emptyset)$ has a $\widetilde{K_5}$ minor.
\end{claims}
\begin{cproof}
By identifying a vertex of each component with $s$, if necessary, we may assume that $G$ is connected.
By (NS4), for each $i\in [3]$, there exists a signature $k$-mate $B_i$ of $L_i$. For distinct $i,j\in [3]$, let $U_{ij}\subseteq V(G)-\{s,t\}$ such that $\delta(U_{ij})=B_i\triangle B_j$; by proposition~\ref{matessignsign}, there exists a shortest path $P_{ij}$ between $L_i$ and $L_j$ in $G[U_{ij}]\setminus (B_i\cup B_j)$.
To finish proving the claim, we will use the odd-$K_5$ lemma~\ref{oddK5} to prove that $L_1\cup L_2\cup L_3\cup P_{12}\cup P_{23}\cup P_{31}$ has a $\widetilde{K_5}$ minor.

Observe that $$\emptyset = (B_1\triangle B_2)\triangle (B_2\triangle B_3)\triangle (B_3\triangle B_1)=\delta(U_{12})\triangle \delta(U_{23})\triangle \delta(U_{31})=\delta(U_{12}\triangle U_{23}\triangle U_{31}),$$ implying that $U_{12}\triangle U_{23}\triangle U_{31}$ is either $\emptyset$ or $V(G)$, as $G$ is connected. However, $s,t\notin U_{12}\triangle U_{23}\triangle U_{31}$, implying that $U_{12}\triangle U_{23}\triangle U_{31}=\emptyset$. As a result, there exist pairwise disjoint subsets $U_1,U_2,U_3\subseteq V(G)$ such that, for distinct $i,j\in [3]$, $U_{ij}=U_i\cup U_j$. Let $U_0:=V(G)-(U_1\cup U_2\cup U_3)$. Since $L_1\cap (B_2\cup B_3)=\{\Omega\}$, it follows that $L_1\cap \delta(U_{23})=\emptyset$, and since $L_1$ is connected, it must be that $V(L_1)\subseteq U_0\cup U_1$. Similarly, $V(L_2)\subseteq U_0\cup U_2$ and $V(L_3)\subseteq U_0\cup U_3$.
Let $B:=B_1\triangle B_2\triangle B_3$, which is a signature for $(G,\Sigma,T)$.
Observe that the edges in $B$ are precisely those with ends in different sets among $U_0,U_1,U_2,U_3$.
Now contract all the edges of $G$ not in $B$ and apply the odd-$K_5$ lemma~\ref{oddK5} to conclude that $L_1\cup L_2\cup L_3\cup P_{12}\cup P_{23}\cup P_{31}$, and in turn $(G,\Sigma,T)$, has a $\widetilde{K_5}$ minor.
\end{cproof}

\section{Non-simple bipartite $\Omega$-system - part II}\label{sec-NS-II}
Here we prove proposition~\ref{main-result-bipartite-nonsimple} when exactly two of $L_1,L_2,L_3$, say $L_1$ and $L_2$, are non-simple.
Observe that $T\neq \emptyset$.
Recall that $T=\{s, t\}$ and $\Omega$ has ends $s,s'$. 

\begin{claims} 
There exists $j\in [m]-\{3\}$ such that $P_j\cup \{\Omega\}$ has no $k$-mate. 
\end{claims}
\begin{cproof}
Suppose otherwise. 
As $P_3$ is a directed odd $st$-join of $\vec{H}$ that is $\Omega$-disjoint from directed odd circuit $C_1$, it has a $k$-mate.
Thus by the mate proposition~\ref{mateprop} there exists $i\in [m]$ such that the $k$-mate of $P_i\cup \{\Omega\}$ is an $st$-cut, contradicting propositions~\ref{ns-signmate} and~\ref{ns-signmate2}.
\end{cproof}

\noindent By swapping the roles of $P_1$ and $P_j$ in $\mathcal{L}$, if necessary, we may assume that \mbox{$j=1$}.
Observe that $P_1\cup \cdots \cup P_m$ is acyclic, as $\vec{H}\setminus \Omega$ is so.
Consider the following partial ordering: for $u,v\in V(P_1\cup \cdots \cup P_m)$, $u\leq v$ if there is a $uv$-dipath in $P_1\cup \cdots \cup P_m$. For $i\in [m]$ let $v_i$ be the second largest vertex of $P_i$ that lies on another $st$-dipath in \mbox{$\{P_1,\ldots,P_m\}-\{P_i\}$}.

\begin{claims} 
$s<v_3$.
\end{claims}
\begin{cproof}
Suppose otherwise. 
In other words, $P_3$ is internally vertex-disjoint from each one of $P_1,P_2,$ $P_4,\ldots,P_m$. 
Let $I_d:=P_1$ and $I_c:= P_2\cup P_4\cup P_5\cup \ldots \cup P_m$.
Let $T':=\emptyset$, for $j\in [2]$ let $L'_j:=C_j$, let $L'_3:=P_3$, and let $\vec{H'}\subseteq \vec{H}\setminus I_d/I_c$ be the union of $L'_1,L'_2,L'_3$. It is clear that (1)-(4) of the disentangling lemma~\ref{ns-disentangle} hold.
By claim~1, $P_1\cup \{\Omega\}=I_d\cup \{\Omega\}$ has no $k$-mate, so (i) holds.
Let $L'$ be a directed odd cycle of $\vec{H'}$.
Then it is clear that $L'\cup I_d$ contains a directed minimal odd $st$-join $L$ of $\vec{H}$ such that $L'\subseteq L$. By remark~\ref{flows}(3), $L$ is $\Omega$-disjoint from a directed odd circuit, so by remark~\ref{ns-disjoint}, $L'\cup I_d$ is $\Omega$-disjoint from a directed odd circuit, implying in turn that (ii) holds, a contradiction with the disentangling lemma~\ref{ns-disentangle}.
\end{cproof}

By proposition~\ref{intersection} there exists an index subset $I\subseteq [m]$ of size at least two such that, for each $i\in I$, \begin{itemize}
\item $v_i\geq v_3$, and there is no $j\in [m]$ such that $v_j>v_i$,
\item for each $j\in [m]$, $v_i=v_j$ if and only if $j\in I$.
\end{itemize} For $i\in I$, since $v_i\geq v_3>s$ by claim~2, $P_i[v_i,t]$ is contained in an odd $st$-dipath of $\vec{H}$, and since $I\cap ([m]-\{3\})\neq \emptyset$, $P_i[v_i,t]$ is also contained in an even $st$-dipath of $\vec{H}$

\begin{claims} 
For each $i\in I$ and $j\in [2]$, $P_i[v_i,t]$ and $C_j$ have no vertex in common.
\end{claims}
\begin{cproof}
Since $P_i[v_i,t]$ is contained in an even $st$-dipath of $\vec{H}$, the claim follows from remark~\ref{ns-disjoint} and the fact that $v_i> s$.
\end{cproof}

\noindent As a result, for each $i\in I$, the internal vertices of $P_i[v_i,t]$ have degree two in $\vec{H}$.

\begin{claims} 
For each $i\in I$, $P_i[v_i,t]\cup \{\Omega\}$ has a $k$-mate. In particular, $1\notin I$.
\end{claims}
\begin{cproof}
Suppose otherwise. 
Let $I_d:=P_i[v_i,t]$ and $I_c:= \bigcup (P_j[v_j,t]: j\in I-\{i\})$.
Let $T':=\{s,t\}$,
for $j\in [3]$ let $L'_j:=L_j-(I_c\cup I_d)$,
and for $j\in [m]-[3]$ let $P'_j:=P_j-(I_c\cup I_d)$. 
Let $\vec{H'}\subseteq \vec{H}\setminus I_d/I_c$ be the union of $L'_1,L'_2,L'_3, P'_4,\ldots,P'_m$. It is clear that (1)-(4) of the disentangling lemma~\ref{ns-disentangle} hold.
By assumption, $I_d\cup \{\Omega\}$ has no $k$-mate, so (i) holds.
Let $L'$ be a directed odd $st$-join of $\vec{H'}$.
Then it is clear that $L'\cup I_d$ contains a directed minimal odd $st$-join $L$ of $\vec{H}$ such that $L'\subseteq L$. By remark~\ref{flows}(3), $L$ is $\Omega$-disjoint from a directed odd circuit, so by remark~\ref{ns-disjoint}, $L'\cup I_d$ is also $\Omega$-disjoint from a directed odd circuit, so (ii) holds as well, a contradiction with the disentangling lemma~\ref{ns-disentangle}.
\end{cproof}

\begin{claims} 
Fix $i\in I$. Then there exists an $s'v_i$-dipath in $\vec{H}\setminus (C_1\cup C_2)$ that is vertex-disjoint from $P_1$.
\end{claims}
\begin{cproof}
Let $v$ be the smallest vertex on $P_1$ for which there exists a $vv_i$-dipath $R$ in $\vec{H}\setminus \Omega$ such that $V(R)\cap V(P_1)=\{v\}$. Since $R$ is contained in an even $st$-dipath, namely $P_1[s,v]\cup R\cup P_i[v_i,t]$, it follows from remark~\ref{ns-disjoint} that $R$ and $C_1\cup C_2$ have at most one vertex in common, namely $s$. 
Our choice of $v$ and $R$ implies the following: \begin{quote} $(\star)$ if $w\in V(R)$ and $Q$ is an $sw$-dipath in $\vec{H}\setminus \Omega$, then $Q$ and $P_1[v,t]$ have a vertex in common.\end{quote}
Suppose for a contradiction that there is no $s'v_i$-dipath in $\vec{H}\setminus (C_1\cup C_2)$ that is vertex-disjoint from $P_1$. This fact, together with $(\star)$ and remark~\ref{ns-disjoint}, implies the following: \begin{quote} $(\star\star)$ if $w\in V(R)$ and $Q$ is an $s'w$-dipath in $\vec{H}$, then $Q$ and $P_1[v,t]$ have a vertex in common.\end{quote}

Let $I_d:=P_1[v,t]$ and $I_c:=R\cup \left[\bigcup (P_j[v_j,t]: j\in I)\right]$.
For $i\in [3]$ let $L'_i$ be $L_i-(I_c\cup I_d)$ minus any directed circuit that does not use $\Omega$, and for $i\in [m]-[3]$ let $P'_i$ be $P_i-(I_c\cup I_d)$ minus any directed circuit that does not use $\Omega$.
If $v=s$, let $T':=\emptyset$ and $\vec{H'}\subseteq \vec{H}\setminus I_d/I_c$ be the union of $L'_1,L'_2,L'_3$.
Otherwise, when $v\neq s$, let $T':=\{s,t\}$ and $\vec{H'}\subseteq \vec{H}\setminus I_d/I_c$ be the union of $L'_1,L'_2,L'_3, P'_4,\ldots,P'_m$. 
It is not hard to see that (1)-(4) of the disentangling lemma~\ref{ns-disentangle} hold.
By claim~1, $P_1[v,t]\cup \{\Omega\}=I_d\cup \{\Omega\}$ has no $k$-mate, so (i) holds.
Let $L'$ be a directed odd $T'$-join of $\vec{H'}$.
Then $L'\cup I_c$ contains a directed odd $st$-join $L$ of $\vec{H}$ such that $L'\subseteq L$.
Choose $w\in V(R)$ (if any) such that $L$ contains an $s'w$-dipath $Q$ in $\vec{H}$ and $V(Q)\cap V(R)=\{w\}$. Then $(\star\star)$ implies that $(L-I_c)\cup I_d$, and therefore $L'\cup I_d$, contains a directed minimal odd $st$-join of $\vec{H}$.
By remark~\ref{flows}(3), $L$ is $\Omega$-disjoint from a directed odd circuit, so by remark~\ref{ns-disjoint}, $L'\cup I_d$ is $\Omega$-disjoint from a directed odd circuit, and so (ii) holds as well, a contradiction with the disentangling lemma~\ref{ns-disentangle}.
\end{cproof}

After redefining $\mathcal{L}$, if necessary, we may assume that $3\in I$ and that $P_3[s',v_3]$ is vertex-disjoint from $P_1$. (See remark~\ref{flows}(1).)

\begin{claims} 
$(G,\Sigma,\{s,t\})$ has an $F_7$ minor.
\end{claims}
\begin{cproof}
For $i\in I$ let $B_i$ be a $k$-mate of $P_i[v_i,t]\cup \{\Omega\}$, whose existence is guaranteed by claim~4. 
For each $i\in I$, since $B_i$ is also a $k$-mate for odd $st$-dipath $P_3[s,v_3]\cup P_i[v_i,t]$, proposition~\ref{ns-signmate} implies that $B_i$ is a signature.
Take $j\in I-\{3\}$.
Choose $U\subseteq V(G)-\{s,t\}$ such that $B_3\triangle B_j=\delta(U)$. Then by proposition~\ref{matessignsign} there exists a path $P$ in $G[U]$ between $V(P_3[v_3,t])\cap U$ and $V(P_j[v_j,t])\cap U$ such that $P\cap (B_3\cup B_j)=\emptyset$, and $P$ is minimal subject to this property.
Observe that $L_1\cup P_3[s',v_3]$ has no vertex in common with $U$.
It is easy (and is left as an exercise) to see that $C_1\cup P_1\cup P_3\cup P_j[v_j,t]\cup P$ has an $F_7$ minor.
\end{cproof}

%
\section{Non-simple bipartite $\Omega$-system - part III}\label{sec-NS-III}
%
Here we prove proposition~\ref{main-result-bipartite-nonsimple} when exactly one of $L_1,L_2,L_3$, say $L_1$, is non-simple. This will complete the proof of proposition~\ref{main-result-bipartite-nonsimple}. Observe that $T\neq \emptyset$, so $T=\{s,t\}$, and recall that $\Omega$ has ends $s,s'$. 

Observe that $P_1\cup \cdots \cup P_m$ is acyclic, as $\vec{H}\setminus \Omega$ is so.
Consider the following partial ordering: for $u,v\in V(P_1\cup \cdots \cup P_m)$, $u\leq v$ if there is a $uv$-dipath in $P_1\cup \cdots \cup P_m$. For $i\in [m]$ let $v_i$ be the second largest vertex of $P_i$ that lies on another $st$-dipath in \mbox{$\{P_1,\ldots,P_m\}-\{P_i\}$}.
By proposition~\ref{intersection} there exists an index subset $I\subseteq [m]$ of size at least two such that, for each $i\in I$, \begin{itemize}
\item $v_i\geq v_3$, and there is no $j\in [m]$ such that $v_j>v_i$,
\item for each $j\in [m]$, $v_i=v_j$ if and only if $j\in I$.
\end{itemize} 

\begin{claims} 
For each $i\in I$, $C_1$ and $P_i[v_i,t]$ have no vertex of $V(G)-\{s'\}$ in common.
\end{claims}
\begin{cproof}
Suppose otherwise.
Then it follows from remark~\ref{ns-disjoint} that $$
(\diamond)\quad I=\{2,3\} \quad\text{and}\quad V(P_i)\cap V(P_j)=\{s,t\} \quad \forall~i\in I, \forall~j\in [m]-I.$$
Let $Q_1:=C_1-\{\Omega\}, Q_2:=P_2-\{\Omega\}$ and $Q_3:=P_3-\{\Omega\}$.
For each $i\in [3]$, let $u_i$ be the second smallest ({\it not} largest) vertex of $Q_i$ that also lies on one of $\{Q_1,Q_2,Q_3\}-\{Q_i\}$.
Then by proposition~\ref{intersection}, there exists an index subset $J$ of $\{1,2,3\}$ of size at least two such that, for each $j\in J$ and $i\in [3]$, $u_i=u_j$ if and only if $i\in J$.

\begin{subclaim} 
For each $j\in J$, $Q_j[s',u_j]\cup \{\Omega\}$ has a $k$-mate.
\end{subclaim}

\begin{subproof} 
Suppose otherwise. 
Let $I_d:=Q_j[s',u_j]$ and $I_c:= \bigcup (Q_i[s',u_i]: i\in J-\{j\})$.
Let $T':=\{s,t\}$,
and for $i\in [3]$, let $L'_i:=L_i-(I_c\cup I_d)$.
Let $\vec{H'}\subseteq \vec{H}\setminus I_d/I_c$ be the union of $L'_1,L'_2,L'_3, P_4,\ldots,P_m$. It is clear that (1)-(4) of the disentangling lemma~\ref{ns-disentangle} hold.
By assumption, $I_d\cup \{\Omega\}$ has no $k$-mate, so (i) holds.
Let $L'$ be a directed odd $st$-join of $\vec{H'}$ that is $\Omega$-disjoint from a directed odd circuit, i.e. $L'$ is an odd $st$-dipath of $\vec{H'}$ by remark~\ref{flows}(2).
Then it is clear that $L'\cup I_d$ contains an odd $st$-dipath of $\vec{H}$, which by remark~\ref{flows}(2) is $\Omega$-disjoint from a directed odd circuit, so (ii) holds as well, a contradiction with the disentangling lemma~\ref{ns-disentangle}.
\end{subproof}

\begin{subclaim} 
Fix $j\in J$. Then there exist an $s't$-dipath $P$ and a $u_jt$-dipath $Q$ in $\vec{H}\setminus s$ that are internally vertex-disjoint.
\end{subclaim}

\begin{subproof}
Suppose otherwise. Then by Menger's theorem there exists a vertex $v\in V(\vec{H}\setminus s)- \{s', t\}$ such that there is no $s't$-dipath in $\vec{H}\setminus \{s,v\}$. 
Note that $v\in V(C_1)$, since $C_1$ and $P_2[v_2,t]$ have a vertex in common.
One of the following holds: \begin{enumerate}[\;\;(a)]
\item there exists an $s'v$-dipath $R$ in $\vec{H}\setminus s$ such that $R\cup \{\Omega\}$ has no $k$-mate: 
\begin{quote}
Let $I_d:=R$, $I_c:= \bigcup (Q_i[s',v]: i\in [3]) - R$, $T':=\{s,t\}$,
for $i\in [3]$ let $L'_i:=L_i-(I_c\cup I_d)$,
and let $\vec{H'}\subseteq \vec{H}\setminus I_d/I_c$ be the union of $L'_1,L'_2,L'_3, P_4,\ldots,P_m$. 
\end{quote}
\item for every $s'v$-dipath $R$ in $\vec{H}\setminus s$, $R\cup \{\Omega\}$ has a $k$-mate: 
\begin{quote}
Let $I_d:=\emptyset$, $I_c:= P_1\cup P_2[v,t]\cup P_3[v,t]\cup P_4\cup \cdots \cup P_m$, $T':=\emptyset$,
for $i\in [3]$ let $L'_i:=Q_i[s',v]\cup \{\Omega\}$,
and let $\vec{H'}\subseteq \vec{H}\setminus I_d/I_c$ be the union of $L'_1,L'_2,L'_3$. 
\end{quote}
\end{enumerate}

It is not difficult to check that in either of the cases above, (1)-(4) and (i), (ii) of the disentangling lemma~\ref{ns-disentangle} hold, which is the desired contradiction.
\end{subproof}

Together with $(\diamond)$, subclaim~2 implies that $J\neq \{1,2,3\}$, so because $|J|\geq 2$, we get that $|J|=2$. We may assume that $J=\{1,2\}$. Let $I_d:=\emptyset$, $I_c:= P_1\cup Q\cup P_4\cup \cdots \cup P_m$, $T':=\emptyset$,
$L'_1:=Q_1[s',u_1]\cup \{\Omega\}$, $L'_2:=Q_2[s',u_2]\cup \{\Omega\}$, $L'_3:=P\cup \{\Omega\}$,
and let $\vec{H'}\subseteq \vec{H}\setminus I_d/I_c$ be the union of $L'_1,L'_2,L'_3$.
It is not difficult to check that (1)-(4) and (i), (ii) of the disentangling lemma~\ref{ns-disentangle} hold, which is a contradiction.
\end{cproof}

\begin{claims} 
There exists $j\in [m]-\{2,3\}$ such that $P_j\cup \{\Omega\}$ has no $k$-mate. 
\end{claims}
\begin{cproof}
Suppose otherwise. 
Observe that $P_2, P_3$, being odd $st$-dipaths in $\vec{H}$ $\Omega$-disjoint from $C_1$, have $k$-mates.
Thus by the mate proposition~\ref{mateprop} there exists $i\in [m]$ such that the $k$-mate of $P_i\cup \{\Omega\}$ is an $st$-cut, contradicting propositions~\ref{ns-signmate} and~\ref{ns-signmate2}.
\end{cproof}

\noindent By swapping the roles of $P_1$ and $P_j$ in $\mathcal{L}$, if necessary, we may assume that $j=1$.

\begin{claims} 
For each $i\in I$, $P_i[v_i,t]\cup \{\Omega\}$ has a $k$-mate. In particular, $1\notin I$.
\end{claims}
\begin{cproof}
Suppose otherwise. 
Let $I_d:=P_i[v_i,t]$ and $I_c:= \bigcup (P_j[v_j,t]: j\in I-\{i\})$.
Let $T':=\{s,t\}$,
for $j\in [3]$ let $L'_j:=L_j-(I_c\cup I_d)$,
and for $j\in [m]-[3]$ let $P'_j:=P_j-(I_c\cup I_d)$. 
Let $\vec{H'}\subseteq \vec{H}\setminus I_d/I_c$ be the union of $L'_1,L'_2,L'_3, P'_4,\ldots,P'_m$. It is clear that (1)-(4) of the disentangling lemma~\ref{ns-disentangle} hold.
By assumption, $I_d\cup \{\Omega\}$ has no $k$-mate, so (i) holds.
Let $L'$ be a directed odd $st$-join of $\vec{H'}$ that is $\Omega$-disjoint from a directed odd circuit, i.e. $L'$ is an odd $st$-dipath of $\vec{H'}$ by remark~\ref{flows}(2).
Then it is clear that $L'\cup I_d$ contains an odd $st$-dipath of $\vec{H}$, which by remark~\ref{flows}(2) is $\Omega$-disjoint from a directed odd circuit, so (ii) holds as well, a contradiction with the disentangling lemma~\ref{ns-disentangle}.
\end{cproof}

\begin{claims} 
Fix $i\in I$. Then there exists an $s'v_i$-dipath in $\vec{H}\setminus C_1$ that is vertex-disjoint from~$P_1$.
\end{claims}
\begin{cproof}
Let $v$ be the second smallest vertex in $P_1$ for which there exists a $vv_i$-dipath $R$ in $\vec{H}$ such that $V(R)\cap V(P_1)=\{v\}$. Since $R$ is contained in an even $st$-dipath, namely $P_1[s,v]\cup R\cup P_i[v_i,t]$, it follows from remark~\ref{ns-disjoint} that $R$ and $C_1$ have no vertex in common.
Suppose for a contradiction that there is no $s'v_i$-dipath in $\vec{H}\setminus C_1$ that is vertex-disjoint from $P_1$. This fact, together with our choice of $v$ and $R$, implies the following: \begin{quote} $(\star)$ if $w\in V(R)$ and $Q$ is an $s'w$-dipath in $\vec{H}\setminus s$, then $Q$ and $P_1[v,t]$ have a vertex in common.\end{quote}

Let $I_d:=P_1[v,t]$ and $I_c:=R\cup \left[\bigcup (P_j[v_j,t]: j\in I)\right]$. 
For $i\in [3]$ let $L'_i$ be $L_i-(I_c\cup I_d)$ minus any directed circuit that does not use $\Omega$, and for $i\in [m]-[3]$ let $P'_i$ be $P_i-(I_c\cup I_d)$ minus any directed circuit that does not use $\Omega$.
Let $T':=\{s,t\}$ and $\vec{H'}\subseteq \vec{H}\setminus I_d/I_c$ be the union of $L'_1,L'_2,L'_3, P'_4,\ldots,P'_m$. 
It is not hard to see that (1)-(4) of the disentangling lemma~\ref{ns-disentangle} hold.
By claim~2, $I_d\cup \{\Omega\}$ has no $k$-mate, so (i) holds.
Let $L'$ be a directed odd $T'$-join of $\vec{H'}$ that is $\Omega$-disjoint from a directed odd circuit, i.e. $L'$ is an odd $st$-dipath of $\vec{H'}$ by remark~\ref{flows}(2).
Then $L'\cup I_c$ contains an odd $st$-dipath $L$ of $\vec{H}$.
Choose $w\in V(R)$ (if any) such that $L$ contains an $s'w$-dipath $Q$ in $\vec{H}$ and $V(Q)\cap V(R)=\{w\}$. Then $(\star)$ implies that $(L-I_c)\cup I_d$, and therefore $L'\cup I_d$, contains an odd $st$-dipath of $\vec{H}$, which by remark~\ref{flows}(2) is $\Omega$-disjoint from a directed odd circuit, so (ii) holds as well, a contradiction with the disentangling lemma~\ref{ns-disentangle}.~\end{cproof}

After redefining $\mathcal{L}$, if necessary, we may assume that $3\in I$ and that $P_3[s',v_3]$ is vertex-disjoint from $P_1$. (See remark~\ref{flows}(1).)

\begin{claims} 
$(G,\Sigma, \{s,t\})$ has an $F_7$ minor.
\end{claims}
\begin{cproof}
For $i\in I$, let $B_i$ be a $k$-mate of $P_i[v_i,t]\cup \{\Omega\}$, whose existence is guaranteed by claim~3. 
For each $i\in I$, since $B_i$ is also a $k$-mate for odd $st$-dipath $P_3[s,v_3]\cup P_i[v_i,t]$, proposition~\ref{ns-signmate} implies that $B_i$ is a signature.
Take $j\in I-\{3\}$.
Choose $U\subseteq V(G)-\{s,t\}$ such that $B_3\triangle B_j=\delta(U)$. Then by proposition~\ref{matessignsign} there exists a path $P$ in $G[U]$ between $V(P_3[v_3,t])\cap U$ and $V(P_j[v_j,t])\cap U$ such that $P\cap (B_3\cup B_j)=\emptyset$, and $P$ is minimal subject to this property.
Observe that $L_1\cup P_3[s',v_3]$ has no vertex in common with $U$.
It is easy (and is left as an exercise) to see that $C_1\cup P_1\cup P_3\cup P_j[v_j,t]\cup P$ has an $F_7$ minor.
\end{cproof}
%
               
\section{A preliminary for simple bipartite and cut $\Omega$-systems: the linkage lemma}\label{sec-linkage}
%
The presentation of this section follows Thomassen~\cite{Thomassen80}.
Let $H_0$ be a plane graph such that the unbounded face is bounded by a circuit $C_0$ on four vertices $s_1,s_2,t_1,t_2$, in this cyclic order.
Suppose further that every other face is bounded by a triangle, and every triangle is a facial circuit.
For each triangle $\Delta$ of $H_0$ we add $K^\Delta$, a possibly empty complete graph disjoint from $H_0$, and we join all its vertices to all the vertices of $\Delta$.
The resulting graph is called an {\it $(s_1,s_2,t_1,t_2)$-web} with {\it frame $C_0$} and {\it rib $H_0$}.

\begin{lma}[\cite{Seymour80,Thomassen80}]\label{linkage}
Let $H$ be a graph and take four distinct vertices $s_1, s_2, t_1, t_2$. 
Suppose there are no two vertex-disjoint paths $P_1,P_2$ such that, for $i=1,2$, $P_i$ is an $s_it_i$-path. Then $H$ is a spanning subgraph of an $(s_1,s_2,t_1,t_2)$-web.
%
\end{lma}

\section{Simple bipartite $\Omega$-system of flavour (SF1)}\label{sec-SF1}

In this section, we prove proposition~\ref{prp-SF1}.

\subsection{A disentangling lemma}

\begin{lma}\label{SF1-disentangle}
Let $\bigl((G,\Sigma,\{s,t\}), \mathcal{L}=(L_1,\ldots,L_k), m, \vec{H}\bigr)$ be a minimal simple bipartite $\Omega$-system of flavour (SF1), where $\Omega\in \delta(s)$, and assume there is no non-simple bipartite $\Omega$-system whose associated signed graft is a minor of $(G,\Sigma,\{s,t\})$. Take disjoint subsets $I_d, I_c\subseteq E(\vec{H}\setminus \Omega)$ and $T'\subseteq \{s,t\}$ where
\begin{enumerate}[\;\;(1)]
\item $I_c$ is non-empty, if $I_c$ contains an $st$-path then $T'=\emptyset$, and if not then $T'=\{s,t\}$,
\item every signature or $st$-cut disjoint from $I_c$ intersects $I_d$ in an even number of edges,
\item if $T'=\emptyset$, there is a directed subgraph $\vec{H'}$ of $\vec{H}/I_c\setminus I_d$ that is the union of directed odd circuits $L'_1, L'_2, L'_3$ where \begin{enumerate}[\;\;]
\item $\Omega\in L'_1\cap L'_2\cap L'_3$ and $L'_1, L'_2, L'_3$ are pairwise $\Omega$-disjoint,
\item $\vec{H'}\setminus \Omega$ is acyclic,
\end{enumerate}
\item if $T'=\{s,t\}$, there is a directed subgraph $\vec{H'}$ of $\vec{H}/I_c\setminus I_d$ that is the union of odd $st$-dipaths $L'_1, L'_2, L'_3$ and even $st$-dipaths $P'_4, \ldots, P'_m$, where \begin{enumerate}[\;\;]
\item $\Omega\in L'_1\cap L'_2\cap L'_3$, $\Omega\notin P'_4\cup \ldots \cup P'_m$ and
$L'_1, L'_2, L'_3,  P'_4, \ldots , P'_m$ are pairwise $\Omega$-disjoint,
\item $\vec{H'}$ is acyclic.
\end{enumerate}
\end{enumerate}
Then one of the following does not hold: 
\begin{enumerate}[\;\;(i)]
\item $I_d\cup \{\Omega\}$ does not have a $k$-mate,
\item for every directed odd $T'$-join $L'$ of $\vec{H'}$, $L'\cup I_d$ contains an odd $st$-dipath of $\vec{H}$.
\end{enumerate} 
\end{lma}
\begin{proof}
Suppose otherwise.
Let $(G',\Sigma',T'):=(G,\Sigma,\{s,t\})/I_c\setminus I_d$ where $\Sigma'=\Sigma$; this signed graft is well-defined by (1).
For $i\in [m]-[3]$, let $L'_i:=L_i-P_i$ if $T'= \emptyset$, and let $L'_i:=(L_i-P_i) \cup P'_i$ otherwise.
Let $\mathcal{L'}:=(L'_1,\ldots,L'_m, L_{m+1},\ldots,L_k)$.
If $T'=\emptyset$, let $m':=3$, and if not, let $m':=m$.
We will show that $((G',\Sigma',T'), \mathcal{L'}, m', \vec{H'})$ is either a non-simple bipartite $\Omega$-system or a simple bipartite $\Omega$-system, and this will yield a contradiction, thereby finishing the proof.

{\bf (B1)} By (2), every signature or $T'$-cut of $(G',\Sigma',T')$ has the same parity as $\tau(G,\Sigma,\{s,t\})$, implying that $(G',\Sigma',T')$ is an Eulerian signed graft.
{\bf (B2)} It also implies that $k, \tau(G,\Sigma,\{s,t\})$ and $\tau(G',\Sigma',T')$ have the same parity, so every minimal cover of $(G',\Sigma',T')$ has the same size parity as $k$.
We claim that $\tau(G',\Sigma',T')\geq k$.
Let $B'$ be a minimal cover of $(G',\Sigma',T')$. 
If $\Omega\notin B'$, then $$|B'|\geq \sum \left( |B'\cap L'|: L'\in \mathcal{L'}\right)\geq k.$$ Otherwise, $\Omega\in B'$. In this case, $B'\cup I_d$ contains a cover $B$ of $(G,\Sigma,\{s,t\})$. By (i), $I_d\cup \{\Omega\}$ does not have a $k$-mate, so $$k-2\leq |B-(I_d\cup \{\Omega\})|\leq |B-I_d|-1\leq |B'|-1,$$ and since $|B'|,k$ have the same parity, it follows that $|B'|\geq k$. Thus, $\mathcal{L'}$ is an $(\Omega,k)$-packing. When $T'=\emptyset$, $m'=3$. When $T'=\{s,t\}$, then $m'=m$ and for $j\in [m']-[3]$, $L'_j$ contains even $st$-path $P'_j$ and some odd circuit in $L'_j-P'_j$, and for $j\in [k]-[m']$, $L_j$ remains connected in $G'$.
{\bf (B3)} follows from construction.

Suppose first that $T'=\emptyset$. We will show that $((G',\Sigma',T'), \mathcal{L'}, m', \vec{H'})$ is a non-simple bipartite $\Omega$-system. {\bf (NS1)} holds as (B1)-(B3) hold.
{\bf (NS2)} holds as $T'=\emptyset$.
{\bf (NS3)} follows from (3).
{\bf (NS4)} Let $L'$ be a directed odd $T'$-join of $\vec{H'}$ that is $\Omega$-disjoint from a directed odd circuit.
By (ii), $L'\cup I_d$ contains an odd $st$-dipath $L$ of $\vec{H}$. Since $((G,\Sigma,\{s,t\}),\mathcal{L},m,\vec{H})$ is of flavour (SF1), $L$ has a signature $k$-mate $B$. 
By proposition~\ref{bpsignature}, $B\cap E(\vec{H})=B\cap L$, implying that $B\cap I_c=\emptyset$.
Thus, $B-I_d$ contains a minimal cover $B'$ of $(G',\Sigma',T')$, and since $$|B'-L'|\leq |(B-I_d)-L'|\leq |B-L|\leq k-3,$$ it follows that $B'$ is a $k$-mate of $L'$.

Suppose now that $T'=\{s,t\}$. We will show that $((G',\Sigma',\{s,t\}), \mathcal{L'}, m, \vec{H'})$ is a simple bipartite $\Omega$-system. {\bf (S1)} holds as (B1)-(B4) hold.
{\bf (S2)} follows from (4).
{\bf (S3)} Let $L'$ be an odd $st$-dipath in $\vec{H'}$.
By (ii), $L'\cup I_d$ contains an odd $st$-dipath $L$ of $\vec{H}$. Since $((G,\Sigma,\{s,t\}),\mathcal{L},m,\vec{H})$ is a simple bipartite $\Omega$-system of flavour (SF1), $L$ has a signature $k$-mate $B$. By proposition~\ref{bpsignature}, $B\cap E(\vec{H})=B\cap L$, implying that $B\cap I_c=\emptyset$.
Then $B-I_d$ contains a minimal cover $B'$ of $(G',\Sigma',\{s,t\})$, and since $$|B'-L'|\leq |(B-I_d)-L'|\leq |B-L|\leq k-3,$$ it follows that $B'$ is a $k$-mate of $L'$.
\end{proof}

\subsection{The proof of proposition~\ref{prp-SF1}}

Let $\bigl((G,\Sigma,\{s,t\}), (L_1,\ldots,L_k), m, \vec{H}\bigr)$ be a minimal simple bipartite $\Omega$-system of flavour (SF1), where $\Omega$ has ends $s,s'$, and assume there is no non-simple bipartite $\Omega$-system whose associated signed graft is a minor of $(G,\Sigma,\{s,t\})$.

\begin{claims} 
$m\geq 4$.
\end{claims}
\begin{cproof}
By (SF1), each one of $P_1,P_2,P_3$ has a signature $k$-mate, so the result follows from the mate proposition~\ref{mateprop}.
\end{cproof}

\begin{claims} 
There is an odd circuit $C$ in $\vec{H}\setminus t$.
\end{claims}
\begin{cproof}
Suppose otherwise. Let $I_c:=P_4$ and $I_d:=\emptyset$.
Let $T':=\emptyset$, for $i\in [3]$ let $L'_i:=P_i$, and let $\vec{H'}\subseteq \vec{H}\setminus I_d/I_c$ be the union of $L'_1,L'_2,L'_3$.
It is clear that (1)-(4) of the disentangling lemma~\ref{SF1-disentangle} hold.
As $\tau(G,\Sigma,\{s,t\})\geq k$, it follows that $I_d\cup \{\Omega\}=\{\Omega\}$ does not have a $k$-mate, so (i) holds.
Moreover, our assumption implies that $P_4$ is internally vertex-disjoint from each of $P_1,P_2,P_3$.
This implies that every directed odd circuit of $\vec{H'}$ is an odd $st$-dipath in $\vec{H}$, so (ii) holds, a contradiction with the disentangling lemma~\ref{SF1-disentangle}.
\end{cproof}

Consider the following partial ordering on $V(\vec{H})$: $u\leq v$ if there exists a $uv$-dipath in $\vec{H}$. For $j\in [m]$ let $v_j$ be the second largest vertex of $P_j$ that lies on another $st$-dipath in $\{P_1,\ldots,P_m\}-\{P_j\}$. By proposition~\ref{intersection} there exists an index subset $I\subseteq [m]$ of size at least two such that, for each $i\in I$, \begin{itemize}
\item $v_i\geq v_1$, and there is no $j\in [m]$ such that $v_j>v_i$,
\item for each $j\in [m]$, $v_i=v_j$ if and only if $j\in I$.
\end{itemize} 
After redefining $\mathcal{L}$, if necessary, we may assume that $1\in I$.

\begin{claims} 
For each $i\in I$, $P_i[v_i,t]\cup \{\Omega\}$ has a $k$-mate.
\end{claims}
\begin{cproof}
Let $I_d:=P_i[v_i,t]$ and $I_c:=\bigcup (P_j[v_j,t]: j\in I-\{i\})$.
Let $T':=\{s,t\}$, for $j\in [3]$ let $L'_j:=P_j-(I_c\cup I_d)$, and for $j\in [m]-[3]$ let $P'_i:=P_i-(I_c\cup I_d)$.
Let $\vec{H'}\subseteq \vec{H}\setminus I_d/I_c$ be the union of $L'_1,L'_2,L'_3,P'_4,\ldots,P'_m$.
Clearly, (1)-(4) and (ii) of the disentangling lemma~\ref{SF1-disentangle} hold.
Hence the lemma implies that (i) does not hold, proving the claim.
\end{cproof}

\begin{claims} 
There are two vertex-disjoint paths $P,Q$ in $H$, where $P$ is between $s,t$ and $Q$ is between $s',v_1$.
\end{claims}
\begin{cproof}
Suppose otherwise.

Assume first that $s'=v_1$.
Then, for each $j\in [m]$, $s'\in V(P_j)$.
By (SF1), for each $j\in [m]$, $P_j[s',t]\cup \{\Omega\}$ has a signature $k$-mate $B_j$.
However, for each $j\in [m]$, $B_j$ is also a signature $k$-mate for $P_j\cup \{\Omega\}$. This is a contradiction with the mate proposition~\ref{mateprop}.

Thus, $s'\neq v_1$.
By the linkage lemma~\ref{linkage}, $H$ is a spanning subgraph of an $(s,v_1,t,s')$-web with frame $C_0$ and rib $H_0$.
Fix a plane drawing of $H_0$, where the unbounded face is bounded by $C_0$.
After redefining $\mathcal{L}$, if necessary, we may assume the following: \begin{quote} $(\star)$ for every $s'v_1$-dipath $P$ of $\vec{H}$, the number of rib vertices that are on the same side of $P$ as $s$ is at least as large as that of $P_1[s',v_1]$. \end{quote} 

For $j\in [m]-[3]$, let $u_j$ be the largest rib vertex on $P_j$ that also lies on $P_1[s',v_1]$.
Observe that if $j\in I\cap ([m]-[3])$, then $u_j=v_j$.
For $j\in [m]-[3]$ let $R_j:=P_j[u_j,t]$, for $j\in [3]\cap I$ let $R_j:=P_j[v_j,t]$, and for $j\in [3]-I$ let $R_j:=P_j[s',t]$.
Observe that a $k$-mate for $R_j\cup \{\Omega\}, j\in [m]$ is also a $k$-mate for any odd $st$-dipath of $\vec{H}$ containing $R_j\cup \{\Omega\}$. Hence, by (SF1), every $k$-mate for $R_j\cup \{\Omega\}, j\in [m]$ must be a signature.
However, every $k$-mate for $R_j\cup \{\Omega\}, j\in [m]$ is also a $k$-mate for $P_j\cup \{\Omega\}$.
Hence, by the mate proposition~\ref{mateprop}, there exists $i\in [m]$ such that $R_i\cup \{\Omega\}$ has no $k$-mate.
By (S3) and claim~3, $i\notin I\cup [3]$.
Observe that $(\star)$ implies the following: \begin{quote} $(\star\star)$ if $w\in V(P_1[u_i,t])$ and $Q$ is an $s'w$-dipath, then $Q$ and $R_i$ have a vertex in common.
\end{quote}

Let $I_d:=R_i$ and $I_c:=P_1[u_i,t]$.
For $j\in [3]$ let $L'_j$ be $P_j-(I_c\cup I_d)$ minus any directed circuit that does not use $\Omega$, and for $j\in [m]-[3]$ let $P'_j$ be $P_j-(I_c\cup I_d)$ minus any directed circuit that does not use $\Omega$.
Let $T':=\{s,t\}$ and $\vec{H'}\subseteq \vec{H}\setminus I_d/I_c$ be the union of $L'_1,L'_2,L'_3, P'_4,\ldots,P'_m$. 
It is clear that (1)-(4) of the disentangling lemma~\ref{SF1-disentangle} hold.
By the choice of $R_i$, (i) holds as well.
To show (ii) holds, let $L'$ be an odd $st$-dipath of $\vec{H'}$.
Then $L'\cup I_c$ contains an odd $st$-dipath of $\vec{H}$, and by $(\star\star)$, $L'\cup I_d$ also contains an odd $st$-dipath of $\vec{H}$, so (ii) holds, a contradiction with lemma~\ref{SF1-disentangle}.
\end{cproof}

\begin{claims} 
$(G,\Sigma,\{s,t\})$ has an $F_7$ minor.
\end{claims}
\begin{cproof}
For $i\in I$, let $B_i$ be a $k$-mate of $P_i[v_i,t]\cup \{\Omega\}$, whose existence is guaranteed by claim~3. 
For each $i\in I$, since $B_i$ is also a $k$-mate for odd $st$-dipath $P_1[s,v_1]\cup P_i[v_i,t]$, (SF1) implies that $B_i$ is a signature.
Take $j\in I-\{1\}$.
Choose $U\subseteq V(G)-\{s,t\}$ such that $B_1\triangle B_j=\delta(U)$. Then by proposition~\ref{matessignsign} there exists a path $R$ in $G[U]$ between $V(P_1[v_1,t])\cap U$ and $V(P_j[v_j,t])\cap U$ such that $R\cap (B_1\cup B_j)=\emptyset$, and $R$ is minimal subject to this property.
Observe that $P\cup Q\cup C$ has no vertex in common with $U$.
It is easy (and is left as an exercise) to see that $C\cup P\cup Q\cup P_1[v_1,t]\cup P_j[v_j,t]\cup R$ has an $F_7$ minor.
\end{cproof}
        
\section{A preliminary for cut $\Omega$-systems: the shore proposition}\label{sec-shore}
%
The following proposition can be the thought of as the second half of the mate proposition~\ref{mateprop}:

\begin{prp}\label{shoreprop}
Let $((G,\Sigma,\{s,t\}), \mathcal{L}=(L_1,\ldots,L_k),m,H)$ be a bipartite $\Omega$-system, where $\Omega$ has ends $s,s'$.
For each $i\in [m]$, let $\widetilde{P_i}\subseteq L_i$ be a connected $st$-join such that $\widetilde{P_i}\cap \Sigma\subseteq \{\Omega\}$, and if $i\in [3]$, $\Omega\in \widetilde{P_i}$ and $\widetilde{P_i}\cap \delta(s)=\{\Omega\}$.
Suppose there exist $B_1,\ldots,B_m$ and $U\subseteq V(G)-\{t\}$ with $s\in U$ such that \begin{enumerate}[\;\;(i)]
\item for $i\in [m]$, $B_i$ is a $k$-mate of $\widetilde{P_i}\cup \{\Omega\}$, 
\item exactly one of $B_1,\ldots,B_m$, say $B_\ell$, is not a signature, and $B_\ell=\delta(U)$,
\item there is no proper subset $W$ of $U$ with $s\in W$ such that $\delta(W)$ is a $k$-mate of $\widetilde{P_\ell}\cup \{\Omega\}$,
\item for $i\in [m]$, $B_i\cap P_i$ has no edge in $G[U]$.
\end{enumerate} Then, for every component of $\widetilde{P_\ell}$ in $G[U]$, there is a path $P$ in $G[U]$ between $s$ and a vertex of the component such that $P\cap (B_1\cup \cdots \cup B_{\ell-1}\cup B_{\ell+1}\cup \cdots \cup B_m)=\emptyset$.
\end{prp}

\begin{proof}
For each $i\in [3]$, let $\widetilde{C_i}:=\emptyset$, and for each $i\in [m]-[3]$, let $\widetilde{C_i}$ be an odd circuit contained in the odd cycle $L_i\triangle \widetilde{P_i}=L_i-\widetilde{P_i}$.
By identifying a vertex of each component with $s$, if necessary, we may assume that $G$ is connected. 
For $n\geq 1$, let $[n]':=[n]-\{\ell\}$.
We will be applying lemma~\ref{mate-shore} to the index set $[m]'$.
For distinct $i,j\in [m]'$, choose $U_{ij}\subseteq V(G)-\{s,t\}$ such that $\delta(U_{ij})=B_i\triangle B_j$. Observe that $[m]'$ contains $m-1$ ordered indices; for every index $j$ other than the two smallest indices in $[m]'$, let $$S_j:= \bigcap (U_{ij}: i\in [m]', i<j).$$ By definition, each $S_j$ is the intersection of at least two sets.
Take $C\in \{\widetilde{C_4},\ldots,\widetilde{C_m}\}$ and an $S_j$. We say {\it $C$ is bad for~$S_j$} if $$|C\cap \delta(S_j)|=2 \quad\mbox{and}\quad C\cap \delta(S_j)\cap B_j=\emptyset.$$ We need a few preliminaries.

\begin{claim} 
Each circuit in $\{C_4,\ldots,C_m\}$ is bad for at most one $S_j$.
\end{claim}
\begin{cproof}
Suppose that $C\in \{C_4,\ldots,C_m\}$ is bad for $S_j$ and that it is not bad for any $S_i$ with $i<j$.
By lemma~\ref{mate-shore}(5), there exist distinct $p,q\in [j-1]'$ such that $C\cap \delta(S_j)\subseteq B_p\cup B_q$.
By lemma~\ref{mate-shore}(11), $V(C)\subseteq U_{jp}\cup U_{jq}$, and subsequently by lemma~\ref{mate-shore}(6), $V(C)\cap S_r=\emptyset$ for $r>j$.
As a result, $C$ cannot be bad for any $S_r$ with $r>j$.
\end{cproof}

\begin{claim} 
Each $S_j$ has a bad circuit.
\end{claim}
\begin{cproof} 
Suppose for a contradiction that some $S_j$ has no bad circuit, and let $B:= B_j\bigtriangleup \delta(S_j)$.
We will prove that $B$ is a cover of size $k-2$, which will yield a contradiction as $|B|\geq \tau(G,\Sigma)\geq k$.
It is clear that $B$ is a cover. 
By lemma~\ref{mate-shore},
$$B\subseteq \bigcup (B_i: i\in [m]', i\leq j)\subseteq \bigcup (L_i: i\in [k]').$$ The first inclusion follows from part (5), and the second inclusion follows from part (1) together with the fact that for each $i\in [m]'$, $B_i\cap \widetilde{P_\ell}\subseteq \{\Omega\}$.
Therefore, as $\Omega\in B$ and $|L_\ell\cap B|=1$, it suffices to show that, for all $i\in [k]'$, $|L_i\cap B| = 1$. Keep in mind that, for all $i\in [k]-\{j\}$, $|L_i\cap B_j|=1$.

Take $i\in [k]-[m]$. If $L_i\cap \delta(S_j)=\emptyset$, then $|L_i\cap B|=|L_i\cap B_j|=1$. Otherwise, when $L_i\cap \delta(S_j)\neq \emptyset$, lemma~\ref{mate-shore} part (9) implies $|L_i\cap \delta(S_j)|=2$ and $|L_i\cap \delta(S_j)\cap B_j|=1$, so $|L_i\cap B|= |L_i\cap (B_j\bigtriangleup \delta(S_j))|=1$.

Next take $i\in [m]'$.
We will first consider $\widetilde{C_i}\cap B$, given that $\widetilde{C_i}\neq \emptyset$. If $\widetilde{C_i}\cap \delta(S_j)=\emptyset$, then $|\widetilde{C_i}\cap B|=|\widetilde{C_i}\cap B_j|=1$. Otherwise, $\widetilde{C_i}\cap \delta(S_j)\neq \emptyset$. Then, by lemma~\ref{mate-shore}(10), $|\widetilde{C_i}\cap \delta(S_j)|=2$. By our choice of $S_j$, $\widetilde{C_i}$ is not bad for $S_j$, so $|\widetilde{C_i}\cap \delta(S_j)\cap B_j|=1$. Thus, $|\widetilde{C_i}\cap B|= |\widetilde{C_i}\cap (B_j\bigtriangleup \delta(S_j))|=1$.
We next consider $(\{\Omega\}\cup \widetilde{P_i})\cap B$. If $i\neq j$, then by lemma~\ref{mate-shore}, \begin{align*}
(\{\Omega\}\cup \widetilde{P_i})\cap B &=(\{\Omega\}\cup \widetilde{P_i})\cap (B_j\bigtriangleup \delta(S_j))\\
&= (\{\Omega\}\cup \widetilde{P_i})\cap B_j \quad \text{ by part (8)}\\
&=\{\Omega\}\quad \text{ by part (3).}
\end{align*} On the other hand, if $i=j$, then by lemma~\ref{mate-shore}, \begin{align*}
(\{\Omega\}\cup \widetilde{P_j})\cap B &=(\{\Omega\}\cup \widetilde{P_j})\cap (B_j\bigtriangleup \delta(S_j))\\
&=[(\{\Omega\}\cup \widetilde{P_j})\cap B_j] \triangle [(\{\Omega\}\cup \widetilde{P_j})\cap \delta(S_j)]\\
&=\{\Omega\}\quad \text{ by part (7).}
\end{align*} 
Since whenever $\Omega\in \widetilde{P_i}$ then $\widetilde{C_i}=\emptyset$, $|L_i\cap B|= |\widetilde{C_i}\cap B|+|\widetilde{P_i}\cap B|= 1$.
\end{cproof}

Let $\mathfrak{U}:=\bigcup (U_{ij}:i,j\in [m]', i\neq j)$.

\begin{claim} 
For each $j\in [m]-[3]$, $V(\widetilde{C_j})\subseteq \mathfrak{U}$
\end{claim}
\begin{cproof}
Claims~1 and~2 imply that each circuit of $\widetilde{C_4},\ldots,\widetilde{C_m}$ is bad for an $S_j$ (of which there are $m-3$ many). The claim now follows from lemma~\ref{mate-shore}(11).
\end{cproof}

\begin{claim} 
Let $e\in E(G)$ be an edge with both ends in $V(G)-\mathfrak{U}$, and let $i\in [m]'$. If $e\in B_i$, then $e\in B_1\cap \cdots \cap B_{\ell-1}\cap B_{\ell+1}\cap \cdots \cap B_m$.
\end{claim}

\begin{cproof} 
As $e$ has both ends in $V(G)-\mathfrak{U}$, for each distinct $p,q\in [m]'$, we have $e\notin \delta(U_{pq})=B_p\bigtriangleup B_q$, proving the claim. 
\end{cproof}

\begin{claim} 
Let $e\in E(G)$ be an edge with both ends in $U-\mathfrak{U}$ such that $e\in B_1\cup \cdots \cup B_{\ell-1}\cup B_{\ell+1}\cup \cdots \cup B_m$. Then $e\in L_{m+1}\cup \cdots \cup L_k$.
\end{claim}
\begin{cproof}
As $e\neq \Omega$, $e\notin L_1\cup L_2\cup L_3$.
By (iv), $e\notin \widetilde{P_4}\cup \cdots \cup \widetilde{P_m}$.
By claim~3, $e\notin \widetilde{C_4}\cup \cdots \cup \widetilde{C_m}$.
The claim now follows from proposition~\ref{packmate}.
\end{cproof}

\begin{claim} 
For each $i\in [m]$, $\widetilde{P_i}$ has no vertex in common with $U\cap \mathfrak{U}$.
\end{claim}
\begin{cproof}
Observe that $\widetilde{P_\ell}$ has no vertex in common with $\mathfrak{U}$, for $\widetilde{P_\ell}\cap \delta(\mathfrak{U})=\emptyset$ and $\widetilde{P_\ell}$ is connected.
We may therefore assume $i\in [m]'$, and for a contradiction, assume $\widetilde{P_i}$ has a vertex $v$ in common with $U\cap \mathfrak{U}$. Since 
Since $|\widetilde{P_i}\cap \delta(U)|=1$, the edges of $\widetilde{P_i}[s,v]$ belong to $G[U]$, so by (iv), $\widetilde{P_i}[s,v]\cap B_i=\emptyset$.
Since $u\in \mathfrak{U}$, there exist distinct $p,q\in [m]'$ such that $u\in U_{pq}$. Since $\widetilde{P_i}[s,v]\cap B_i=\emptyset$, we may assume that $p\neq i$ and $\widetilde{P_i}[s,v]\cap B_p\neq \emptyset$. However, as $B_p$ is a signature, $\widetilde{P_i}\cap B_p\subseteq \{\Omega\}$, a contradiction as $\Omega\in \delta(U)$.
\end{cproof}

\begin{claim} 
For every component of $\widetilde{P_\ell}$ in $G[U]$, there is a path $P$ in $G[U-\mathfrak{U}]$ between $s$ and a vertex of the component such that $P\cap (B_1\cup \cdots \cup B_{\ell-1}\cup B_{\ell+1}\cup \cdots \cup B_m)=\emptyset$.
\end{claim}
\begin{cproof}
Suppose otherwise.
By claim~4, there exists $W\subseteq (U-\mathfrak{U})-\{s\}$ where $\widetilde{P_\ell}\cap \delta(W)\neq \emptyset$ such that, for every edge $e\in E(G)$ with one end in $W$ and another in $(U-\mathfrak{U})-W$, we have $e\in B_1\cap \cdots \cap B_{\ell-1}\cap B_{\ell+1}\cap \cdots \cap B_m$.
Let $U':=U-W$.
We will show that $\delta(U')$ is a cap of $L_\ell$ in $\mathcal{L}$.

{\bf (T1)} and {\bf (T2)} clearly hold. {\bf (T3)} We have $$\delta(U')\subseteq \delta(U)\cup \delta(W)\subseteq (B_1\cup \cdots \cup B_m)\cup \delta(\mathfrak{U})
\subseteq B_1\cup \cdots \cup B_m\subseteq L_1\cup \cdots \cup L_k.$$ In fact, the argument of the last inclusion can be replaced with $$(\widetilde{P_1}\cup \cdots\cup \widetilde{P_m})\cup (\widetilde{C_4}\cup \cdots \cup \widetilde{C_m})\cup (L_{m+1}\cup \cdots \cup L_k). $$
{\bf (T4)} 
Let $i\in [m]'$. When $i\in [3]$, we have $V(L_i)\cap U=\{s\}$, implying that $L_i\cap \delta(U')=\{\Omega\}$.
Otherwise, when $i\in [m]-[3]$, claim~3 implies that $\widetilde{C_i}\cap \delta(U')=\emptyset$ and claims~5 and~6 imply that $|\widetilde{P_i}\cap \delta(U')|=|\widetilde{P_i}\cap \delta(U)|=1$, so $|L_i\cap \delta(U')|=1$.

Let $i\in [k]-[m]$. Recall that $L_i$ is a connected odd $st$-join.
If $L_i\cap \delta(W)=\emptyset$, then $|L_i\cap \delta(U')|=|L_i\cap \delta(U)|=1$.
We may therefore assume that $L_i\cap \delta(W)\neq \emptyset$.
We claim that $|L_i\cap \delta(W)|=2$ and that one of the edges in $L_i\cap \delta(W)$ belongs to $\delta(U)$. Note that this will prove that $|L_i\cap \delta(U')|=1$.
If $L_i\cap \delta(W)$ contains an edge $e$ with one end in $W$ and another in $U'-\mathfrak{U}$, then $e\in B_1\cap \cdots \cap B_{\ell-1}\cap B_{\ell+1}\cap \cdots \cap B_m$.
However, $|L_i\cap B_1|=\cdots=|L_i\cap B_{\ell-1}|=|L_i\cap B_{\ell+1}|=\cdots=|L_i\cap B_m|=1$, so $|L_i\cap \delta(W)|=2$ and the edge in $(L_i\cap \delta(W))-\{e\}$ belongs to $\delta(U)$, and we are done.
Otherwise, it suffices to show that $L_i$ does not contain two edges $e,f$, each with one end in $U\cap \mathfrak{U}$ and another in $W$. 
Suppose otherwise. 
Let $v_e,v_f$ be the ends of $e,f$ in $U\cap \mathfrak{U}$, respectively, and let $u_e,u_f$ be the ends of $e,f$ in $W$, respectively.
\begin{center}
\includegraphics[scale=0.3]{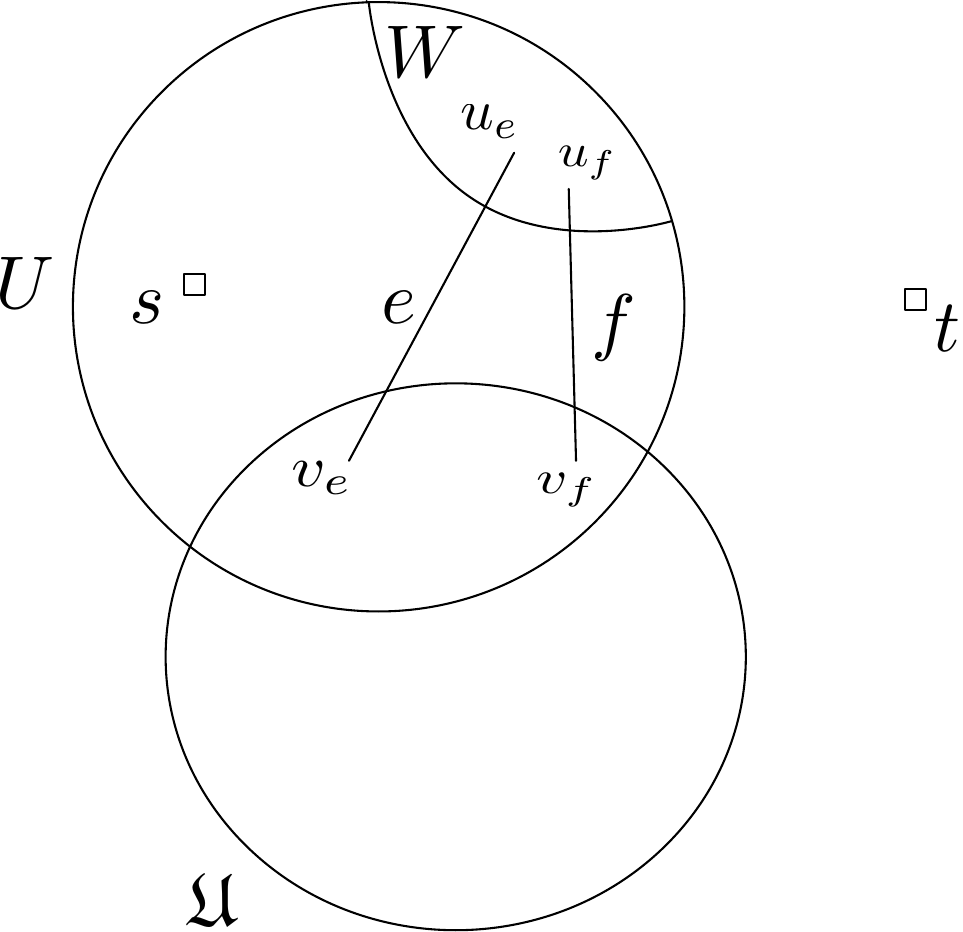}
\end{center}
Since $e,f\in \delta(\mathfrak{U})$, each of $e,f$ belongs to $\cup_{j\in [m]'} B_j$. Since $L_i$ intersects each one of $B_j,j\in [m]'$ exactly once, there are distinct $p,q\in [m]'$ such that $e\in B_p,f\in B_q$ and $\{e,f\}\subseteq B_p\triangle B_q=\delta(U_{pq})$.
Since $|L_i\cap \delta(U)|=1$ and $L_i$ is connected, we get that $L_i$ contains a path $Q$ in $G[U]$ containing the vertex $s$ and edges $e,f$. Since $L_i\cap \delta(W)$ does not contain an edge with one end in $W$ and another in $U'-\mathfrak{U}$, it follows that $Q\cap \delta(W)$ does not contain an edge with one end in $W$ and another in $U'-\mathfrak{U}$, implying in turn that $|Q\cap \delta(U_{pq})|\geq 3$, so $|L_i\cap \delta(U_{pq})|\geq 3$, a contradiction. Hence, $|L_i\cap \delta(U')|=1$.

Moreover, $L_\ell\cap \delta(U')\subsetneq L_\ell\cap \delta(U)$, and since $\tau(G,\Sigma)\geq k$, it follows that $|L_\ell\cap \delta(U')|\geq 3$. As a result, (T4) holds, so $\delta(U')$ is a cap of $L_\ell$ in $\mathcal{L}$. Proposition~\ref{packmate} therefore implies that $\delta(U')$ is a $k$-mate of $L_\ell$, but $\delta(U')\cap L_\ell= \delta(U')\cap \widetilde{P_\ell}$, so $\delta(U')$ is a $k$-mate for $\widetilde{P_\ell}$, a contradiction with (iii).
\end{cproof}

Note that claim~7 finishes the proof of the shore proposition.\end{proof}
               
\section{Primary cut $\Omega$-system}\label{sec-cut-primary}
%
\subsection{Signature mates and the brace proposition}
%
\begin{prp}\label{cut-primary-signature}
Let $((G,\Sigma,\{s,t\}), \mathcal{L}=(L_1,\ldots,L_k), m, (U_1,\ldots,U_n), \vec{H})$ be a primary cut $\Omega$-system.
Let $P$ be an odd $st$-dipath with $V(P)\cap U_n=\{s\}$, and let $B$ be a $k$-mate of it.
Then $B$ is not an $st$-cut.
\end{prp}
\begin{proof}
After redefining $\mathcal{L}$, if necessary, we may assume that $P=P_2=L_2$. (Note the acyclicity condition in (C3).)
Suppose, for a contradiction, that $B$ is an $st$-cut. 
Choose $W\subseteq V(G)-\{t\}$ with $s\in W$ such that $B=\delta(W)$.
Since $L_2$ is simple, it follows that $\delta(U_n\cap W)\cap L_2=\{\Omega\}$.
As the brace and the base of $L_1$ intersect $\delta(W)$ at only $\Omega$, it follows that $q,d\in U_n-W$, and since the residue of $L_1$ is a connected $qd$-join, it follows that $\delta(U_n\cap W)\cap L_1=\{\Omega\}$, contradicting proposition~\ref{usefulparity} part (4).
\end{proof}
\begin{prp}\label{braceprop}
Let $((G,\Sigma,\{s,t\}), \mathcal{L}=(L_1,\ldots,L_k), m, \mathcal{U}=(U_1,\ldots,U_n), \vec{H})$ be a minimal cut $\Omega$-system that is primary.
Let $P^+$ be an $st$-dipath in $\vec{H}^+\setminus \Omega$. Then $P^+$ and the brace share no vertex outside $U_n$.
\end{prp}
\begin{proof}
After redefining $\mathcal{L}$, if necessary, we may assume that
$P:=P^+\cap E(\vec{H})$ is the base for one of $P_4,\ldots,P_m$.
(Note the acyclicity condition in (C3).) Suppose for a contradiction that $P^+$ and the brace share a vertex outside $U_n$.

In the first case, assume that $P$ is the base for one of $L_{n+3},\ldots,L_m$, say $P=Q_{n+3}$.
Let $x$ be the closest vertex to $t$ on $Q_{n+3}$ that belongs to the both of $D$ and $V(G)-U_n$.
Let $L'_1:=D[s,x]\cup Q_{n+3}[x,t]$ and $L'_{n+3}:=(Q_{n+3}[s,x]\cup D[x,d]\cup R\cup Q)\cup C_{n+3}$.
Let $$\mathcal{L'}:=(L'_1,L_2,L_3,\ldots,L_{n+2},L'_{n+3},L_{n+4},\ldots,L_k).$$ Note that $\mathcal{U}$ is a secondary cut structure for $((G,\Sigma,\{s,t\}),\mathcal{L'},m)$, where the base for $L'_{n+3}$ is $Q$.
Let $\vec{H'}:=\vec{H}\setminus (Q_{n+3}[s,x]\cup D[x,d])$.
Then it is easily seen that $((G,\Sigma,\{s,t\}), \mathcal{L'}, m, \mathcal{U}, \vec{H'})$ is a secondary cut structure, contradicting the minimality of the original $\Omega$-system.

In the remaining case, assume that $P=Q_j$ for some $j\in [n+2]-[3]$.
Let $x$ be the closest vertex to $t$ on $Q_j$ that belongs to the both of $D$ and $V(G)-U_n$.
Let $L'_1:=D[s,x]\cup Q_j[x,t]$ and $L'_j:=(R_j\cup P[q_j,x]\cup D[x,d]\cup R\cup Q)\cup C_j$.
Let 
\begin{align*}
\mathcal{L'}&:=(L'_1,L_2,\ldots,L_{j-1},L_{j+1},\ldots,L_{n+2},L'_j,L_{n+3}, \ldots,L_k)\\
\mathcal{U'}&:=(U_1,\ldots,U_{j-4},U_{j-2},\ldots,U_n).
\end{align*}
Then $\mathcal{U'}$ is a secondary cut structure for $((G,\Sigma,\{s,t\}),\mathcal{L'},m)$, where the base for $L'_j$ is $Q$, and $\delta(U_n)$ is a $k$-mate for $L'_j-C_j$.
Let $\vec{H'}:=\vec{H}\setminus (Q_j[q_j,x]\cup D[x,d])$.
Then it is easily seen that $((G,\Sigma,\{s,t\}), \mathcal{L'}, m, \mathcal{U'}, \vec{H'})$ is a secondary cut structure, contradicting the minimality of the original $\Omega$-system.
\end{proof}
\subsection{A disentangling lemma}
%
\begin{lma}\label{cut-primary-disentangle}
Let $((G,\Sigma,\{s,t\}), \mathcal{L}=(L_1,\ldots,L_k), m, \mathcal{U}=(U_1,\ldots,U_{n-1},U), \vec{H})$ be a minimal cut $\Omega$-system that is primary, and assume there is no non-simple bipartite $\Omega$-system whose associated signed graft is a minor of $(G,\Sigma,\{s,t\})$.
Take disjoint subsets $I_d,I_c\subseteq E(\vec{H}\setminus \Omega)$ and $T'\subseteq \{s,t\}$ where \begin{enumerate}[\;\;(1)]
\item $I_c$ is non-empty, if $I_c$ contains an $st$-path then $T'=\emptyset$, and if not then $T'=\{s,t\}$,
\item every signature or $st$-cut disjoint from $I_c$ intersects $I_d$ in an even number of edges,
\item if $T'=\emptyset$, there is a directed subgraph $\vec{H'}$ of $\vec{H}/I_c\setminus I_d$ that is the union of directed odd circuits $L'_1,L'_2,L'_3$ where \begin{enumerate}[\;\;]
\item $\Omega\in L'_1\cap L'_2\cap L'_3$ and $L'_1,L'_2,L'_3$ are pairwise $\Omega$-disjoint,
\item $\vec{H'}\setminus \Omega$ is acyclic,
\end{enumerate}
\item if $T'=\{s,t\}$, then $I_d,I_c\subseteq E(\vec{H}\setminus U)$ and there is a directed subgraph $\vec{H'}$ of $\vec{H}/I_c\setminus I_d$ that is the union of $D',Q'$, odd $st$-dipaths $L'_2, L'_3$, and 
dipaths $Q'_4,\ldots,Q'_m$, where \begin{enumerate}[\;\;-]
\item $D'$ is an $sd$-dipath containing $\Omega$ with $V(D')\cap U=\{s,d\}$, $Q'$ is a $qt$-dipath with $V(Q')\cap U=\{q\}$, and $D',Q'$ have no vertex outside $U$ in common,
\item for $i=4,\ldots,n+2$, $Q'_i$ is a $q_{i-3}t$-dipath with $V(Q'_i)\cap U_{i-3}=\{q_{i-3}\}$, and for $i=n+3,\ldots,m$, $Q'_i$ is an even $st$-dipath,
\item $D',Q',L'_2, L'_3,Q'_4, \ldots , Q'_m$ are pairwise $\Omega$-disjoint,
\item $D',Q',Q'_4,\ldots,Q'_m$ coincide with $D,Q,Q_4,\ldots,Q_m$ on $E(G[U])\cup \delta(U)$, respectively,
\item the following digraph is acyclic: start from $\vec{H'}$, for each $q_i$ add arc $(s,q_i)$, and if $d\neq q$, add arc $(d,q)$.
\end{enumerate}
\end{enumerate}
Then one of the following does not hold: 
\begin{enumerate}[\;\;(i)]
\item $I_d\cup \{\Omega\}$ does not have a $k$-mate,
\item if $T'=\emptyset$, then for every directed odd circuit $L'$ of $\vec{H'}$, either $L'\cup I_d$ contains an odd $st$-dipath $P$ of $\vec{H}$ with $V(P)\cap U=\{s\}$, or $L'\cup I_d$ has a $k$-mate in $(G,\Sigma,\{s,t\})$ disjoint from $I_c$,
\item if $T'=\{s,t\}$, then for every odd $st$-dipath $P'$ of $\vec{H'}$ with $V(P')\cap U=\{s\}$, either $P'\cup I_d$ contains an odd $st$-dipath of $\vec{H}$, or $P'\cup I_d$ has a $k$-mate in $(G,\Sigma,\{s,t\})$ disjoint from~$I_c$.
\end{enumerate} 
\end{lma}
\begin{proof}
Suppose otherwise.
Let $(G',\Sigma',T'):=(G,\Sigma,\{s,t\})/I_c\setminus I_d$ where $\Sigma'=\Sigma$; this signed graft is well-defined by (1).
Let $\mathcal{L'}:=(L'_1,\ldots,L'_m, L_{m+1},\ldots,L_k)$, where $L'_1,\ldots,L'_m$ are defined as follows.
If $T'=\emptyset$,
let $m':=3$, 
and for $i\in [m]-[3]$, let $L'_i:=L_i-P_i$.
Otherwise, when $T'=\{s,t\}$, 
let $m':=m$,
$L'_1:=D'\cup Q'\cup R$,
and for $i\in [m]-[3]$, let $L'_i:=(L_i-Q_i)\cup Q'_i$.

We will first show that $((G',\Sigma',T'), \mathcal{L'}, m')$ is a bipartite $\Omega$-system.
{\bf (B1)} By (2), every signature of $(G',\Sigma',T')$ has the same parity as $\tau(G,\Sigma,\{s,t\})$, implying that $(G',\Sigma',T')$ is an Eulerian signed graft.
{\bf (B2)} It also implies that $k, \tau(G,\Sigma,\{s,t\})$ and $\tau(G',\Sigma',T')$ have the same parity, so every minimal cover of $(G',\Sigma',T')$ has the same size parity as $k$.
We claim that $\tau(G',\Sigma',T')\geq k$.
Let $B'$ be a minimal cover of $(G',\Sigma',T')$. 
If $\Omega\notin B'$, then $$|B'|\geq \sum \left( |B'\cap L'|: L'\in \mathcal{L'}\right)\geq k.$$ Otherwise, $\Omega\in B'$. In this case, $B'\cup I_d$ contains a cover $B$ of $(G,\Sigma,\{s,t\})$. By (i), $I_d\cup \{\Omega\}$ does not have a $k$-mate, so $$k-2\leq |B-(I_d\cup \{\Omega\})|\leq |B-I_d|-1\leq |B'|-1,$$ and since $|B'|,k$ have the same parity, it follows that $|B'|\geq k$. Thus, $\mathcal{L'}$ is an $(\Omega,k)$-packing.
When $T'=\emptyset$ then $m'=3$.
When $T=\{s,t\}$, then $m'=m$ and for $j\in [m']-[3]$, $L'_j$ contains an even $st$-path in the bipartite $st$-join $L'_j-C_j$ and some odd circuit in $C_j$, and for $j\in [k]-[m']$, $L_j$ remains connected in $G'$.
{\bf (B3)} follows from construction.

Suppose first that $T'=\emptyset$.
We will show that $((G',\Sigma',\emptyset), \mathcal{L'}, 3, \vec{H'})$ is a non-simple bipartite $\Omega$-system, yielding a contradiction.
{\bf (NS1)} holds as (B1)-(B3) hold.
{\bf (NS2)} holds as $T'=\emptyset$.
{\bf (NS3)} follows from (3).
{\bf (NS4)} Let $L'$ be a directed odd circuit of $\vec{H'}$.
If $L'\cup I_d$ has a $k$-mate $B$ in $(G,\Sigma,\{s,t\})$ disjoint from $I_c$, then $B-I_d$ contains a minimal cover $B'$ of $(G',\Sigma',\emptyset)$, and since $$|B'-L'|\leq |(B-I_d)-L'|= |B-(L'\cup I_d)|\leq k-3,$$ it follows that $B'$ is a $k$-mate of $L'$.
Otherwise by (ii) $L'\cup I_d$ contains an odd $st$-dipath $P$ of $\vec{H}$ with $V(P)\cap U=\{s\}$. Since $((G,\Sigma,\{s,t\}),\mathcal{L},m,\mathcal{U},\vec{H})$ is a primary cut $\Omega$-system, $P$ has a $k$-mate $B$ which by proposition~\ref{cut-primary-signature} is a signature.
By proposition~\ref{bpsignature}, $B\cap E(\vec{H})=B\cap P$, implying that $B\cap I_c=\emptyset$.
Thus, $B-I_d$ contains a minimal cover $B'$ of $(G',\Sigma',\emptyset)$, and since $$|B'-L'|\leq |(B-I_d)-L'|\leq |B-P|\leq k-3,$$ it follows that $B'$ is a $k$-mate of $L'$.

Suppose otherwise that $T'=\{s,t\}$.
To obtain a contradiction, we will show that $((G',\Sigma',\{s,t\}), \mathcal{L'},$ $m, \mathcal{U}, \vec{H'})$ is a primary cut $\Omega$-system.
{\bf (C1)} holds because (B1)-(B3) are true.
{\bf (C2)}-{\bf (C3)} follow from (4).
{\bf (C4)} Let $P'$ be an odd $st$-dipath in $\vec{H'}$ with $V(P')\cap U=\{s\}$.
If $P'\cup I_d$ has a $k$-mate $B$ in $(G,\Sigma,\{s,t\})$ disjoint from $I_c$, then $B-I_d$ contains a minimal cover $B'$ of $(G',\Sigma',\{s,t\})$, and since $$|B'-P'|\leq |(B-I_d)-P'|= |B-(P'\cup I_d)|\leq k-3,$$ it follows that $B'$ is a $k$-mate of $P'$. Otherwise by (iii) $P'\cup I_d$ contains an odd $st$-dipath $P$ of $\vec{H}$. 
As $I_d\subseteq E(\vec{H}\setminus U)$, it follows that $V(P)\cap U=\{s\}$.
Since $((G,\Sigma,\{s,t\}),\mathcal{L},m,\mathcal{U},\vec{H})$ is a primary cut $\Omega$-system, $P$ has a $k$-mate $B$. By proposition~\ref{cut-primary-signature}, $B$ is a signature, so by proposition~\ref{bpsignature}, $B\cap E(\vec{H})=B\cap P$, implying that $B\cap I_c=\emptyset$.
Thus $B-I_d$ contains a minimal cover $B'$ of $(G',\Sigma',\{s,t\})$, and since $$|B'-P'|\leq |(B-I_d)-P'|\leq |B-P|\leq k-3,$$ it follows that $B'$ is a $k$-mate of $P'$.
\end{proof}
\subsection{The proof of proposition~\ref{prp-cut-primary}}
%
In this section, we prove proposition~\ref{prp-cut-primary}.
We assume $\Omega$ has ends $s,s'$.
Reset $C_1:=D$ and $Q_1:=Q$.
Let $Q^+_1$ be the $st$-dipath obtained from $Q_1$ after adding arc $(s,q)$.
For $i=4,\ldots,n+2$, let $Q^+_i$ be the $st$-dipath obtained from $Q_i$ after adding $(s,q_{i-3})$ to it.
Let $\vec{H}^+$ be the union of $C_1$, arc $(d,q)$ if $d\neq q$, and $st$-dipaths $Q^+_1,Q_2,Q_3,Q^+_4,\ldots,Q^+_{n+2},Q_{n+3},\ldots,Q_m$.
For $u,v\in V(Q^+_1\cup Q_2\cup Q_3\cup Q^+_4\cup \ldots\cup Q^+_{n+2}\cup Q_{n+3}\cup \ldots \cup Q_m)$, $u\leq v$ if there is a $uv$-dipath in $Q^+_1\cup Q_2\cup Q_3\cup Q^+_4\cup \ldots\cup Q^+_{n+2}\cup Q_{n+3}\cup \ldots \cup Q_m$; this partial ordering is well-defined as $\vec{H}^+$ is acyclic, by (C3). 
For $i\in [m]$, let $v_i$ be the second largest vertex of the $i^\text{th}$ $st$-dipath that lies on one of the other $st$-dipaths.
By proposition~\ref{intersection} there exists an index subset $I\subseteq [m]$ of size at least two such that, for each $i\in I$, \begin{itemize}
\item $v_i\geq v_3$,  and there is no $j\in [m]$ such that $v_j>v_i$,
\item for each $j\in [m]$, $v_i=v_j$ if and only if $j\in I$.
\end{itemize} 

\begin{claims} 
For each $i\in I$, $U$ and $Q_i[v_i,t]$ have no vertex in common.
\end{claims}
\begin{cproof}
Suppose otherwise.
Among the arcs of $\vec{H}$ in $\delta(U)$, there is only one arc, say $e$, entering $U$, and $e$ is the arc in $(C_1\cap \delta(U))-\{\Omega\}$.
However, $(Q_1\cup \cdots\cup Q_m)\cap C_1= \{\Omega\}$, implying that $e\notin \bigcup (Q_j:j\in [m])$.
In particular, $Q_i[v_i,t]$ does not enter $U$, so $v_i\in U$. 
As $v_i\geq v_3$, there is a $v_3v_i$-dipath $P\subset\bigcup (Q_j:j\in [m])$.
However, $v_3\in V(Q_3[s',t])$, so $v_3\notin U$, implying that $e\in P\subset \bigcup (Q_j:j\in [m])$, a contradiction.
\end{cproof}

\begin{claims} 
For each $i\in I$, $C_1$ and $Q_i[v_i,t]$ have no vertex of $V(G)-\{s'\}$ in common.
\end{claims}
\begin{cproof}
Suppose otherwise.
Then it follows from the brace proposition~\ref{braceprop} and the acyclicity of $\vec{H}^+$ that $$
(\diamond) \quad
I=\{2,3\}
\quad \text{and} \quad V(Q_i)\cap V(Q_j)\subseteq\{s,t\} \quad \forall~i\in I, \forall~j\in [m]-I.$$
Let $X_1:=C_1-\{\Omega\}, X_2:=Q_2-\{\Omega\}$ and $X_3:=Q_3-\{\Omega\}$.
For each $i\in [3]$, let $u_i$ be the second smallest vertex of $X_i$ that also lies on one of $\{X_1,X_2,X_3\}-\{X_i\}$.
Then by proposition~\ref{intersection}, there exists an index subset $J\subseteq [3]$ of size at least two such that, for each $j\in J$ and $i\in [3]$, $u_i=u_j$ if and only if $i\in J$. Observe that, for each $j\in J$, $X_j[s',u_j]\subseteq E(\vec{H}\setminus U)$, and as $(\diamond)$ holds, each internal vertex of $X_i[s',u_i]$ has degree $2$.

\begin{subclaim} 
For each $j\in J$, $X_j[s',u_j]\cup \{\Omega\}$ has a $k$-mate.
\end{subclaim}

\begin{subproof} 
Suppose otherwise. 
Let $I_d:=X_j[s',u_j]$ and $I_c:= \bigcup (X_i[s',u_i]: i\in J-\{j\})$.
Let $T':=\{s,t\}$, $D':=C_1-(I_c\cup I_d)$, and for $i=2,3$, let $L'_i:=L_i-(I_c\cup I_d)$.
Let $\vec{H'}\subseteq \vec{H}\setminus I_d/I_c$ be the union of $D',Q,L'_2,L'_3, Q_4,\ldots,Q_m$. 
It is clear that (1)-(4) of the disentangling lemma~\ref{cut-primary-disentangle} hold.
By assumption, $I_d\cup \{\Omega\}$ has no $k$-mate, so (i) holds.
However, since each internal vertex of $X_i[s',u_i]$ has degree $2$, so (ii) and (iii) hold as well, a contradiction with the disentangling lemma~\ref{cut-primary-disentangle}.
\end{subproof}

\begin{subclaim} 
Fix $j\in J$. Then there exist an $s't$-dipath $X$ and a $u_jt$-dipath $Y$ in $\vec{H}$ that are internally vertex-disjoint.
\end{subclaim}

\begin{subproof}
Suppose otherwise. Then there exists a vertex $v\in V(\vec{H})- \{s', t\}$ such that there is no $s't$-dipath in $\vec{H}\setminus v$. 
Note that $v\in V(C_1)$.
By proposition~\ref{cut-primary-signature}, one of the following holds: \begin{enumerate}[\;\;(a)]
\item there exists an $s'v$-dipath $Z$ in $\vec{H}$ such that $Z\cup \{\Omega\}$ has no $k$-mate: 
\begin{quote}
Let $I_d:=Z$, $I_c:= \bigcup (X_i[s',v]: i\in [3]) - Z$, $T':=\{s,t\}$, $D':=C_1-(I_c\cup I_d)$,
for $i=2,3$ let $L'_i:=L_i-(I_c\cup I_d)$,
and let $\vec{H'}\subseteq \vec{H}\setminus I_d/I_c$ be the union of $D',Q,L'_2,L'_3, Q_4,\ldots,Q_m$. 
\end{quote}

\item for every $s'v$-dipath $Z$ in $\vec{H}$, $Z\cup \{\Omega\}$ has a signature $k$-mate, and $m>3$: 
\begin{quote}
Let $I_d:=\emptyset$, $I_c:= Q_2[v,t]\cup Q_3[v,t]\cup Q_4\cup R_4$, $T':=\emptyset$,
for $i\in [3]$ let $L'_i:=Q_i[s',v]\cup \{\Omega\}$,
and let $\vec{H'}\subseteq \vec{H}\setminus I_d/I_c$ be the union of $L'_1,L'_2,L'_3$. 
\end{quote}

\item for every $s'v$-dipath $Z$ in $\vec{H}$, $Z\cup \{\Omega\}$ has a signature $k$-mate, and $m=3$.
\end{enumerate}
It is not difficult to check that in either of the cases (a), (b) above, (1)-(4) and (i) of the disentangling lemma~\ref{cut-primary-disentangle} hold, and as $(\diamond)$ holds, (ii) and (iii) hold as well, which cannot be the case.
(For (b), note that $V(R_4)\subseteq U$.)
Hence, (c) holds. For each $j\in [3]$, let $B_j$ be a signature $k$-mate for $Q_j[s',v]\cup \{\Omega\}$, which is also a signature $k$-mate for $L_j$.
However, this is in contradiction with the mate proposition~\ref{mateprop}. (Observe that $L_1$ is a connected odd $st$-join with $L_1\cap \delta(s)=\{\Omega\}$.)
\end{subproof}

\noindent Hence, in particular, $|J|=2$ and after redefining $\mathcal{L}$, if necessary, we may assume $J=\{1,2\}$ and $X=X_3$.

\begin{subclaim} 
$m>3$.
\end{subclaim}
\begin{subproof}
By subclaim~1, for $j=1,2$, there exists a $k$-mate $B_j$ of $Q_j[s',u_j]\cup \{\Omega\}$, and by (C4), $Q_3$ has a $k$-mate $B_3$.
By proposition~\ref{cut-primary-signature}, $B_1,B_2,B_3$ are signatures, and for $j\in [3]$, $B_j$ is also a $k$-mate for $L_j$. The result now follows from the mate proposition~\ref{mateprop}.
\end{subproof}

\noindent Now let $I_d:=\emptyset$, $I_c:= Y\cup Q_4\cup R_4$, $T':=\emptyset$,
$L'_1:=Q_1[s',u_1]\cup \{\Omega\}$, $L'_2:=Q_2[s',u_2]\cup \{\Omega\}$, $L'_3:=P_3$,
and let $\vec{H'}\subseteq \vec{H}\setminus I_d/I_c$ be the union of $L'_1,L'_2,L'_3$. (Note that $V(R_4)\subseteq U$.)
It is easy to check that (1)-(4) and (i)-(iii) of the disentangling lemma~\ref{cut-primary-disentangle} hold, which is a contradiction.
\end{cproof}

\begin{claims} 
For each $i\in I$, $Q_i[v_i,t]\cup \{\Omega\}$ has a signature $k$-mate.
\end{claims}
\begin{cproof}
Suppose otherwise. 
Since $v_i\geq v_3$, $Q_i[v_i,t]\cup \{\Omega\}$ is contained in an odd $st$-dipath $P$ such that $V(P)\cap U=\{s\}$.
Hence, by proposition~\ref{cut-primary-signature}, $Q_i[v_i,t]\cup \{\Omega\}$ has no $k$-mate at all.
Let $I_d:=Q_i[v_i,t]$ and $I_c:= \bigcup (Q_j[v_j,t]: j\in I-\{i\})$.
Let $T':=\{s,t\}$,
$Q':=Q_1-(I_c\cup I_d)$,
for $j=2,3$ let $L'_j:=L_j-(I_c\cup I_d)$,
and for $j\in [m]-[3]$ let $Q'_j:=Q_j-(I_c\cup I_d)$. 
Let $\vec{H'}\subseteq \vec{H}\setminus I_d/I_c$ be the union of $D,Q',L'_2,L'_3, Q'_4,\ldots,Q'_m$. 
It is clear that (1)-(4) and (ii), (iii) of the disentangling lemma~\ref{cut-primary-disentangle} hold.
However, $I_d\cup \{\Omega\}$ has no $k$-mate, so (i) holds, contradicting the disentangling lemma~\ref{cut-primary-disentangle}.
\end{cproof}

After redefining $\mathcal{L}$, if necessary, we may assume that $3\in I$.

\begin{claims} 
There exist vertex-disjoint paths $X$ and $Y$ in $\vec{H}$ such that
$X$ is an $s'v_3$-path in $\vec{H}\setminus U$
and $Y$ connects a vertex of $U$ to $t$.
\end{claims}
\begin{cproof}
Suppose otherwise.

Assume first that $s'=v_3$.
Then, for each $j\in [m]$, $s'\in V(Q_j)$ and by claim~1, $Q_j[s',t]$ has no vertex in common with $U$.
Hence, for each $j\in [m]$, by (C4) and proposition~\ref{cut-primary-signature}, $Q_j[s',t]\cup \{\Omega\}$ has a signature $k$-mate $B_j$.
However, $B_1$ is also a signature $k$-mate for $L_1$, and for each $j\in [m]-[1]$, $B_j$ is also a signature $k$-mate for $P_j\cup \{\Omega\}$. (Note $P_j-E(G[U])$ contains all the edges of $Q_j-E(G[U])$.) 
This is a contradiction with the mate proposition~\ref{mateprop}.

Thus, $s'\neq v_3$.
Let $\vec{H}^\star$ be the digraph obtained from $\vec{H}$ after shrinking $U$ to a single vertex $u^\star$ and removing all loops.
Notice that every odd $st$-dipath in $\vec{H}$ whose intersection with $U$ is $\{s\}$, is a $u^\star t$-dipath in $\vec{H}^\star$ that uses $\Omega$, and vice-versa.
Also, note that the acyclicity condition in (C3) implies that $\vec{H}^\star \setminus u^\star$ is acyclic.
By the linkage lemma~\ref{linkage}, $H^\star$ is a spanning subgraph of a $(u^\star,v_3,t,s')$-web with frame $C_0$ and rib $H^\star_0$.
Fix a plane drawing of $H^\star_0$, where the unbounded face is bounded by $C_0$.
After redefining $\mathcal{L}$, if necessary, we may assume the following: \begin{quote} $(\star)$ for every $s'v_3$-dipath $P$ of $\vec{H}^\star \setminus u^\star$, the number of rib vertices that are on the same side of $P$ as $u^*$ is at least as large as that of $Q_3[s',v_3]$. \end{quote} 

For $j\in [m]-\{2,3\}$, let $u_j$ be the largest rib vertex on $Q_j$ that also lies on $Q_3[s',v_3]$.
Observe that if $j\in I\cap ([m]-\{2,3\})$, then $u_j=v_j$.
For $j\in [m]-\{2,3\}$ let $X_j:=Q_j[u_j,t]$, 
for $j\in \{2,3\}\cap I$ let $X_j:=Q_j[v_j,t]$, 
and for $j\in \{2,3\}-I$ let $X_j:=Q_j[s',t]$.
For each $j\in [m]$, since $X_j\cup \{\Omega\}$ is contained in a $u^\star t$-dipath of $\vec{H}^\star$, proposition~\ref{cut-primary-signature} implies that every $k$-mate for $X_j\cup \{\Omega\}$ (if any) must be a signature.
However, every $k$-mate for $X_j\cup \{\Omega\}, j\in [m]$ is also a $k$-mate for $P_j\cup \{\Omega\}$.
Hence, by the mate proposition~\ref{mateprop}, there exists $i\in [m]$ such that $X_i\cup \{\Omega\}$ has no $k$-mate.
By (C4) and claim~3, $i\notin I\cup \{2,3\}$.
Observe that $(\star)$ implies the following: \begin{quote} $(\star\star)$ if $w\in V(Q_3[u_i,t])$ and $P$ is an $s'w$-dipath in $\vec{H}^\star\setminus u^\star$, then $P$ and $X_i$ have a vertex in common.
\end{quote} Observe that $(\star\star)$, together with the brace proposition~\ref{braceprop}, implies that $D=C_1$ is vertex-disjoint from $Q_3[u_i,t]$.

Let $I_d:=X_i$ and $I_c:=Q_3[u_i,t]$.
Let $T':=\{s,t\}$,
let $Q'$ be $Q_1-(I_c\cup I_d)$ minus any directed circuit (of $\vec{H}$) it contains,
for $j\in \{2,3\}$ let $L'_j$ be $Q_j-(I_c\cup I_d)$ minus any directed circuit it contains, 
and for $j\in [m]-[3]$ let $Q'_j$ be $Q_j-(I_c\cup I_d)$ minus any directed circuit it contains.
Let $\vec{H'}\subseteq \vec{H}\setminus I_d/I_c$ be the union of $D,Q',L'_2,L'_3, Q'_4,\ldots,Q'_m$. 
It is clear that (1)-(4) and (ii) of the disentangling lemma~\ref{cut-primary-disentangle} hold.
By the choice of $X_i$, (i) holds as well.
To show (iii) holds, let $P'$ be an odd $st$-dipath of $\vec{H'}$ with $V(P')\cap U=\{s\}$.
Then $P'\cup I_c$ contains an odd $st$-dipath of $\vec{H}$, 
so $P'\cup I_c$ contains a $u^\star t$-dipath of $\vec{H}^\star$ containing $\Omega$
and by $(\star\star)$, $P'\cup I_d$ also contains a $u^\star t$-dipath of $\vec{H}^\star$ containing $\Omega$, 
implying in turn that $P'\cup I_d$ contains an odd $st$-dipath of $\vec{H}$.
Hence, (iii) holds, a contradiction with the disentangling lemma~\ref{cut-primary-disentangle}.
\end{cproof}

For each $i\in I$, let $B_i$ be an extremal $k$-mate of $Q_i[v_i,t]\cup \{\Omega\}$. 
Note that $B_i\cap Q_i[v_i,t]\neq\emptyset$.
As $v_i\geq v_3$, $Q_i[v_i,t]\cup \{\Omega\}$ is contained in an odd $st$-dipath $P$ such that $V(P)\cap U=\{s\}$. 
Note that $B_i$ is also a $k$-mate for $P$, so by proposition~\ref{cut-primary-signature}, $B_i$ is a signature.
Fix $z\in I-\{3\}$.
Choose $W\subseteq V(G)-\{s,t\}$ such that $\delta(W)=B_3\triangle B_z$.
By proposition~\ref{matessignsign}, there is a path in $G[W]\setminus B_3$ between $Q_3$ and $Q_z$.
Moreover, by proposition~\ref{matescutsign}, there is a path between $s$ and each of $d,q$ in $G[U]\setminus B_3$.
We say that {\it property (S)} holds if there exist paths $S_d,S_q,S$ in $G$ such that \begin{enumerate}[\;\;]
\item $S_d$ is an $sd$-path and $S_q$ is an $sq$-path, and they are contained in $G[U]\setminus B_3$,
\item $S$ connects a vertex of $Q_3$ to a vertex of $Q_z$ in $G[W]\setminus B_3$, and
\item each of $S_d,S_q$ is vertex-disjoint from $S$.
\end{enumerate}

\begin{claims} 
If property (S) holds, then $(G,\Sigma,\{s,t\})$ has an $F_7$ minor.
\end{claims}
\begin{cproof}
Take $X$ and $Y$ from claim~4.
Notice that each edge in $Y\cap \delta(U)$ belongs to either of $D,Q,$ $P_4,\ldots,P_m$,
so we may assume that, for some $u\in \{s,d,q\}$, $Y$ is a $ut$-path.
It is now easy (and is left as an exercise) to see that $C_1\cup X\cup Y\cup S_d\cup S_q\cup Q_3[v_3,t]\cup Q_z[v_z,t]\cup S$ has an $F_7$ minor.
\end{cproof}

\begin{claims} 
Suppose property (S) does not hold.
Then $m\geq 4$.
\end{claims}
\begin{cproof}
Suppose for a contradiction that $m=3$.
By proposition~\ref{cut-primary-signature}, $L_2$ and $L_3$ have signature $k$-mates.
As $m=3$, the mate proposition~\ref{mateprop} therefore implies that $L_1$ does not have a signature $k$-mate.
Hence, by claim~3, $1\notin I$ and so $I=\{2,3\}$.
Since property (S) does not hold,
there is $u\in \{d,q\}$ for which there is no $su$-path contained in $G[U]\setminus (B_2\cup B_3)$.
Let $B_1:=\delta(U)$.
Clearly, (i) and (ii) of the shore proposition~\ref{shoreprop} hold.
By (PC5), (iii) also holds.
Moreover, for $i\in \{2,3\}$, $B_i\cap P_i=B_i\cap (Q_i[v_i,t]\cup \{\Omega\})$, so by claim~1, $B_i\cap P_i$ has no edge in $G[U]$, so (iv) holds.
Thus, by the shore proposition~\ref{shoreprop}, there is an $su$-path contained in $G[U]\setminus (B_2\cup B_3)$, a contradiction.
\end{cproof}

\begin{claims} 
Suppose property (S) does not hold. 
Then there exist vertex-disjoint paths $X$ and $Y$ in $\vec{H}$ where $X$ is an $s'v_3$-path and $Y$ is an $st$-path.
\end{claims}
\begin{cproof}
Suppose otherwise.
By claim~6, $m\geq 4$ and by the brace proposition~\ref{braceprop}, none of $Q_4,\ldots,Q_m$ contains vertex $s'$.
Thus, $s'\neq v_3$.
By the linkage lemma~\ref{linkage}, $H$ is a spanning subgraph of an $(s,v_3,t,s')$-web with frame $C_0$ and rib $H_0$.
Fix a plane drawing of $H_0$, where the unbounded face is bounded by $C_0$.
After redefining $\mathcal{L}$, if necessary, we may assume the following: \begin{quote} $(\star)$ for every $s'v_3$-dipath $P$ of $\vec{H}$ with $V(P)\cap U=\emptyset$, the number of rib vertices that are on the same side of $P$ as $s$ is at least as large as that of $Q_3[s',v_3]$. \end{quote} 

For $j\in [m]-[3]$, let $u_j$ be the largest rib vertex on $Q_j$ that also lies on $Q_3[s',v_3]$.
Observe that if $j\in I\cap ([m]-[3])$, then $u_j=v_j$.
For $j\in [m]-[3]$ let $X_j:=Q_j[u_j,t]$, 
for $j\in \{2,3\}\cap I$ let $X_j:=Q_j[v_j,t]$, 
and for $j\in \{2,3\}-I$ let $X_j:=Q_j[s',t]$.
Observe that each $X_j, j\in [m]-\{1\}$ is contained in an odd $st$-dipath whose intersection with $U$ is $\{s\}$.
As a result, by proposition~\ref{cut-primary-signature}, every $k$-mate for $X_j\cup \{\Omega\}, j\in [m]-\{1\}$ (if any) must be a signature.
However, every $k$-mate for $X_j\cup \{\Omega\}, j\in [m]-\{1\}$ is also a $k$-mate for $P_j\cup \{\Omega\}$.
Hence, since property (S) does not hold, the (contrapositive equivalent of the) shore proposition~\ref{shoreprop} implies that, for some $i\in [m]-\{1\}$, $X_i\cup \{\Omega\}$ has no $k$-mate.
By (C4) and claim~3, $i\notin I\cup \{2,3\}$.
Observe that $(\star)$ implies the following: \begin{quote} $(\star\star)$ if $w\in V(Q_3[u_i,t])$ and $P$ is an $s'w$-dipath in $\vec{H}\setminus U$, then $P$ and $X_i$ have a vertex in common.
\end{quote} Note that $(\star\star)$, together with the brace proposition~\ref{braceprop}, implies that $C_1=D$ is vertex-disjoint from $Q_3[u_i,t]$.

Let $I_d:=X_i$ and $I_c:=Q_3[u_i,t]$.
Let $T':=\{s,t\}$,
let $Q'$ be $Q_1-(I_c\cup I_d)$ minus any directed circuit it contains,
for $j\in \{2,3\}$ let $L'_j$ be $Q_j-(I_c\cup I_d)$ minus any directed circuit it contains, 
and for $j\in [m]-[3]$ let $Q'_j$ be $Q_j-(I_c\cup I_d)$ minus any directed circuit it contains.
Let $\vec{H'}\subseteq \vec{H}\setminus I_d/I_c$ be the union of $D,Q',L'_2,L'_3, Q'_4,\ldots,Q'_m$. 
It is clear that (1)-(4) and (ii) of the disentangling lemma~\ref{cut-primary-disentangle} hold.
By the choice of $X_i$, (i) holds as well.
To show (iii) holds, let $P'$ be an odd $st$-dipath of $\vec{H'}$ with $V(P')\cap U=\{s\}$.
Then $P'\cup I_c$ contains an odd $st$-dipath of $\vec{H}$ whose intersection with $U$ is $\{s\}$, 
so by $(\star\star)$, $P'\cup I_d$ also contains an $st$-dipath of $\vec{H}$.
Hence, (iii) holds, a contradiction with the disentangling lemma~\ref{cut-primary-disentangle}.
\end{cproof}

\begin{claims} 
Suppose property (S) does not hold.
If $\vec{H}\setminus t$ is non-bipartite, then $(G,\Sigma,\{s,t\})$ has an $F_7$ minor.
\end{claims}
\begin{cproof}
Take $X$ and $Y$ from claim~7, and let $C$ be an odd circuit in $\vec{H}\setminus t$.
Note that $\Omega\in C$.
By proposition~\ref{matessignsign}, there is a shortest path $S$ in $G[W]\setminus B_3$ between $P_3$ and $P_z$.
Note that $S\cap E(H)=\emptyset$.
It is easy (and is left as an exercise) to see that $C\cup X\cup Y\cup P_3[v_3,t]\cup P_z[v_z,t]\cup S$ has an $F_7$ minor.
\end{cproof}

Notice that if $\vec{H}\setminus t$ is bipartite, then for all $i\in \{2,3\}$ and $j\in [m]-[3]$, $Q_i$ and $Q_j\cup R_j$ are internally vertex-disjoint.

We say that {\it property (S')} holds if there exist vertex-disjoint paths $S_d,S$ in $G$ such that \begin{enumerate}[\;\;]
\item $S_d$ is an $sd$-path in $G[U]\setminus B_3$,
\item $S$ connects a vertex of $P_3$ to a vertex of $P_z$ in $G[W]\setminus B_3$.
\end{enumerate} Notice that if property (S') does not hold, then neither does property (S).

\begin{claims} 
Suppose property (S) does not hold, $\vec{H}\setminus t$ is bipartite, and property (S') holds.
Then $(G,\Sigma,\{s,t\})$ has an $F_7$ minor.
\end{claims}
\begin{cproof}
By claim~6, $m\geq 4$. Note that $Q_4\cup R_4$ is internally vertex-disjoint from $Q_3$.
It is easy to see that $C_1\cup S_d\cup Q_3\cup Q_z[v_z,t]\cup  S\cup Q_4\cup R_4$ has an $F_7$ minor.
\end{cproof}

\begin{claims} 
Suppose property (S') does not hold and that $\vec{H}\setminus t$ is bipartite.
Then $(G,\Sigma,\{s,t\})$ has an $F_7$ minor.
\end{claims}
\begin{cproof}
We will find an $F_7$ minor in a different way than we have done so far, by using edges from $L_{m+1},\ldots,L_k$.

Since property (S') does not hold, there does not exist a path connecting a vertex of $Q_3$ to a vertex of $Q_z$ in $G[W-U]\setminus B_3$.
So there is a partition of $W-U$ into two parts $W_3,W_z$ such that $W_3$ shares no vertex with $Q_z$, $W_z$ shares no vertex with $Q_3$, and every edge with one end in $W_3$ and another in $W_z$ belongs to $B_3$. Observe that $\delta(W_3)\cup \delta(W_z)\subseteq B_3\cup B_z\cup \delta(U)$.

\begin{subclaim} 
There is no edge with one end in $W_3$ and another in $W_z$.
\end{subclaim}
\begin{subproof}
Suppose otherwise, and let $e$ be such an edge. Then $e\in B_3$, and since $e\notin \delta(W)$, it follows that $e\in B_z$.
Note $e\in C_4\cup \cdots \cup C_m\cup L_{m+1}\cup \cdots \cup L_k$, and since each of $L_{m+1},\ldots,L_k$ is a connected odd $st$-join intersecting each of $B_3,B_z$ exactly once, it follows that $e\in C_4\cup \cdots \cup C_m$.
We may assume $e\in C_4$. However, $C_4\cap \delta(U)=\emptyset$, implying that there is another edge $f$ of $C_4$ with one end in $W_3$ and another in $W_z$. But then $\{e,f\}\subseteq C_4\cap B_3$, a contradiction as $|C_4\cap B_3|=1$.
\end{subproof}

Given $L\in \{L_{m+1},\ldots,L_k\}$ and $Q_j\in \{Q_3,Q_z\}$, we say that {\it $L$ is bad for $Q_j$} if $|L\cap \delta(W_j)|=2$, $L\cap \delta(W_j)\cap B_j=\emptyset$, and there exists a path in $G[W_j]\setminus B_3$ between $Q_j$ and $L$.

\begin{subclaim} 
Each $L\in \{L_{m+1},\ldots,L_k\}$ is bad for at most one of $Q_3,Q_z$.
\end{subclaim}
\begin{subproof}
Suppose otherwise. Then $|L\cap \delta(W_3)|=|L\cap \delta(W_z)|=2$, and by subclaim~1, $L$ shares exactly four edges with $\delta(W_3)\cup \delta(W_z)$. However, $\delta(W_3)\cup \delta(W_z)\subseteq B_3\cup B_z\cup \delta(U)$, implying that $L$ shares at least two edges with one of $B_3,B_z,\delta(U)$, a contradiction.
\end{subproof}

\begin{subclaim} 
Each of $Q_3,Q_z$ has a bad odd $st$-join.
\end{subclaim}
\begin{subproof}
We prove that $Q_3$ has a bad odd $st$-join, and proving $Q_z$ has a bad odd $st$-join can be done similarly. Suppose for a contradiction that $Q_3$ has no bad odd $st$-join.
Let $W'_3$ be the set of all vertices in $W_3$ that are reachable from a vertex of $Q_3$ in $G[W_3]\setminus B_3$.
A similar argument as in subclaim~1 shows that there is no edge with one end in $W'_3$ and another in $W_3-W'_3$.
Moreover, our contrary assumption implies that, for every $L\in \{L_{m+1},\ldots,L_k\}$ such that $L\cap \delta(W'_3)\neq \emptyset$, we have $$|L\cap \delta(W'_3)|=2 \quad \text{and} \quad |L\cap \delta(W'_3)\cap B_3|=1.$$ This implies that $B_3\triangle \delta(W'_3)$ is also a $k$-mate of $Q_3[v_3,t]\cup \{\Omega\}$. However, $(B_3\triangle \delta(W'_3))\cap Q_3[v_3,t]=\emptyset$, contradicting the extremality of $B_3$.
\end{subproof}

\begin{subclaim} 
$(G,\Sigma,\{s,t\})$ has an $F_7$ minor.
\end{subclaim}
\begin{subproof}
Since property (S') does not hold, there is no path in $G[U-W]\setminus B_3$ between $s$ and $d$. So there is a partition $U_s,U_d$ of $U-W$ such that $U_s$ contains $s$, $U_d$ contains $d$, and every edge with one end in $U_s$ and another in $U_d$ belongs to $B_3$.

By proposition~\ref{matescutsign}, there is a path $S_d$ between $s$ and $d$ in $G[U]\setminus B_3$.
By proposition~\ref{matessignsign}, there is a shortest path $S$ in $G[W]\setminus B_3$ between $Q_3$ and $Q_z$.
Suppose $S$ has ends $r_3\in V(Q_3)$ and $r_z\in V(Q_z)$.
Since property (S') does not hold, $S$ and $S_d$ have a vertex in common in $U\cap W$.
After contracting edges in $G[U_s]\setminus B_3$, if necessary, we may assume that $S_d$ and $P_4$ share only the vertex $s$. (We may assume $P_4\subseteq Q_4\cup R_4$.)

By subclaims~2 and~3, we may assume that $L_{m+1}$ is bad for $Q_3$ and that $L_{m+2}$ is bad for~$Q_z$. After contracting the path between $L_{m+1},Q_3$ in $G[W_3]\setminus B_3$ and the path between $L_{m+2},Q_z$ in $G[W_z]\setminus B_3$, we may assume that $r_3\in V(L_{m+1})$ and $r_z\in V(L_{m+2})$.
After contracting edges in $G[U_s]\setminus B_3$, if necessary, we may assume that $L_{m+1}$ and each one of $P_4,S_d$ share only the vertex $s$ in $U_s$.
Similarly, we may assume that $L_{m+2}$ and $S$ share only the vertex $r_z$ in $W_z$.

To construct the desired $F_7$ minor, we will need three odd circuits and an even $st$-path, described as follows.

\noindent {\bf Even $st$-path:}
Our even $st$-path will be $P_4$. Recall that $P_4$ is internally vertex-disjoint from each one of $Q_2,Q_3,Q_z[v_z,t]$. Moreover, by the brace proposition~\ref{braceprop}, $V(P_4)\cap V(C_1)\subseteq \{s,d\}$. In fact, since property (S') does not hold, $V(P_4)\cap V(C_1)=\{s\}$. In fact, notice that $P_4$ has no vertex in common with $U_d\cup W$.

\noindent {\bf Middle odd circuit:}
Along $S_d$, let $x$ be the closest vertex to $d$ that also lies on $S$.
Note that $x\in U\cap W$.
Our middle circuit will be $$C_\text{middle}:=S_d[d,x]\cup S[x,r_3]\cup Q_3[r_3,s']\cup C_1[s',d].$$ Observe that the even $st$-path $P_4$ is vertex-disjoint from $C_\text{middle}$. Moreover, $C_\text{middle}\cap B_3 = Q_3[r_3,s']\cap B_3$, so $C_\text{middle}$ is an odd circuit.

\noindent {\bf First odd circuit:} Our first odd circuit $C_\text{first}$ will be one contained in the odd cycle $$S_d[s,x]\cup S[x,r_3]\cup L_{m+1}[r_3,s].$$ 
(The intersection of this cycle with $B_3$ is $L_{m+1}[r_3,s]\cap B_3$, so the cycle is indeed odd.)
Note that $C_\text{first}$ is contained in $G[U\cup W]$.

\noindent {\bf Last odd circuit:}
Our last odd circuit $C_\text{last}$ will be one contained in the set $$L_{m+2}[r_z,t]\cup Q_3[s',v_3]\cup Q_z[v_3,t]$$ whose intersection with $B_3$ is $B_3\cap L_{m+2}[r_z,t]$, which has odd cardinality.
Note that $V(C_\text{last})$ is contained in $(V(G)-(U\cup W))\cup W_z$. However, as can be easily seen, $C_\text{first}$ and $C_\text{last}$ share no vertex in $W_z$. Hence, $C_\text{first}$ and $C_\text{last}$ have no vertex in common.

It is now quite easy to see that $(G,\Sigma,\{s,t\})$ has an $F_7$ minor, finishing the proof.\end{subproof}
\end{cproof}

Observe that claims~5,~8,~9 and~10 finish the proof of proposition~\ref{prp-cut-primary}.
\section{Secondary cut $\Omega$-system}\label{sec-cut-secondary}
%
\subsection{Signature mates}
%
\begin{prp}\label{cut-secondary-signature}
Let $((G,\Sigma,\{s,t\}), \mathcal{L}=(L_1,\ldots,L_k), m, (U_1,\ldots,U_n), \vec{H})$ be a minimal cut $\Omega$-system that is secondary.
Let $P$ be an odd $st$-dipath with $V(P)\cap U_n=\{s\}$, and let $B$ be a $k$-mate of it.
Then $B$ is not an $st$-cut.
\end{prp}
\begin{proof}
After redefining $\mathcal{L}$, if necessary, we may assume that $P=Q_1$.
Suppose for a contradiction that $B$ is an $st$-cut. Choose $W\subseteq V(G)-\{t\}$ with $s\in W$ such that $B=\delta(W)$, and assume that there is no proper subset $W'$ of $W$ with $s\in W'$ such that $\delta(W')$ is a $k$-mate for $Q_1=L_1$.
Observe that $Q_1\cap \delta(U_n)=\{\Omega\}$, and since $Q_1$ is an odd $st$-dipath, it follows that $Q_1\cap \delta(U_n\cap W)=\{\Omega\}$. It now follows that $\delta(U_n\cap W)$ is also a $k$-mate for $L_{n+3}-C_{n+3}$.
Hence, by the minimality condition of (SC3), it follows that $U_n\subset W$.
Let $\mathcal{U}:=(U_1,\ldots,U_n,W)$.
Let $d$ (resp. $q$) be the closest (resp. furthest) vertex to (resp. from) $s$ on $Q_1$ that also belongs to $W-U_n$.
It is easily seen that $\mathcal{U}$ is a primary cut structure for $((G,\Sigma,\{s,t\}),\mathcal{L},m)$, where $L_1$ has brace $Q_1[s,d]$, residue $Q_1[d,q]$ and base $Q_1[q,t]$.
Let $\vec{H'}:=\vec{H}\setminus Q_1[d,q]$.
Then it is easily seen that $((G,\Sigma,\{s,t\}), \mathcal{L}, m, \mathcal{U}, \vec{H'})$ is a primary cut structure, contradicting the minimality of the original $\Omega$-system.
\end{proof}
\subsection{A disentangling lemma}
%
\begin{lma}\label{cut-secondary-disentangle}
Let $((G,\Sigma,\{s,t\}), \mathcal{L}=(L_1,\ldots,L_k), m, \mathcal{U}=(U_1,\ldots,U_n), \vec{H})$ be a minimal cut $\Omega$-system that is secondary, and assume there is no non-simple bipartite $\Omega$-system whose associated signed graft is a minor of $(G,\Sigma,\{s,t\})$.
Take disjoint subsets $I_d,I_c\subseteq E(\vec{H}\setminus \Omega)$ and $T'\subseteq \{s,t\}$ where \begin{enumerate}[\;\;(1)]
\item $I_c$ is non-empty, if $I_c$ contains an $st$-path then $T'=\emptyset$, and if not then $T'=\{s,t\}$,
\item every signature or $st$-cut disjoint from $I_c$ intersects $I_d$ in an even number of edges,
\item if $T'=\emptyset$, there is a directed subgraph $\vec{H'}$ of $\vec{H}/I_c\setminus I_d$ that is the union of directed odd circuits $L'_1,L'_2,L'_3$ where \begin{enumerate}[\;\;]
\item $\Omega\in L'_1\cap L'_2\cap L'_3$ and $L'_1,L'_2,L'_3$ are pairwise $\Omega$-disjoint,
\item $\vec{H'}\setminus \Omega$ is acyclic,
\end{enumerate}
\item if $T'=\{s,t\}$, then $I_d,I_c\subseteq E(\vec{H}\setminus U_n)$ and there is a directed subgraph $\vec{H'}$ of $\vec{H}/I_c\setminus I_d$ that is the union of odd $st$-dipaths $L'_1,L'_2, L'_3$ and 
dipaths $Q'_4,\ldots,Q'_m$, where \begin{enumerate}[\;\;-]
\item for $i=4,\ldots,n+3$, $Q'_i$ is a $q_{i-3}t$-dipath with $V(Q'_i)\cap U_{i-3}=\{q_{i-3}\}$, and for $i=n+4,\ldots,m$, $Q'_i$ is an even $st$-dipath,
\item $L'_1,L'_2, L'_3,Q'_4, \ldots , Q'_m$ are pairwise $\Omega$-disjoint,
\item $L'_1,L'_2,L'_3,Q'_4,\ldots,Q'_m$ coincide with $L_1,L_2,L_3,Q_4,\ldots,Q_m$ on $E(G[U_n])\cup \delta(U_n)$, respectively,
\item the following digraph is acyclic: start from $\vec{H'}$, and for each $q_i$ add arc $(s,q_i)$.
\end{enumerate}
\end{enumerate}
Then one of the following does not hold: 
\begin{enumerate}[\;\;(i)]
\item $I_d\cup \{\Omega\}$ does not have a $k$-mate,
\item if $T'=\emptyset$, then for every directed odd circuit $L'$ of $\vec{H'}$, either $L'\cup I_d$ contains an odd $st$-dipath $P$ of $\vec{H}$ with $V(P)\cap U_n=\{s\}$, or $L'\cup I_d$ has a $k$-mate in $(G,\Sigma,\{s,t\})$ disjoint from $I_c$,
\item if $T'=\{s,t\}$, then for every odd $st$-dipath $P'$ of $\vec{H'}$ with $V(P')\cap U_n=\{s\}$, either $P'\cup I_d$ contains an odd $st$-dipath of $\vec{H}$, or $P'\cup I_d$ has a $k$-mate in $(G,\Sigma,\{s,t\})$ disjoint from~$I_c$.
\end{enumerate} 
\end{lma}
\begin{proof}
Suppose otherwise.
Let $(G',\Sigma',T'):=(G,\Sigma,\{s,t\})/I_c\setminus I_d$ where $\Sigma'=\Sigma$; this signed graft is well-defined by (1).
Let $\mathcal{L'}:=(L'_1,\ldots,L'_m, L_{m+1},\ldots,L_k)$, where $L'_1,\ldots,L'_m$ are defined as follows.
If $T'=\emptyset$,
let $m':=3$, 
and for $i\in [m]-[3]$, let $L'_i:=L_i-P_i$.
Otherwise, when $T'=\{s,t\}$, 
let $m':=m$,
and for $i\in [m]-[3]$, let $L'_i:=(L_i-Q_i)\cup Q'_i$.

We will first show that $((G',\Sigma',T'), \mathcal{L'}, m')$ is a bipartite $\Omega$-system.
{\bf (B1)} By (2), every signature of $(G',\Sigma',T')$ has the same parity as $\tau(G,\Sigma,\{s,t\})$, implying that $(G',\Sigma',T')$ is an Eulerian signed graft.
{\bf (B2)} It also implies that $k, \tau(G,\Sigma,\{s,t\})$ and $\tau(G',\Sigma',T')$ have the same parity, so every minimal cover of $(G',\Sigma',T')$ has the same size parity as $k$.
We claim that $\tau(G',\Sigma',T')\geq k$.
Let $B'$ be a minimal cover of $(G',\Sigma',T')$. 
If $\Omega\notin B'$, then $$|B'|\geq \sum \left( |B'\cap L'|: L'\in \mathcal{L'}\right)\geq k.$$ Otherwise, $\Omega\in B'$. In this case, $B'\cup I_d$ contains a cover $B$ of $(G,\Sigma,\{s,t\})$. By (i), $I_d\cup \{\Omega\}$ does not have a $k$-mate, so $$k-2\leq |B-(I_d\cup \{\Omega\})|\leq |B-I_d|-1\leq |B'|-1,$$ and since $|B'|,k$ have the same parity, it follows that $|B'|\geq k$. Thus, $\mathcal{L'}$ is an $(\Omega,k)$-packing.
When $T'=\emptyset$ then $m'=3$.
When $T=\{s,t\}$, then $m'=m$ and for $j\in [m']-[3]$, $L'_j$ contains an even $st$-path in the bipartite $st$-join $L'_j-C_j$ and some odd circuit in $C_j$, and for $j\in [k]-[m']$, $L_j$ remains connected in $G'$.
{\bf (B3)} follows from construction.

Suppose first that $T'=\emptyset$.
We will show that $((G',\Sigma',\emptyset), \mathcal{L'}, 3, \vec{H'})$ is a non-simple bipartite $\Omega$-system, yielding a contradiction.
{\bf (NS1)} holds as (B1)-(B3) hold.
{\bf (NS2)} holds as $T'=\emptyset$.
{\bf (NS3)} follows from (3).
{\bf (NS4)} Let $L'$ be a directed odd circuit of $\vec{H'}$.
If $L'\cup I_d$ has a $k$-mate $B$ in $(G,\Sigma,\{s,t\})$ disjoint from $I_c$, then $B-I_d$ contains a minimal cover $B'$ of $(G',\Sigma',\emptyset)$, and since $$|B'-L'|\leq |(B-I_d)-L'|= |B-(L'\cup I_d)|\leq k-3,$$ it follows that $B'$ is a $k$-mate of $L'$.
Otherwise by (ii) $L'\cup I_d$ contains an odd $st$-dipath $P$ of $\vec{H}$ with $V(P)\cap U_n=\{s\}$. Since $((G,\Sigma,\{s,t\}),\mathcal{L},m,\mathcal{U},\vec{H})$ is a minimal secondary cut $\Omega$-system, $P$ has a $k$-mate $B$ which by proposition~\ref{cut-secondary-signature} is a signature.
By proposition~\ref{bpsignature}, $B\cap E(\vec{H})=B\cap P$, implying that $B\cap I_c=\emptyset$.
Thus, $B-I_d$ contains a minimal cover $B'$ of $(G',\Sigma',\emptyset)$, and since $$|B'-L'|\leq |(B-I_d)-L'|\leq |B-P|\leq k-3,$$ it follows that $B'$ is a $k$-mate of $L'$.

Suppose otherwise that $T'=\{s,t\}$.
To obtain a contradiction, we will show that $((G',\Sigma',\{s,t\}), \mathcal{L'},$ $m, \mathcal{U}, \vec{H'})$ is a secondary cut $\Omega$-system.
{\bf (C1)} holds because (B1)-(B3) are true.
{\bf (C2)}-{\bf (C3)} follow from (4).
{\bf (C4)} Let $P'$ be an odd $st$-dipath in $\vec{H'}$ with $V(P')\cap U=\{s\}$.
If $P'\cup I_d$ has a $k$-mate $B$ in $(G,\Sigma,\{s,t\})$ disjoint from $I_c$, then $B-I_d$ contains a minimal cover $B'$ of $(G',\Sigma',\{s,t\})$, and since $$|B'-P'|\leq |(B-I_d)-P'|= |B-(P'\cup I_d)|\leq k-3,$$ it follows that $B'$ is a $k$-mate of $P'$. Otherwise by (iii) $P'\cup I_d$ contains an odd $st$-dipath $P$ of $\vec{H}$. 
As $I_d\subseteq E(\vec{H}\setminus U)$, it follows that $V(P)\cap U=\{s\}$.
Since $((G,\Sigma,\{s,t\}),\mathcal{L},m,\mathcal{U},\vec{H})$ is a minimal secondary cut $\Omega$-system, $P$ has a $k$-mate $B$. By proposition~\ref{cut-secondary-signature}, $B$ is a signature, so by proposition~\ref{bpsignature}, $B\cap E(\vec{H})=B\cap P$, implying that $B\cap I_c=\emptyset$.
Thus $B-I_d$ contains a minimal cover $B'$ of $(G',\Sigma',\{s,t\})$, and since $$|B'-P'|\leq |(B-I_d)-P'|\leq |B-P|\leq k-3,$$ it follows that $B'$ is a $k$-mate of $P'$.
\end{proof}
\subsection{The proof of proposition~\ref{prp-cut-secondary}}
%
In this section, we prove proposition~\ref{prp-cut-secondary}.
We assume $\Omega$ has ends $s,s'$.
For $i=4,\ldots,n+3$, let $Q^+_i$ be the $st$-dipath obtained from $Q_i$ after adding arc $(s,q_{i-3})$ to it.
Let $\vec{H}^+$ be the union of $Q_1,Q_2,Q_3,Q^+_4,\ldots,Q^+_{n+3},Q_{n+4},\ldots,Q_m$.
For $u,v\in V(\vec{H}^+)$, $u\leq v$ if there is a $uv$-dipath in $\vec{H}^+$. This partial ordering is well-defined as $\vec{H}^+$ is acyclic, by (C3).
For $i\in [m]$, let $v_i$ be the second largest vertex of the $i^\text{th}$ dipath that lies on one of the other $st$-dipaths.
By proposition~\ref{intersection}, there exists an index subset $I\subseteq [m]$ of size at least two such that, for each $i\in I$, \begin{itemize}
\item $v_i\geq v_1$, and there is no $j\in [m]$ such that $v_j>v_i$,
\item for each $j\in [m]$, $v_i=v_j$ if and only if $j\in I$.
\end{itemize}

\begin{claims} 
For each $i\in I$, $Q_i[v_i,t]$ and $U_n$ have no vertex in common.
\end{claims}
\begin{cproof}
Suppose otherwise. Since $\vec{H}$ has no arc entering $U_n$, it follows that $v_i\in U_n$. As $v_i\geq v_1$, there is a $v_1v_i$-dipath $P\subset E(\vec{H})$. However, as $v_1\in V(Q_1[s',t])$, so $v_1\notin U$, implying that $P$ has an arc that enters $U_n$, a contradiction.
\end{cproof}

\begin{claims} 
For each $i\in I$, $Q_i[v_i,t]\cup \{\Omega\}$ has a signature $k$-mate.
\end{claims}
\begin{cproof}
Suppose otherwise. 
Since $v_i\geq v_1$, $Q_i[v_i,t]\cup \{\Omega\}$ is contained in an odd $st$-dipath $P$ such that $V(P)\cap U_n=\{s\}$.
Hence, by proposition~\ref{cut-secondary-signature}, $Q_i[v_i,t]\cup \{\Omega\}$ has no $k$-mate at all.
Let $I_d:=Q_i[v_i,t]$ and $I_c:= \bigcup (Q_j[v_j,t]: j\in I-\{i\})$.
Let $T':=\{s,t\}$,
for $j\in [3]$ let $L'_j:=Q_j-(I_c\cup I_d)$,
and for $j\in [m]-[3]$ let $Q'_j:=Q_j-(I_c\cup I_d)$. 
Let $\vec{H'}\subseteq \vec{H}\setminus I_d/I_c$ be the union of $L'_1,L'_2,L'_3, Q'_4,\ldots,Q'_m$. 
It is clear that (1)-(4) and (ii), (iii) of the disentangling lemma~\ref{cut-secondary-disentangle} hold.
However, $I_d\cup \{\Omega\}$ has no $k$-mate, so (i) holds, contradicting the disentangling lemma~\ref{cut-secondary-disentangle}.
\end{cproof}

After redefining $\mathcal{L}$, if necessary, we may assume that $1\in I$.

\begin{claims} 
If $m=4$, then $I\subseteq [3]$.
\end{claims}
\begin{cproof}
Suppose otherwise. By claim~2, there exists a signature $k$-mate $B_4$ for $Q_4[v_4,t]\cup \{\Omega\}$. By (C4) and proposition~\ref{cut-secondary-signature}, for each $i\in [3]$, there exists a signature $k$-mate $B_i$ for $Q_i$. However, $B_1,B_2,B_3,B_4$ contradict the mate proposition~\ref{mateprop}.
\end{cproof}

\begin{claims} 
Suppose $m=4$. Then there exists $i\in [3]$ such that $Q_i$ and $Q_4$ are not internally vertex-disjoint.
\end{claims}
\begin{cproof}
Suppose for a contradiction that $Q_4$ is internally vertex-disjoint from $Q_1\cup Q_2\cup Q_3$. Notice that $I\subseteq [3]$, by claim~3.

\begin{subclaim} 
There exist an $s'v_1$-dipath $X$ and an $s't$-dipath $Y$ in $\vec{H}$ that are internally vertex-disjoint.
\end{subclaim}
\begin{subproof}
Suppose otherwise. Then $s'\neq v_1$ and there exists a vertex $v\in V(\vec{H})- \{s', t\}$ such that there is no $s't$-dipath in $\vec{H}\setminus v$. 
By proposition~\ref{cut-secondary-signature}, one of the following holds: \begin{enumerate}[\;\;(a)]
\item there exists an $s'v$-dipath $Z$ in $\vec{H}$ such that $Z\cup \{\Omega\}$ has no $k$-mate: 
\begin{quote}
Let $I_d:=Z$, $I_c:= \bigcup (Q_i[s',v]: i\in [3]) - Z$, $T':=\{s,t\}$, 
for $i\in [3]$ let $L'_i:=Q_i-(I_c\cup I_d)$,
and let $\vec{H'}\subseteq \vec{H}\setminus I_d/I_c$ be the union of $L'_1,L'_2,L'_3, Q_4$. 
\end{quote}

\item for every $s'v$-dipath $Z$ in $\vec{H}$, $Z\cup \{\Omega\}$ has a signature $k$-mate: 
\begin{quote}
Let $I_d:=\emptyset$, $I_c:= \bigcup (Q_i[v,t]: i\in [3])$, $T':=\{s,t\}$,
for $i\in [3]$ let $L'_i:=Q_i[s',v]\cup \{\Omega\}$,
and let $\vec{H'}\subseteq \vec{H}\setminus I_d/I_c$ be the union of $L'_1,L'_2,L'_3,Q_4$. 
\end{quote}
\end{enumerate}
It is not difficult to check that in either of the cases above, (1)-(4) and (i)-(iii) of the disentangling lemma~\ref{cut-secondary-disentangle} hold, a contradiction.
\end{subproof}

After redefining $\mathcal{L}$, if necessary, we may assume that $\{1,2\}\subseteq I$ and $Y=Q_3[s',t]$.
For $i=1,2$, let $B_i$ be a signature $k$-mate for $Q_i[v_i,t]\cup \{\Omega\}$, whose existence is guaranteed by claim~2.
Moreover, by (C4) and proposition~\ref{cut-secondary-signature}, $Q_3$ has a signature $k$-mate $B_3$.
Observe that by proposition~\ref{bpsignature}, for each $i\in [3]$, $B_i\cap (Q_4\cup X)=\emptyset$.

\begin{subclaim} 
There exists a path $R$ between $s$ and $Q_4$ in $G[U_n]\setminus (B_1\cup B_2\cup B_3)$.
\end{subclaim}
\begin{subproof}
This is an immediate consequence of the shore proposition~\ref{shoreprop} and the fact that $m=4$.
\end{subproof}

Let $I_c:=R\cup Q_4\cup X$ and $I_d:=\emptyset$. Let $T':=\emptyset$, for $i=1,2$ let $L'_i:=Q_i[v_i,t]\cup \{\Omega\}$, and let $L'_3:=Q_3$.
Let $\vec{H'}\subseteq \vec{H}\setminus I_d/I_c$ be the union of $L'_1,L'_2,L'_3$.
Note that $L'_1,L'_2,L'_3$ are internally vertex-disjoint in $\vec{H'}$ and have signature $k$-mates $B_1,B_2,B_3$, respectively.
It is now clear that (1)-(4) and (i)-(iii) of the disentangling lemma~\ref{cut-secondary-disentangle} hold, a contradiction.
\end{cproof}

\begin{claims} 
Suppose $m=4$. Then there exists an $s'v_1$-dipath $P$ in $\vec{H}$ that is vertex-disjoint from $Q_4$.
\end{claims}
\begin{cproof}
By claim~3, $I\subseteq [3]$. 
Suppose for a contradiction that there is no $s'v_1$-dipath in $\vec{H}$ that is vertex-disjoint from $Q_4$.
Let $v$ be the smallest vertex of $Q_4$ outside $U_n$ for which there exists a $vv_1$-dipath $R$ in $\vec{H}$ such that $V(R)\cap V(Q_4)=\{v\}$.
Our contrary assumption together with the choice of $v$ and $R$, implies the following: \begin{quote} $(\star)$ if $w\in V(R)$ and $Q$ is an $s'w$-dipath in $\vec{H}$, then $Q$ and $Q_4[v,t]$ have a vertex in common.\end{quote}

Let $I_d:=Q_4[v,t]$ and $I_c:=R\cup \left[\bigcup (Q_j[v_j,t]: j\in I)\right]$. 
For $i\in [3]$ let $L'_i$ be $Q_i-(I_c\cup I_d)$ minus any directed circuit, and let $Q'_4:=Q_4[q_n,t]$.
Let $T':=\{s,t\}$ and $\vec{H'}\subseteq \vec{H}\setminus I_d/I_c$ be the union of $L'_1,L'_2,L'_3, Q'_4$. 
It is not hard to see that (1)-(4) and (ii) of the disentangling lemma~\ref{cut-secondary-disentangle} hold.
By proposition~\ref{cut-secondary-signature} and the mate proposition~\ref{mateprop}, $I_d\cup \{\Omega\}$ has no $k$-mate, so (i) holds.
Let $P'$ be an odd $st$-dipath of $\vec{H'}$ for which $V(P')\cap U_n=\{s\}$.
Then $P'\cup I_c$ contains an odd $st$-dipath $P$ of $\vec{H}$.
Choose $w\in V(R)$ (if any) such that $P$ contains an $s'w$-dipath $Q$ in $\vec{H}$ and $V(Q)\cap V(R)=\{w\}$. Then $(\star)$ implies that $(P-I_c)\cup I_d$, and therefore $P'\cup I_d$, contains an odd $st$-dipath of $\vec{H}$, so (iii) holds as well, a contradiction with the disentangling lemma~\ref{cut-secondary-disentangle}
\end{cproof}

\begin{claims} 
Suppose $m=4$. Then $(G,\Sigma,\{s,t\})$ has an $F_7$ minor.
\end{claims}
\begin{cproof}
Take $P$ from claim~5.
By claim~3, $I\subseteq [3]$.
After redefining $\mathcal{L}$, if necessary, we may assume that $\{1,2\}\subseteq I$ and that $P=Q_1[s',v_1]$.
For each $i\in \{1,2\}$, by claim~2, there exists a signature $k$-mate $B_i$ for $Q_i[v_i,t]\cup \{\Omega\}$.
Choose $W\subseteq V(G)-\{s,t\}$ such that $\delta(W)=B_1\triangle B_2$.
By proposition~\ref{matessignsign}, there exists a shortest path $R$ in $G[W]\setminus B_1$ between $Q_1$ and $Q_2$.
By the shore proposition~\ref{shoreprop}, there exists a path $R_q$ in $G[U_n]\setminus (B_1\cup B_2)$ between $s$ and $Q_4$.
By claim~4, there exists $i\in \{2,3\}$ and vertex $v\in V(Q_i)\cap V(Q_4)$ such that $Q_i[s',v]$ is vertex-disjoint from $Q_1\cup Q_2[v_2,t]\cup Q_4$.
It is now easy (and is left as an exercise) to see that $R_q\cup Q_4\cup Q_i[s',v]\cup Q_1\cup Q_2[v_2,t]\cup R$ has an $F_7$ minor.
\end{cproof}

\begin{claims} 
There exist vertex-disjoint paths $X$ and $Y$ in $\vec{H}$ such that
$X$ is an $s'v_1$-path in $\vec{H}\setminus U_n$
and $Y$ connects a vertex of $U_n$ to $t$.
\end{claims}
\begin{cproof}
Suppose otherwise.

Assume first that $s'=v_1$.
Then, for each $j\in [m]$, $s'\in V(Q_j)$ and by claim~1, $Q_j[s',t]$ has no vertex in common with $U_n$.
Hence, for each $j\in [m]$, by (C4) and proposition~\ref{cut-secondary-signature}, $Q_j[s',t]\cup \{\Omega\}$ has a signature $k$-mate $B_j$.
However, $B_1$ is also a signature $k$-mate for $L_1$, and for each $j\in [m]-[1]$, $B_j$ is also a signature $k$-mate for $Q_j\cup \{\Omega\}$. 
This is a contradiction with the mate proposition~\ref{mateprop}.

Thus, $s'\neq v_1$.
Let $\vec{H}^\star$ be the digraph obtained from $\vec{H}$ after shrinking $U_n$ to a single vertex $u^\star$ and removing all loops.
Notice that every odd $st$-dipath in $\vec{H}$ whose intersection with $U_n$ is $\{s\}$, is a $u^\star t$-dipath in $\vec{H}^\star$ that uses $\Omega$, and vice-versa.
Also, note that the acyclicity condition in (C3) implies that $\vec{H}^\star \setminus u^\star$ is acyclic.
By the linkage lemma~\ref{linkage}, $H^\star$ is a spanning subgraph of a $(u^\star,v_1,t,s')$-web with frame $C_0$ and rib $H^\star_0$.
Fix a plane drawing of $H^\star_0$, where the unbounded face is bounded by $C_0$.
After redefining $\mathcal{L}$, if necessary, we may assume the following: \begin{quote} $(\star)$ for every $s'v_3$-dipath $P$ of $\vec{H}^\star \setminus u^\star$, the number of rib vertices that are on the same side of $P$ as $u^*$ is at least as large as that of $Q_1[s',v_1]$. \end{quote} 

For $j\in [m]-[3]$, let $u_j$ be the largest rib vertex on $Q_j$ that also lies on $Q_1[s',v_1]$.
Observe that if $j\in I\cap ([m]-[3])$, then $u_j=v_j$.
For $j\in [m]-[3]$ let $X_j:=Q_j[u_j,t]$, 
for $j\in [3]\cap I$ let $X_j:=Q_j[v_j,t]$, 
and for $j\in [3]-I$ let $X_j:=Q_j[s',t]$.
For each $j\in [m]$, since $X_j\cup \{\Omega\}$ is contained in a $u^\star t$-dipath of $\vec{H}^\star$, proposition~\ref{cut-secondary-signature} implies that every $k$-mate for $X_j\cup \{\Omega\}$ (if any) must be a signature.
However, every $k$-mate for $X_j\cup \{\Omega\}, j\in [m]$ is also a $k$-mate for $Q_j\cup \{\Omega\}$.
Hence, by the mate proposition~\ref{mateprop}, there exists $i\in [m]$ such that $X_i\cup \{\Omega\}$ has no $k$-mate.
By (C4) and claim~2, $i\notin I\cup [3]$.
Observe that $(\star)$ implies the following: \begin{quote} $(\star\star)$ if $w\in V(Q_1[u_i,t])$ and $P$ is an $s'w$-dipath in $\vec{H}^\star\setminus u^\star$, then $P$ and $X_i$ have a vertex in common.
\end{quote}

Let $I_d:=X_i$ and $I_c:=Q_1[u_i,t]$.
Let $T':=\{s,t\}$,
for $j\in [3]$ let $L'_j$ be $Q_j-(I_c\cup I_d)$ minus any directed circuit it contains, 
and for $j\in [m]-[3]$ let $Q'_j$ be $Q_j-(I_c\cup I_d)$ minus any directed circuit it contains.
Let $\vec{H'}\subseteq \vec{H}\setminus I_d/I_c$ be the union of $D,Q',L'_2,L'_3, Q'_4,\ldots,Q'_m$. 
It is clear that (1)-(4) and (ii) of the disentangling lemma~\ref{cut-secondary-disentangle} hold.
By the choice of $X_i$, (i) holds as well.
To show (iii) holds, let $P'$ be an odd $st$-dipath of $\vec{H'}$ with $V(P')\cap U_n=\{s\}$.
Then $P'\cup I_c$ contains an odd $st$-dipath of $\vec{H}$, 
so $P'\cup I_c$ contains a $u^\star t$-dipath of $\vec{H}^\star$ containing $\Omega$
and by $(\star\star)$, $P'\cup I_d$ also contains a $u^\star t$-dipath of $\vec{H}^\star$ containing $\Omega$, 
implying in turn that $P'\cup I_d$ contains an odd $st$-dipath of $\vec{H}$.
Hence, (iii) holds, a contradiction with the disentangling lemma~\ref{cut-secondary-disentangle}.
\end{cproof}

\begin{claims} 
Suppose $m\geq 5$. Then there exists $i\in [3]$ and $j\in [m]-\{1,2,3,n+3\}$ such that $Q_i$ and $Q_j$ are not internally vertex-disjoint.
\end{claims}
\begin{cproof}
Suppose otherwise. Choose $j\in [m]-\{1,2,3,n+3\}$.
Observe that $R_j\cup Q_j$ is internally vertex-disjoint from each of $Q_1,Q_2,Q_3$, and that by (C4) and propositions~\ref{cut-secondary-signature} and~\ref{bpsignature}, every odd $st$-dipath contained in $Q_1\cup Q_2\cup Q_3$ has a signature $k$-mate disjoint from $R_j\cup Q_j$.
With this in mind, let $I_c:=R_j\cup Q_j$ and $I_d:=\emptyset$.
Let $T':=\emptyset$ and let $\vec{H'}\subseteq \vec{H}\setminus I_d/I_c$ be the union of $L_1,L_2,L_3$. It can be readily checked that (1)-(4) and (i)-(iii) of the disentangling lemma~\ref{cut-secondary-disentangle} hold, a contradiction.
\end{cproof}

For each $i\in I$, let $B_i$ be an extremal $k$-mate of $Q_i[v_i,t]\cup \{\Omega\}$. 
Note that $B_i\cap Q_i[v_i,t]\neq\emptyset$.
As $v_i\geq v_1$, $Q_i[v_i,t]\cup \{\Omega\}$ is contained in an odd $st$-dipath $P$ such that $V(P)\cap U_n=\{s\}$. 
Note that $B_i$ is also a $k$-mate for $P$, so by proposition~\ref{cut-secondary-signature}, $B_i$ is a signature.
Fix $z\in I-\{1\}$.
Choose $W\subseteq V(G)-\{s,t\}$ such that $\delta(W)=B_1\triangle B_z$.
By proposition~\ref{matessignsign}, there is a path in $G[W]\setminus B_1$ between $Q_1$ and $Q_z$.
Moreover, by proposition~\ref{matescutsign}, there is a path between $s$ and $q_n$ in $G[U_n]\setminus B_1$.
We say that {\it property (S)} holds if there exist paths $S_n,S$ in $G$ such that \begin{enumerate}[\;\;]
\item $S_n$ is an $sq_n$-path contained in $G[U_n]\setminus B_1$,
\item $S$ connects a vertex of $Q_1$ to a vertex of $Q_z$ in $G[W]\setminus B_1$, and
\item $S_n$ and $S$ are vertex-disjoint.
\end{enumerate}

\begin{claims} 
Suppose $m\geq 5$ and property (S) holds. Then $(G,\Sigma,\{s,t\})$ has an $F_7$ minor.
\end{claims}
\begin{cproof}
Take $X$ and $Y$ from claim~7.
Notice that each edge in $Y\cap \delta(U_n)$ belongs to either of $Q_4,\ldots,Q_m$,
so we may assume that, for some $u\in \{s,q_1,\ldots,q_n\}$, $Y$ is a $ut$-path.
By claim~8, there is an odd circuit $C$ in $(\vec{H}\cup R_1\cup \cdots\cup R_m)\setminus R_{n+3}$ that shares no vertex with $Q_1[v_1,t]\cup Q_z[v_z,t]$ in $V(G)-\{v_1\}$.
It is now easy (and is left as an exercise) to see that $(C\cup S_n\cup X\cup Y\cup Q_1[v_1,t]\cup Q_z[v_z,t]\cup S\cup R_1\cup \ldots\cup R_m)- R_{n+3}$ has an $F_7$ minor.
\end{cproof}

\begin{claims} 
Suppose $m\geq 5$ and property (S) does not hold. Then there exist vertex-disjoint paths $X$ and $Y$ in $(H\cup R_1\cup \cdots\cup R_m)\setminus R_{n+3}$ where
$X$ is an $s'v_1$-path and $Y$ is an $st$-path.
\end{claims}
\begin{cproof}
Suppose otherwise.
Since property (S) does not hold, the (contrapositive equivalent of the) shore proposition~\ref{shoreprop} implies that $s'\neq v_1$.
Hence, by the linkage lemma~\ref{linkage}, $(H\cup R_1\cup \cdots\cup R_m)\setminus R_{n+3}$ is a spanning subgraph of an $(s,v_1,t,s')$-web with frame $C_0$ and rib $H_0$.
Fix a plane drawing of $H_0$, where the unbounded face is bounded by $C_0$.
After redefining $\mathcal{L}$, if necessary, we may assume the following: \begin{quote} $(\star)$ for every $s'v_1$-dipath $P$ of $\vec{H}$ with $V(P)\cap U_n=\emptyset$, the number of rib vertices that are on the same side of $P$ as $s$ is at least as large as that of $Q_1[s',v_1]$. \end{quote} 

For $j\in [m]-\{1,2,3,n+3\}$, let $u_j$ be the largest rib vertex on $Q_j$ that also lies on $Q_1[s',v_1]$; such $u_j$ exists as $R_j\cup Q_j$ intersects $Q_1[s',v_1]$, but $R_j$ cannot have any vertex in common with $Q_1[s',v_1]$.
Observe that if $j\in I\cap ([m]-[3])$, then $u_j=v_j$.
For $j\in [m]-\{1,2,3,n+3\}$ let $X_j:=Q_j[u_j,t]$, 
for $j\in [3]\cap I$ let $X_j:=Q_j[v_j,t]$, 
and for $j\in [3]-I$ let $X_j:=Q_j[s',t]$.
Observe that each $X_j, j\in [m]-\{n+3\}$ is contained in an odd $st$-dipath whose intersection with $U_n$ is $\{s\}$.
As a result, by proposition~\ref{cut-secondary-signature}, every $k$-mate for $X_j\cup \{\Omega\}, j\in [m]-\{n+3\}$ (if any) must be a signature.
However, every $k$-mate for $X_j\cup \{\Omega\}, j\in [m]-\{n+3\}$ is also a $k$-mate for $P_j\cup \{\Omega\}$.
Hence, since property (S) does not hold, the (contrapositive equivalent of the) shore proposition~\ref{shoreprop} implies that, for some $i\in [m]-\{n+3\}$, $X_i\cup \{\Omega\}$ has no $k$-mate.
By (C4) and claim~2, $i\notin I\cup [3]$.
Observe that $(\star)$ implies the following: \begin{quote} $(\star\star)$ if $w\in V(Q_1[u_i,t])$ and $P$ is an $s'w$-dipath in $\vec{H}\setminus U_n$, then $P$ and $X_i$ have a vertex in common.
\end{quote} 

Let $I_d:=X_i$ and $I_c:=Q_1[u_i,t]$.
Let $T':=\{s,t\}$,
for $j\in [3]$ let $L'_j$ be $Q_j-(I_c\cup I_d)$ minus any directed circuit it contains, 
and for $j\in [m]-[3]$ let $Q'_j$ be $Q_j-(I_c\cup I_d)$ minus any directed circuit it contains.
Let $\vec{H'}\subseteq \vec{H}\setminus I_d/I_c$ be the union of $D,Q',L'_2,L'_3, Q'_4,\ldots,Q'_m$. 
It is clear that (1)-(4) and (ii) of the disentangling lemma~\ref{cut-secondary-disentangle} hold.
By the choice of $X_i$, (i) holds as well.
To show (iii) holds, let $P'$ be an odd $st$-dipath of $\vec{H'}$ with $V(P')\cap U_n=\{s\}$.
Then $P'\cup I_c$ contains an odd $st$-dipath of $\vec{H}$ whose intersection with $U_n$ is $\{s\}$, 
so by $(\star\star)$, $P'\cup I_d$ also contains an $st$-dipath of $\vec{H}$.
Hence, (iii) holds, a contradiction with the disentangling lemma~\ref{cut-secondary-disentangle}.
\end{cproof}

\begin{claims} 
Suppose $m\geq 5$ and property (S) does not hold. Then $(G,\Sigma,\{s,t\})$ has an $F_7$ minor.
\end{claims}
\begin{cproof}
Take $X$ and $Y$ from claim~10.
By proposition~\ref{matessignsign}, there is a path $S$ in $G[W]\setminus B_1$ between $Q_1$ and $Q_z$.
By claim~8, there is an odd circuit $C$ in $(\vec{H}\cup R_1\cup \cdots\cup R_m)\setminus R_{n+3}$ that shares no vertex with $Q_1[v_1,t]\cup Q_z[v_z,t]$ in $V(G)-\{v_1\}$.
It is now easy (and is left as an exercise) to see that $C\cup X\cup Y\cup Q_1[v_1,t]\cup Q_z[v_z,t]\cup S$ has an $F_7$ minor.
\end{cproof}

Observe that claims~6,~9 and~11 finish the proof of proposition~\ref{prp-cut-secondary}.


\section*{Acknowledgements}

The first author was supported partially by an NSERC grant while the second author was supported by a Discovery Grant from NSERC and ONR grant N00014-12-1-0049. We are grateful to three anonymous referees who carefully read the paper and helped us improve the presentation of the paper.

\bibliographystyle{plain}

\begin{thebibliography}{99}

\bibitem{Cohen97}
  Cohen, J. and Lucchesi, C.:
  Minimax relations for $T$-join packing problems.
  Proceedings of the Fifth Israeli Symposium on Theory of Computing and Systems (ISTCS '97), 38--44
  (1997)
\bibitem{Edmonds70}
  Edmonds, J. and Fulkerson, D.R.:
  Bottleneck Extrema.
  J. Combin. Theory Ser. B {\bf 8}, 299--306
  (1970)
\bibitem{Geelen02}
  Geelen, J.F. and Guenin, B.:
  Packing odd circuits in Eulerian graphs.
  J. Combin. Theory Ser. B {\bf 86}, 280--295
  (2002)
\bibitem{Gerards93}
  Gerards, A.M.H.:
  Multicommodity flows and polyhedra.
  CWI Quart. {\bf 6}, 281--296
  (1993)   
\bibitem{Guenin01}
  Guenin, B.:
  A characterization of weakly bipartite graphs.
  J. Combin. Theory Ser. B {\bf 83}, 112--168
  (2001)
\bibitem{Guenin02}
  Guenin, B.:
  Integral polyhedra related to even-cycle and even-cut matroids.
  Math. Oper. Res. {\bf 27}(4), 693--710
  (2002)
\bibitem{Hu63}
  Hu, T.C.:
  Multicommodity network flows.
  Oper. Res. {\bf 11}, 344--360
  (1963)
\bibitem{Lehman64}
  Lehman, A.:
  A solution of the Shannon switching game.
  J. Soc. Indust. Appl. Math. {\bf 12}(4), 687--725
  (1964)
\bibitem{Lehman79}
  Lehman, A.:
  On the width-length inequality.
  Math. Program. {\bf 17}(1), 403--417
  (1979)
\bibitem{Rothschild66}
  Rothschild, B. and Whinston, A.:
  Feasibility of two-commodity network flows.
  Oper. Res. {\bf 14}, 1121--1129
  (1966)
\bibitem{Schrijver02}
  Schrijver, A.:
  A short proof of Guenin's characterization of weakly bipartite graphs.
  J. Combin. Theory Ser. B {\bf 85}, 255--260
  (2002)
\bibitem{Schrijver03}
  Schrijver, A.:
  Combinatorial optimization. Polyhedra and efficiency.
  Springer, 1408--1409
  (2003)
\bibitem{Seymour80}
  Seymour, P.D.:
  Disjoint paths in graphs.
  Discrete Math. {\bf 29}, 293--309
  (1980)
\bibitem{Seymour81}
  Seymour, P.D.:
  Matroids and multicommodity flows.
  Europ. J. Combinatorics {\bf 2}, 257--290
  (1981)
\bibitem{Seymour76}
  Seymour, P.D.:
  The forbidden minors of binary matrices.
  J. London Math. Society {\bf 2}(12), 356--360
  (1976)
\bibitem{Seymour77}
  Seymour, P.D.:
  The matroids with the max-flow min-cut property.
  J. Combin. Theory Ser. B {\bf 23}, 189--222
  (1977)
\bibitem{Thomassen80}
  Thomassen, C.:
  2-Linked graphs.
  European J. Combin. {\bf 1}, 371--378
  (1980)
  
\bibitem{Zaslavsky82}
  Zaslavsky, T.:
  Signed graphs.
  Discrete Appl. Math. {\bf 4}, 47--74
  (1982)

\end{thebibliography}

\end{document}